\documentclass[12pt]{amsart}
\usepackage{graphicx}
\usepackage{amsmath, amscd, amssymb, amsthm}
\usepackage[subnum]{cases}
\usepackage{changepage}
\usepackage{xcolor}
\usepackage{marginnote}
\usepackage{hyperref}
\usepackage{cleveref}
\usepackage[all]{xy}
\usepackage{tabularx}
\usepackage{ltablex}
\usepackage{longtable}
\usepackage[T1]{fontenc}

\newtheorem{theorem}{Theorem}[section]

\newtheorem{proposition}[theorem]{Proposition}
\newtheorem{corollary}[theorem]{Corollary}

\theoremstyle{definition}
\newtheorem{definition}[theorem]{Definition}
\newtheorem{remark}[theorem]{Remark}
\numberwithin{equation}{section}

\newtheorem{assumption}[theorem]{Assumption}
\newtheorem{setting}[theorem]{Setting}

\usepackage{fancyhdr}
\begin{document}

\normalfont

\title{Analytic Geometry and Hodge-Frobenius Structure Continued}
\author{Xin Tong}

\maketitle

\begin{abstract}
\rm This is our sequel to our previous work on the corresponding generalized Frobenius modules over some big multivariate Robba rings. We will go beyond our previous discussion where we focused on the corresponding analytic functions on polydiscs and polyannuli in the strictly affinoid situation, and general Hodge-Frobenius structures which are admissible in the corresponding context in our previous work. 
\end{abstract}

\newpage

\tableofcontents

\newpage

\section{Introduction}

\subsection{Introduction and Summary}

\noindent Generally speaking analytic geometry studies very general manifolds or varieties which are locally related to analytic functions of several variables. In complex analytic geometry, the special domains (like those classical domains) form a very important subject in the corresponding understanding of higher dimensional analytic geometry and their application to other subjects such as number theory. One definitely would like to understand the corresponding story in nonarchimedean geometry, namely what are the corresponding significant special domains. Another issue is that if we could have the chance to find the links with other geometries, such as formal ones and pure algebraic geometry. In the study of analytic geometry, certainly not only the analytic methods will be definitely applied, but also we have algebraic ones. Algebraic ones certainly work well in some very significant cases, for instance those in the Stein cases in classical complex geometry, quasi-Stein cases within the generality of Kedlaya-Liu \cite{KL2} and Fr\'echet-Stein cases after Schneider-Teitelbaum \cite{ST}.\\

\indent While the corresponding nonarchimedean analytic geometry has its own very significant geometric rights, many useful applications also exist extensively as well. For instance in \cite{T1} we followed \cite{CKZ} and \cite{PZ} to have investigated some relative cohomologies corresponding to very specific multi Frobenius structures. The picture in \cite{T1} was a very practical first step towards more general pictures and more interesting pictures.\\

\indent The tools in \cite{T1} extend to our current paper since in more general framework of geometry as in \cite{KL1} and \cite{KL2} the methods and ideas could be applied directly to our situation with possible suitable more detailization. For instance in our situation we could consider nonnoetherian base rings. Note that this does not mean they are necessarily preperfectoid or perfect, since in our situation we can consider some very interesting mixed-type Robba rings in very general sense. For instance if we have two variables, we could have one living in some preperfectoid part while the other one living in the some rigid analytic affinoid component. Note that we do not have to maintain in the corresponding strictly analytic situation, especially if one would like to work with Kedlaya's reified adic spaces \cite{Ked1}.\\

\indent We will also consider the corresponding generalized $B$-pairs after Berger and Nakamura. However for some deep study we could introduce some mixed-type Hodge structures. Namely we could introduce $(\varphi,\Gamma)$-$B$-objects. For instance in the situation where we have two factors we can defined the corresponding period rings:
\begin{displaymath}
\mathrm{B}_{\mathrm{dR},\mathbb{Q}_p}\widehat{\otimes}_{\mathbb{Q}_p}\Pi_{\mathrm{an},\mathrm{con},A},\mathrm{B}^+_{\mathrm{dR},\mathbb{Q}_p}\widehat{\otimes}_{\mathbb{Q}_p}\Pi_{\mathrm{an},\mathrm{con},A}	
\end{displaymath}
with finite projective objects defined over them, carrying one partial action from the Galois groups and carrying one partial action from the corresponding $(\varphi,\Gamma)$ operators.\\

\indent To make summary, we have:\\

\noindent A. Over really general Banach rings, we have defined many useful mixed type Robba rings by perfection along some partial variables. In fact we also have defined some $\infty$-period rings following the recent work from Bambozzi-Kremnizer \cite{BK}. The corresponding issue we encountered is certainly some common issue around sheafiness in adic geometry in some very general classical sense. These directly will be expected to give the chance for one to compare Hodge structures over multivariate imperfect Robba rings with the corresponding Hodge structures over multivariate perfect Robba rings as those with accents $\breve{.}$ and $\widetilde{.}$ in \cite{KL2}, see sections from 2-3. \\

\noindent B. Over really general Banach rings, we defined many useful  $(\varphi_I,\Gamma_I)$-modules over the mixed type Robba rings by perfection along some partial variables. This directly gives the chance for one to compare multivariate $(\varphi_I,\Gamma_I)$-modules over partially imperfect Robba rings with the corresponding multivariate $(\varphi_I,\Gamma_I)$-modules over partially perfect Robba rings as those with accents $\breve{.}$ and $\widetilde{.}$ in \cite{KL2}, see section 4.\\

\noindent C. Over really general Banach rings over $\mathbb{Q}_p$, we defined many useful mixed type big period rings by taking product with Fontaine's de Rham period ring along some partial variables. First of all, we then immediately have the definition for some $B_I'$-$(\varphi_I,\Gamma_I)$-modules. Therefore relying on these rings (although they are very interesting in their rights) proved the equivalence between multivariate $B_I'$-pairs and multivariate $(\varphi_I,\Gamma_I)$-modules generalizing Berger's work and \cite{KP}, see section 5.\\

\noindent D. In sections 6-8, we defined the corresponding $B_I'$-$(\varphi_I,\Gamma_I)$-cohomologies for $B_I'$-$(\varphi_I,\Gamma_I)$-modules, and we promote the equivalence to some quasi-isomorphisms within the derived category $\mathrm{D}^\flat(A)$, where $A$ is the base Banach relative ring. \\

\subsection{Comments on the Notation}

Our notations on the corresponding multivariate mixed type big Hodge structures are inspired by essentially some Langlands programs in $\ell$-adic situation such as in \cite{VL}, namely the Drinfeld's Lemma. However the work \cite{PZ} ad \cite{CKZ} use $\Delta$ which is inspired by essentially some Langlands programs in $p$-adic situation such as that rooted in the work of Z\'abr\'adi and is related to reductive datum. We remind the readers that these are actually not the corresponding intervals for the Robba rings in order to eliminate the corresponding possible confusion. \\

\subsection{Future Study}

\indent The current geometric discussion covered in this paper is basically around the commutative analytic geometry. We have not made it to add the corresponding discussion on the noncommutative setting, but this will be pushed to our further study. Certainly the corresponding noncommutative deformation will require some further well-established understanding on the foundational issues, such as the corresponding noncommutative descent as in the commutative situation from \cite{KL1} and \cite{KL2}. Definitely any good understanding on the noncommutative settings of these sorts will be essential to the corresponding good understanding on noncommutative analytic geometry and noncommutative Tamagawa number conjectures after \cite{FK1}, \cite{BF1} and \cite{BF2}.  \\

\indent Since our dreams will be really those where we can handle very general analytic spaces. Certainly adic spaces need extensively restrictive requirement on the sheafiness of the Banach rings. However this might be resolved completely by considering \cite{BK} (or possibly equivalently the work of Clausen-Scholze \cite{CS}). We have already defined many interesting $\infty$-analytic stacks and the $\infty$-Robba rings. We will study Kedlaya-Liu glueing on this level after \cite{KL1} and \cite{KL2} in future work.\\

\indent We have discussed the corresponding higher dimensional $B$-pairs by introducing the corresponding higher dimensional de Rham period rings. We certainly hope to amplify the corresponding discussion in more $p$-adic Hodge theoretic sense. However one could definitely maintain in the world of $(\varphi_I,\Gamma_I)$-modules in some very flexible way, literally after Berger \cite{Ber2}. We believe that we have more rigidity in the current context. In fact, we should mention that this higher dimensionalization of \cite{Ber2} is literally motivated by the work \cite{KL1} and \cite{KL2}, as well as the work \cite{CKZ} and \cite{PZ}.

\newpage

\section{Big Robba Rings over Rigid Analytic Affinoids and Fr\'echet Objects in Mixed-characteristic Case}

\subsection{Big Robba Rings over Rigid Analytic Affinoids}

\noindent We follow \cite{T1} to give the thorough definition and discussion of those very big period rings we will need in the further discussions in the following body of the paper, where we consider as in \cite{T1} a finite set $I$ with some subset $J$. Our Robba rings in this paper will be depending on the $I$ and $J$ simultaneously.

\begin{definition}
Let $A$ be any affinoid algebra over $\mathbb{Q}_p$ in rigid analytic geometry. We consider the corresponding multi intervals $[\omega^{r_I},\omega^{s_I}]$. Recall we have the corresponding Robba rings defined in \cite[Definition 2.4]{T1}:
\begin{displaymath}
\Pi_{[s_I,r_I],I,A}	
\end{displaymath}
which is defined to be the corresponding affinoid:
\begin{displaymath}
A\widehat{\otimes}_{\mathbb{Q}_p}\mathbb{Q}_p\{\omega^{r_1}/T_1,...,\omega^{r_I}/T_I,T_1/\omega^{s_1},...,T_I/\omega^{s_I}\}.	
\end{displaymath}
Then we have the corresponding rings:
\begin{displaymath}
\Pi_{\mathrm{an},r_I,I,A}:= \varprojlim_{s_I} \Pi_{[s_I,r_I],I,A}.			
\end{displaymath}
with 
\begin{displaymath}
\Pi_{\mathrm{an,con},I,A}:= \bigcup_{r_I}\varprojlim_{s_I} \Pi_{[s_I,r_I],I,A}.			
\end{displaymath}
However we will in this paper to consider some more complicated version of the rings. We will use some partial Frobenius to perfectize partially the rings defined above. Therefore we will in some more uniform way to denote the rings in the following different way:
\begin{align}
\Pi_{[s_I,r_I],I,I,\emptyset,A}:=\Pi_{[s_I,r_I],I,A}\\
\Pi_{\mathrm{an},r_I,I,I,\emptyset,A}:=\Pi_{\mathrm{an},r_I,I,A}\\
\Pi_{\mathrm{an,con},I,I,\emptyset,A}:=\Pi_{\mathrm{an,con},I,\emptyset,A}.	
\end{align}
	
\end{definition}

\indent Now we follow the idea in \cite[Definition 5.2.1]{Ked2} to define some extended version of the rings. We will have the following rings to be:

\begin{align}
\Pi_{[s_I,r_I],I,\breve{J},I\backslash J,A},\\	
\Pi_{[s_I,r_I],I,\widetilde{J},I\backslash J,A},\\
\Pi_{[s_I,r_I],I,J,\breve{I\backslash J},A},\\	
\Pi_{[s_I,r_I],I,\breve{J},\breve{I\backslash J},A},\\	
\Pi_{[s_I,r_I],I,\widetilde{J},\breve{I\backslash J},A},\\
\Pi_{[s_I,r_I],I,J,\widetilde{I\backslash J},A},\\	
\Pi_{[s_I,r_I],I,\breve{J},\widetilde{I\backslash J},A},\\	
\Pi_{[s_I,r_I],I,\widetilde{J},\widetilde{I\backslash J},A}.	
\end{align}

and 

\begin{align}
\Pi_{\mathrm{an},r_I,I,\breve{J},I\backslash J,A},\\	
\Pi_{\mathrm{an},r_I,I,\widetilde{J},I\backslash J,A},\\
\Pi_{\mathrm{an},r_I,I,J,\breve{I\backslash J},A},\\	
\Pi_{\mathrm{an},r_I,I,\breve{J},\breve{I\backslash J},A},\\	
\Pi_{\mathrm{an},r_I,I,\widetilde{J},\breve{I\backslash J},A},\\
\Pi_{\mathrm{an},r_I,I,J,\widetilde{I\backslash J},A},\\	
\Pi_{\mathrm{an},r_I,I,\breve{J},\widetilde{I\backslash J},A},\\	
\Pi_{\mathrm{an},r_I,I,\widetilde{J},\widetilde{I\backslash J},A}.	
\end{align}

and 

\begin{align}
\Pi_{\mathrm{an},\mathrm{con},I,\breve{J},I\backslash J,A},\\	
\Pi_{\mathrm{an},\mathrm{con},I,\widetilde{J},I\backslash J,A},\\
\Pi_{\mathrm{an},\mathrm{con},I,J,\breve{I\backslash J},A},\\	
\Pi_{\mathrm{an},\mathrm{con},I,\breve{J},\breve{I\backslash J},A},\\
\Pi_{\mathrm{an},\mathrm{con},I,\widetilde{J},\breve{I\backslash J},A},\\
\Pi_{\mathrm{an},\mathrm{con},I,J,\widetilde{I\backslash J},A},\\	
\Pi_{\mathrm{an},\mathrm{con},I,\breve{J},\widetilde{I\backslash J},A},\\	
\Pi_{\mathrm{an},\mathrm{con},I,\widetilde{J},\widetilde{I\backslash J},A}.	
\end{align}

\begin{definition}\mbox{\bf{(After Kedlaya-Liu, \cite[Definition 5.2.1]{KL2})}}
We first define the corresponding first group of the rings. The corresponding rings in groups as mentioned above are defined by using the corresponding partial Frobenius $\varphi_1,...,\varphi_I$ and the corresponding Fr\'echet completion.	For the ring $\Pi_{[s_I,r_I],I,\breve{J},I\backslash J,A}$, this is defined by:
\begin{align}
\Pi_{[s_I,r_I],I,\breve{J},I\backslash J,A}:=\varinjlim_{n_\alpha\geq 0,\alpha  \in J}\prod_{\alpha\in J}\varphi_\alpha^{n_\alpha}\Pi_{[s_I,r_I],I,{J},I\backslash J,A}.
\end{align}
Note that for the corresponding rings getting involved in the corresponding definition above we consider the corresponding various Fr\'echet norms for each $t_I>0$:
\begin{displaymath}
\|.\|_{\prod_{\alpha\in J}\varphi_\alpha^{n_\alpha}\Pi_{[s_I,r_I],I,{J},I\backslash J,A},t_I}	
\end{displaymath}
Then we define the corresponding ring $\Pi_{[s_I,r_I],I,\widetilde{J},I\backslash J,A}$, this is defined by the following Fr\'echet completion process:
\begin{align}
\Pi_{[s_I,r_I],I,\widetilde{J},I\backslash J,A}:=\left(\varinjlim_{n_\alpha\geq 0,\alpha  \in J}\prod_{\alpha\in J}\varphi_\alpha^{n_\alpha}\Pi_{[s_I,r_I],I,{J},I\backslash J,A}\right)^\wedge_{\|.\|_{\prod_{\alpha\in J}\varphi_\alpha^{n_\alpha}\Pi_{[s_I,r_I],I,{J},I\backslash J,A},t_I},t_I\in [s_I,r_I]}.
\end{align}

\end{definition}


\indent Then in the corresponding symmetric way we have the following definition:

\begin{definition}\mbox{\bf{(After Kedlaya-Liu, \cite[Definition 5.2.1]{KL2})}}
For the ring $\Pi_{[s_I,r_I],I,{J},\breve{I\backslash J},A}$, this is defined by:
\begin{align}
\Pi_{[s_I,r_I],I,{J},\breve{I\backslash J},A}:=\varinjlim_{n_\alpha\geq 0,\alpha  \in I\backslash J}\prod_{\alpha\in I\backslash J}\varphi_\alpha^{n_\alpha}\Pi_{[s_I,r_I],I,{J},I\backslash J,A}.
\end{align}
Note that for the corresponding rings getting involved in the corresponding definition above we consider the corresponding various Fr\'echet norms for each $t_I>0$:
\begin{displaymath}
\|.\|_{\prod_{\alpha\in I\backslash J}\varphi_\alpha^{n_\alpha}\Pi_{[s_I,r_I],I,{J},I\backslash J,A},t_I}	
\end{displaymath}
Then we define the corresponding ring $\Pi_{[s_I,r_I],I,{J},\widetilde{I\backslash J},A}$, this is defined by the following Fr\'echet completion process:
\begin{align}
\Pi_{[s_I,r_I],I,{J},\widetilde{I\backslash J},A}:=\left(\varinjlim_{n_\alpha\geq 0,\alpha  \in I\backslash J}\prod_{\alpha\in I\backslash J}\varphi_\alpha^{n_\alpha}\Pi_{[s_I,r_I],I,\breve{J},\breve{I\backslash J},A}\right)^\wedge_{\|.\|_{\prod_{\alpha\in I}\varphi_\alpha^{n_\alpha}\Pi_{[s_I,r_I],I,{J},I\backslash J,A},t_I},t_I\in [s_I,r_I]}.
\end{align}

\end{definition}

\indent Then we do the following one:

\begin{definition}\mbox{\bf{(After Kedlaya-Liu, \cite[Definition 5.2.1]{KL2})}}
For the ring $\Pi_{[s_I,r_I],I,\breve{J},\breve{I\backslash J},A}$, this is defined by:
\begin{align}
\Pi_{[s_I,r_I],I,\breve{J},\breve{I\backslash J},A}:=\varinjlim_{n_\alpha\geq 0,\alpha  \in I}\prod_{\alpha\in I}\varphi_\alpha^{n_\alpha}\Pi_{[s_I,r_I],I,{J},I\backslash J,A}.
\end{align}
Note that for the corresponding rings getting involved in the corresponding definition above we consider the corresponding various Fr\'echet norms for each $t_I>0$:
\begin{displaymath}
\|.\|_{\prod_{\alpha\in I}\varphi_\alpha^{n_\alpha}\Pi_{[s_I,r_I],I,{J},I\backslash J,A},t_I}	
\end{displaymath}
Then we define the corresponding ring $\Pi_{[s_I,r_I],I,\widetilde{J},\widetilde{I\backslash J},A}$, this is defined by the following Fr\'echet completion process:
\begin{align}
\Pi_{[s_I,r_I],I,\widetilde{J},\widetilde{I\backslash J},A}:=\left(\varinjlim_{n_\alpha\geq 0,\alpha  \in I}\prod_{\alpha\in I}\varphi_\alpha^{n_\alpha}\Pi_{[s_I,r_I],I,{J},I\backslash J,A}\right)^\wedge_{\|.\|_{\prod_{\alpha\in I}\varphi_\alpha^{n_\alpha}\Pi_{[s_I,r_I],I,{J},I\backslash J,A},t_I},t_I\in [s_I,r_I]}.
\end{align}

\end{definition}

\indent Then we consider the following definition building on the definitions above:

\begin{definition}\mbox{\bf{(After Kedlaya-Liu, \cite[Definition 5.2.1]{KL2})}}
For the ring $\Pi_{[s_I,r_I],I,\widetilde{J},\breve{I\backslash J},A}$, this is defined by:
\begin{align}
\Pi_{[s_I,r_I],I,\widetilde{J},\breve{I\backslash J},A}:=\varinjlim_{n_\alpha\geq 0,\alpha  \in I\backslash J}\prod_{\alpha\in I\backslash J}\varphi_\alpha^{n_\alpha}\Pi_{[s_I,r_I],I,\widetilde{J},I\backslash J,A}.
\end{align}
For the ring $\Pi_{[s_I,r_I],I,\breve{J},\widetilde{I\backslash J},A}$, this is defined by:
\begin{align}
\Pi_{[s_I,r_I],I,\breve{J},\widetilde{I\backslash J},A}:=\varinjlim_{n_\alpha\geq 0,\alpha  \in J}\prod_{\alpha\in J}\varphi_\alpha^{n_\alpha}\Pi_{[s_I,r_I],I,{J},\widetilde{I\backslash J},A}.
\end{align}

\end{definition}

\indent Then we have the following definitions:

\begin{definition} \mbox{\bf{(After Kedlaya-Liu, \cite[Definition 5.2.1]{KL2})}}
\begin{align}
\Pi_{\mathrm{an},r_I,I,\breve{J},I\backslash J,A}:=\varprojlim_{s_I}\Pi_{[s_I,r_I],I,\breve{J},I\backslash J,A},\\	
\Pi_{\mathrm{an},r_I,I,\widetilde{J},I\backslash J,A}:=\varprojlim_{s_I} \Pi_{[s_I,r_I],I,\widetilde{J},I\backslash J,A},\\
\Pi_{\mathrm{an},r_I,I,J,\breve{I\backslash J},A}:=\varprojlim_{s_I}\Pi_{[s_I,r_I],I,J,\breve{I\backslash J},A},\\	
\Pi_{\mathrm{an},r_I,I,\breve{J},\breve{I\backslash J},A}:=\varprojlim_{s_I} \Pi_{[s_I,r_I],I,\breve{J},\breve{I\backslash J},A},\\	
\Pi_{\mathrm{an},r_I,I,\widetilde{J},\breve{I\backslash J},A}:=\varprojlim_{s_I} \Pi_{[s_I,r_I],I,\widetilde{J},\breve{I\backslash J},A},\\
\Pi_{\mathrm{an},r_I,I,J,\widetilde{I\backslash J},A}:=\varprojlim_{s_I} \Pi_{[s_I,r_I],I,J,\widetilde{I\backslash J},A},\\	
\Pi_{\mathrm{an},r_I,I,\breve{J},\widetilde{I\backslash J},A}:=\varprojlim_{s_I} \Pi_{[s_I,r_I],I,\breve{J},\widetilde{I\backslash J},A},\\	
\Pi_{\mathrm{an},r_I,I,\widetilde{J},\widetilde{I\backslash J},A}:=\varprojlim_{s_I} \Pi_{[s_I,r_I],I,\widetilde{J},\widetilde{I\backslash J},A}.	
\end{align}

and 

\begin{align}
\Pi_{\mathrm{an},\mathrm{con},I,\breve{J},I\backslash J,A}:=\varinjlim_{r_I}\varprojlim_{s_I}\Pi_{[s_I,r_I],I,\breve{J},I\backslash J,A},\\	
\Pi_{\mathrm{an},\mathrm{con},I,\widetilde{J},I\backslash J,A}:=\varinjlim_{r_I}\varprojlim_{s_I} \Pi_{[s_I,r_I],I,\widetilde{J},I\backslash J,A},\\
\Pi_{\mathrm{an},\mathrm{con},I,J,\breve{I\backslash J},A}:=\varinjlim_{r_I}\varprojlim_{s_I}\Pi_{[s_I,r_I],I,J,\breve{I\backslash J},A},\\	
\Pi_{\mathrm{an},\mathrm{con},I,\breve{J},\breve{I\backslash J},A}:=\varinjlim_{r_I}\varprojlim_{s_I} \Pi_{[s_I,r_I],I,\breve{J},\breve{I\backslash J},A},\\
\Pi_{\mathrm{an},\mathrm{con},I,\widetilde{J},\breve{I\backslash J},A}:=\varinjlim_{r_I}\varprojlim_{s_I} \Pi_{[s_I,r_I],I,\widetilde{J},\breve{I\backslash J},A},\\
\Pi_{\mathrm{an},\mathrm{con},I,J,\widetilde{I\backslash J},A}:=\varinjlim_{r_I}\varprojlim_{s_I} \Pi_{[s_I,r_I],I,J,\widetilde{I\backslash J},A},\\	
\Pi_{\mathrm{an},\mathrm{con},I,\breve{J},\widetilde{I\backslash J},A}:=\varinjlim_{r_I}\varprojlim_{s_I} \Pi_{[s_I,r_I],I,\breve{J},\widetilde{I\backslash J},A},\\	
\Pi_{\mathrm{an},\mathrm{con},I,\widetilde{J},\widetilde{I\backslash J},A}:=\varinjlim_{r_I}\varprojlim_{s_I} \Pi_{[s_I,r_I],I,\widetilde{J},\widetilde{I\backslash J},A}.\\	
\end{align}	
\end{definition}


\begin{setting}
The corresponding construction we established above is also motivated from the corresponding construction of multivariate Robba rings in \cite{Ked2}.\\	
\end{setting}

\subsection{Fr\'echet Objects}

\indent Now we perform some of the corresponding construction parallel to the corresponding Fr\'echet-Stein construction used originally in \cite{KPX}, which is essentially developed in \cite[Section 2.6]{KL2}. In effect, the corresponding construction was already used in \cite{T1}, although the paper behaves as if we are not fully after \cite{KL2}. The reason of such impression left for the readers is essentially due to the fact that we are in the noetherian situation. Now we are not definitely in the noetherian situation, but \cite[Section 2.6]{KL2} has already tackled this issue.

\begin{setting} \mbox{\bf{(After Kedlaya-Liu, \cite[Definition 2.6.2]{KL2})}}
Following \cite[Definition 2.6.2]{KL2} we call a Banach uniform adic ring $(R,R^+)$ quasi-Stein if it could be written as the following inverse limit with the corresponding transition map of dense image for some inverse system $\{\alpha\}$ of indexes:
\begin{displaymath}
(R,R^+):=\varprojlim_{\alpha}  (R_\alpha,R_\alpha^+).	
\end{displaymath}
And we call the ring ind-Fr\'echet-Stein if it is further could be written the following injective-projective limit:
\begin{displaymath}
(R,R^+):=\varinjlim_{\alpha'}\varprojlim_{\alpha}(R_{\alpha,\alpha'} ,R_{\alpha,\alpha'}^+).	
\end{displaymath}
	
\end{setting}

\indent After this axiomization we could now study the corresponding sheaves over these rings. In what follow, we consider the corresponding radii living in the set of all the rational numbers.

\begin{definition} \mbox{\bf{(After KPX, \cite[Definition 2.1.3]{KPX})}}
Over $\Pi_{\mathrm{an},r_{I,0},I,\breve{J},I\backslash J,A}$ we define the corresponding stably pseudocoherent sheaves to mean a collection of stably pseudocoherent modules $(M_{[s_I,r_I]})$	over each $\Pi_{[s_I,r_I],I,\breve{J},I\backslash J,A}$ satisfying the corresponding compatibility condition and the obvious cocycle condition with respect to the family of the corresponding multi-intervals $\{[s_I,r_I]\}$.
\end{definition}

\begin{definition} \mbox{\bf{(After KPX, \cite[Definition 2.1.3]{KPX})}}
Over $\Pi_{\mathrm{an},r_{I,0},I,\widetilde{J},I\backslash J,A}$ we define the corresponding stably pseudocoherent sheaves to mean a collection of stably pseudocoherent modules $(M_{[s_I,r_I]})$	over each $\Pi_{[s_I,r_I],I,\widetilde{J},I\backslash J,A}$ satisfying the corresponding compatibility condition and the obvious cocycle condition with respect to the family of the corresponding multi-intervals $\{[s_I,r_I]\}$.
\end{definition}

\begin{definition} \mbox{\bf{(After KPX, \cite[Definition 2.1.3]{KPX})}}
Over $\Pi_{\mathrm{an},r_{I,0},I,{J},\breve{I\backslash J},A}$ we define the corresponding stably pseudocoherent sheaves to mean a collection of stably pseudocoherent modules $(M_{[s_I,r_I]})$	over each $\Pi_{[s_I,r_I],I,{J},\breve{I\backslash J},A}$ satisfying the corresponding compatibility condition and the obvious cocycle condition with respect to the family of the corresponding multi-intervals $\{[s_I,r_I]\}$.
\end{definition}

\begin{definition} \mbox{\bf{(After KPX, \cite[Definition 2.1.3]{KPX})}}
Over $\Pi_{\mathrm{an},r_{I,0},I,\breve{J},\breve{I\backslash J},A}$ we define the corresponding stably pseudocoherent sheaves to mean a collection of stably pseudocoherent modules $(M_{[s_I,r_I]})$	over each $\Pi_{[s_I,r_I],I,\breve{J},\breve{I\backslash J},A}$ satisfying the corresponding compatibility condition and the obvious cocycle condition with respect to the family of the corresponding multi-intervals $\{[s_I,r_I]\}$.
\end{definition}

\begin{definition} \mbox{\bf{(After KPX, \cite[Definition 2.1.3]{KPX})}}
Over $\Pi_{\mathrm{an},r_{I,0},I,\widetilde{J},\breve{I\backslash J},A}$ we define the corresponding stably pseudocoherent sheaves to mean a collection of stably pseudocoherent modules $(M_{[s_I,r_I]})$ over each $\Pi_{[s_I,r_I],I,\widetilde{J},\breve{I\backslash J},A}$ satisfying the corresponding compatibility condition and the obvious cocycle condition with respect to the family of the corresponding multi-intervals $\{[s_I,r_I]\}$.
\end{definition}

\begin{definition} \mbox{\bf{(After KPX, \cite[Definition 2.1.3]{KPX})}}
Over $\Pi_{\mathrm{an},r_{I,0},I,{J},\widetilde{I\backslash J},A}$ we define the corresponding stably pseudocoherent sheaves to mean a collection of stably pseudocoherent modules $(M_{[s_I,r_I]})$	over each $\Pi_{[s_I,r_I],I,{J},\widetilde{I\backslash J},A}$ satisfying the corresponding compatibility condition and the obvious cocycle condition with respect to the family of the corresponding multi-intervals $\{[s_I,r_I]\}$.
\end{definition}

\begin{definition} \mbox{\bf{(After KPX, \cite[Definition 2.1.3]{KPX})}}
Over $\Pi_{\mathrm{an},r_{I,0},I,\breve{J},\widetilde{I\backslash J},A}$ we define the corresponding stably pseudocoherent sheaves to mean a collection of stably pseudocoherent modules $(M_{[s_I,r_I]})$ over each $\Pi_{[s_I,r_I],I,\breve{J},\widetilde{I\backslash J},A}$ satisfying the corresponding compatibility condition and the obvious cocycle condition with respect to the family of the corresponding multi-intervals $\{[s_I,r_I]\}$.
\end{definition}

\begin{definition} \mbox{\bf{(After KPX, \cite[Definition 2.1.3]{KPX})}}
Over $\Pi_{\mathrm{an},r_{I,0},I,\widetilde{J},\widetilde{I\backslash J},A}$ we define the corresponding stably pseudocoherent sheaves to mean a collection of stably pseudocoherent modules $(M_{[s_I,r_I]})$ over each $\Pi_{[s_I,r_I],I,\widetilde{J},\widetilde{I\backslash J},A}$ satisfying the corresponding compatibility condition and the obvious cocycle condition with respect to the family of the corresponding multi-intervals $\{[s_I,r_I]\}$.
\end{definition}

\begin{remark}
There is some overlap and repeating on some objects within the eight categories. Therefore in the following we are going to then work with only 1st, 2nd, 4th, 5th, 8th categories. 	
\end{remark}

\begin{proposition} \mbox{\bf{(After Kedlaya-Liu, \cite[Corollary 2.6.8]{KL2})}}
The corresponding pseudocoherent finitely projective bundles over
\begin{align}
\Pi_{\mathrm{an},r_{I,0},I,\breve{J},I\backslash J,A},\\
\Pi_{\mathrm{an},r_{I,0},I,\widetilde{J},I\backslash J,A},\\
\Pi_{\mathrm{an},r_{I,0},I,\breve{J},\breve{I\backslash J},A},\\
\Pi_{\mathrm{an},r_{I,0},I,\widetilde{J},\breve{I\backslash J},A},\\
\Pi_{\mathrm{an},r_{I,0},I,\widetilde{J},\widetilde{I\backslash J},A}
\end{align}
defined as above have global sections which are finite projective if and only if the global sections are finitely generated. 	
\end{proposition}

\begin{proof}
Just apply \cite[Corollary 2.6.8]{KL2}.	
\end{proof}

\begin{proposition} \mbox{\bf{(After Kedlaya-Liu, \cite[Proposition 2.6.17]{KL2})}} \label{proposition2.18}
The corresponding pseudocoherent sheaves over
\begin{align}
\Pi_{\mathrm{an},r_{I,0},I,\breve{J},I\backslash J,A},\\
\Pi_{\mathrm{an},r_{I,0},I,\widetilde{J},I\backslash J,A},\\
\Pi_{\mathrm{an},r_{I,0},I,\breve{J},\breve{I\backslash J},A},\\
\Pi_{\mathrm{an},r_{I,0},I,\widetilde{J},\breve{I\backslash J},A},\\
\Pi_{\mathrm{an},r_{I,0},I,\widetilde{J},\widetilde{I\backslash J},A}
\end{align}
defined as above have global sections which are finitely generated as long as we have the uniform bound on the rank of the bundles over quasi-compacts with respect to closed multi-intervals taking the general form of $[s_I,r_I]$. 	
\end{proposition}

\begin{proof}
This is actually a direct consequence of \cite[Proposition 2.6.17]{KL2}, where the space $(0,r_{I,0}]$ admits $2^{|I|}$-uniform covering.	
\end{proof}

\begin{proposition} \mbox{\bf{(After Kedlaya-Liu, \cite[Corollary 2.6.8, Proposition 2.6.17]{KL2})}} \label{proposition2.19}
The corresponding pseudocoherent sheaves over
\begin{align}
\Pi_{\mathrm{an},r_{I,0},I,\breve{J},I\backslash J,A},\\
\Pi_{\mathrm{an},r_{I,0},I,\widetilde{J},I\backslash J,A},\\
\Pi_{\mathrm{an},r_{I,0},I,\breve{J},\breve{I\backslash J},A},\\
\Pi_{\mathrm{an},r_{I,0},I,\widetilde{J},\breve{I\backslash J},A},\\
\Pi_{\mathrm{an},r_{I,0},I,\widetilde{J},\widetilde{I\backslash J},A}
\end{align}
defined as above have global sections which are finite projective as long as we have the uniform bound on the rank of the bundle over each quasi-compact with respect to each closed multi-interval $[s_I,r_I]$, and we have that the sheaves admits section actually finite projective over each quasi-compact with respect to each closed multi-interval $[s_I,r_I]$. 	
\end{proposition}

\begin{proof}
This is actually a direct corollary of the previous two propositions. 	
\end{proof}

\newpage

\section{Big Robba Rings over General Banach Affinoids and Fr\'echet Objects in Mixed-characteristic Case}

\subsection{Big Robba Rings over General Banach Affinoids}

\indent Since we have already considered the corresponding foundation from \cite{KL2} on the quasi-Stein nonnoetherian adic Banach uniform algebra over $\mathbb{Q}_p$, we hope then now study more general $p$-adic analysis of several variables. Certainly as mentioned in \cite{KPX} one could carry some strongly noetherian coefficients, where everything is sheafy, but one might be very curious about the situation where we do not have so strong condition on the noetherianness. Actually we could then apply the derived analytic geometry in \cite{BK}.

\begin{definition}
Let $A$ be any commutative Banach algebra over $\mathbb{Q}_p$. We consider the corresponding multi intervals $[\omega^{r_I},\omega^{s_I}]$. We have the corresponding Robba rings defined as in \cite[Definition 2.4]{T1}:
\begin{displaymath}
\Pi_{[s_I,r_I],I,A}	
\end{displaymath}
which is defined to be the corresponding affinoid:
\begin{displaymath}
A\widehat{\otimes}_{\mathbb{Q}_p}\mathbb{Q}_p\{\omega^{r_1}/T_1,...,\omega^{r_I}/T_I,T_1/\omega^{s_1},...,T_I/\omega^{s_I}\}.	
\end{displaymath}
Then we have the corresponding rings:
\begin{displaymath}
\Pi_{\mathrm{an},r_I,I,A}:= \varprojlim_{s_I} \Pi_{[s_I,r_I],I,A}.			
\end{displaymath}
with 
\begin{displaymath}
\Pi_{\mathrm{an,con},I,A}:= \bigcup_{r_I}\varprojlim_{s_I} \Pi_{[s_I,r_I],I,A}.			
\end{displaymath}
However we will in this paper to consider some more complicated version of the rings. We will use some partial Frobenius to perfectize partially the rings defined above. Therefore we will in some more uniform way to denote the rings in the following different way:
\begin{align}
\Pi_{[s_I,r_I],I,I,\emptyset,A}:=\Pi_{[s_I,r_I],I,A}\\
\Pi_{\mathrm{an},r_I,I,I,\emptyset,A}:=\Pi_{\mathrm{an},r_I,I,A}\\
\Pi_{\mathrm{an,con},I,I,\emptyset,A}:=\Pi_{\mathrm{an,con},I,\emptyset,A}.	
\end{align}
	
\end{definition}

\indent Now we follow the idea in \cite[Definition 5.2.1]{Ked2} to define some extended version of the rings. We will have the following rings to be:

\begin{align}
\Pi_{[s_I,r_I],I,?,?',A},?=J,\widetilde{J},\breve{J},?'=I\backslash J,\widetilde{I\backslash J},\breve{I\backslash J},\\		
\end{align}

and 

\begin{align}
\Pi_{\mathrm{an},r_I,I,?,?',A},?=J,\widetilde{J},\breve{J},?'=I\backslash J,\widetilde{I\backslash J},\breve{I\backslash J},\\ 	
\end{align}

and 

\begin{align}
\Pi_{\mathrm{an},\mathrm{con},I,?,?',A},?=J,\widetilde{J},\breve{J},?'=I\backslash J,\widetilde{I\backslash J},\breve{I\backslash J}.\\	
\end{align}

\begin{definition}\mbox{\bf{(After Kedlaya-Liu, \cite[Definition 5.2.1]{KL2})}}
We first define the corresponding first group of the rings. The corresponding rings in groups as mentioned above are defined by using the corresponding partial Frobenius $\varphi_1,...,\varphi_I$ and the corresponding Fr\'echet completion.	For the ring $\Pi_{[s_I,r_I],I,\breve{J},I\backslash J,A}$, this is defined by:
\begin{align}
\Pi_{[s_I,r_I],I,\breve{J},I\backslash J,A}:=\varinjlim_{n_\alpha\geq 0,\alpha  \in J}\prod_{\alpha\in J}\varphi_\alpha^{n_\alpha}\Pi_{[s_I,r_I],I,{J},I\backslash J,A}.
\end{align}
Note that for the corresponding rings getting involved in the corresponding definition above we consider the corresponding various Fr\'echet norms for each $t_I>0$:
\begin{displaymath}
\|.\|_{\prod_{\alpha\in J}\varphi_\alpha^{n_\alpha}\Pi_{[s_I,r_I],I,{J},I\backslash J,A},t_I}	
\end{displaymath}
Then we define the corresponding ring $\Pi_{[s_I,r_I],I,\widetilde{J},I\backslash J,A}$, this is defined by the following Fr\'echet completion process:
\begin{align}
\Pi_{[s_I,r_I],I,\widetilde{J},I\backslash J,A}:=\left(\varinjlim_{n_\alpha\geq 0,\alpha  \in J}\prod_{\alpha\in J}\varphi_\alpha^{n_\alpha}\Pi_{[s_I,r_I],I,{J},I\backslash J,A}\right)^\wedge_{\|.\|_{\prod_{\alpha\in J}\varphi_\alpha^{n_\alpha}\Pi_{[s_I,r_I],I,{J},I\backslash J,A},t_I},t_I\in [s_I,r_I]}.
\end{align}

\end{definition}

\indent Then in the corresponding symmetric way we have the following definition:

\begin{definition}\mbox{\bf{(After Kedlaya-Liu, \cite[Definition 5.2.1]{KL2})}}
For the ring $\Pi_{[s_I,r_I],I,{J},\breve{I\backslash J},A}$, this is defined by:
\begin{align}
\Pi_{[s_I,r_I],I,{J},\breve{I\backslash J},A}:=\varinjlim_{n_\alpha\geq 0,\alpha  \in I\backslash J}\prod_{\alpha\in I\backslash J}\varphi_\alpha^{n_\alpha}\Pi_{[s_I,r_I],I,{J},I\backslash J,A}.
\end{align}
Note that for the corresponding rings getting involved in the corresponding definition above we consider the corresponding various Fr\'echet norms for each $t_I>0$:
\begin{displaymath}
\|.\|_{\prod_{\alpha\in I\backslash J}\varphi_\alpha^{n_\alpha}\Pi_{[s_I,r_I],I,{J},I\backslash J,A},t_I}	
\end{displaymath}
Then we define the corresponding ring $\Pi_{[s_I,r_I],I,{J},\widetilde{I\backslash J},A}$, this is defined by the following Fr\'echet completion process:
\begin{align}
\Pi_{[s_I,r_I],I,{J},\widetilde{I\backslash J},A}:=\left(\varinjlim_{n_\alpha\geq 0,\alpha  \in I\backslash J}\prod_{\alpha\in I\backslash J}\varphi_\alpha^{n_\alpha}\Pi_{[s_I,r_I],I,\breve{J},\breve{I\backslash J},A}\right)^\wedge_{\|.\|_{\prod_{\alpha\in I}\varphi_\alpha^{n_\alpha}\Pi_{[s_I,r_I],I,{J},I\backslash J,A},t_I},t_I\in [s_I,r_I]}.
\end{align}

\end{definition}

\indent Then we do the following one:

\begin{definition}\mbox{\bf{(After Kedlaya-Liu, \cite[Definition 5.2.1]{KL2})}}
For the ring $\Pi_{[s_I,r_I],I,\breve{J},\breve{I\backslash J},A}$, this is defined by:
\begin{align}
\Pi_{[s_I,r_I],I,\breve{J},\breve{I\backslash J},A}:=\varinjlim_{n_\alpha\geq 0,\alpha  \in I}\prod_{\alpha\in I}\varphi_\alpha^{n_\alpha}\Pi_{[s_I,r_I],I,{J},I\backslash J,A}.
\end{align}
Note that for the corresponding rings getting involved in the corresponding definition above we consider the corresponding various Fr\'echet norms for each $t_I>0$:
\begin{displaymath}
\|.\|_{\prod_{\alpha\in I}\varphi_\alpha^{n_\alpha}\Pi_{[s_I,r_I],I,{J},I\backslash J,A},t_I}	
\end{displaymath}
Then we define the corresponding ring $\Pi_{[s_I,r_I],I,\widetilde{J},\widetilde{I\backslash J},A}$, this is defined by the following Fr\'echet completion process:
\begin{align}
\Pi_{[s_I,r_I],I,\widetilde{J},\widetilde{I\backslash J},A}:=\left(\varinjlim_{n_\alpha\geq 0,\alpha  \in I}\prod_{\alpha\in I}\varphi_\alpha^{n_\alpha}\Pi_{[s_I,r_I],I,{J},\breve{I\backslash J},A}\right)^\wedge_{\|.\|_{\prod_{\alpha\in I}\varphi_\alpha^{n_\alpha}\Pi_{[s_I,r_I],I,{J},I\backslash J,A},t_I},t_I\in [s_I,r_I]}.
\end{align}

\end{definition}

\indent Then we consider the following definition building on the definitions above:

\begin{definition}\mbox{\bf{(After Kedlaya-Liu, \cite[Definition 5.2.1]{KL2})}}
For the ring $\Pi_{[s_I,r_I],I,\widetilde{J},\breve{I\backslash J},A}$, this is defined by:
\begin{align}
\Pi_{[s_I,r_I],I,\widetilde{J},\breve{I\backslash J},A}:=\varinjlim_{n_\alpha\geq 0,\alpha  \in I\backslash J}\prod_{\alpha\in I\backslash J}\varphi_\alpha^{n_\alpha}\Pi_{[s_I,r_I],I,\widetilde{J},I\backslash J,A}.
\end{align}
For the ring $\Pi_{[s_I,r_I],I,\breve{J},\widetilde{I\backslash J},A}$, this is defined by:
\begin{align}
\Pi_{[s_I,r_I],I,\breve{J},\widetilde{I\backslash J},A}:=\varinjlim_{n_\alpha\geq 0,\alpha  \in J}\prod_{\alpha\in J}\varphi_\alpha^{n_\alpha}\Pi_{[s_I,r_I],I,{J},\widetilde{I\backslash J},A}.
\end{align}

\end{definition}

\indent Then we have the following definitions:

\begin{definition} \mbox{\bf{(After Kedlaya-Liu, \cite[Definition 5.2.1]{KL2})}}
\begin{align}
\Pi_{\mathrm{an},r_I,I,?,?',A}:=\varprojlim_{s_I}\Pi_{[s_I,r_I],I,?,?',A},?=J,\widetilde{J},\breve{J},?'=I\backslash J,\widetilde{I\backslash J},\breve{I\backslash J},\\	
\end{align}

and 

\begin{align}
\Pi_{\mathrm{an},\mathrm{con},I,?,?',A}:=\varinjlim_{r_I}\varprojlim_{s_I}\Pi_{[s_I,r_I],I,?,?',A},?=J,\widetilde{J},\breve{J},?'=I\backslash J,\widetilde{I\backslash J},\breve{I\backslash J}.\\	
\end{align}	
\end{definition}

\subsection{$\infty$-Robba Rings over General Banach Affinoids}

\indent We now apply the construction of \cite{BK} to the rings defined in the previous section. Recall from \cite{BK}, for any Banach adic algebra $R$ over $\mathbb{Q}_p$ we have the derived spectrum $\mathrm{Spa}^h(R):=\mathrm{Spa}^h_{\mathrm{Rat}}(R)$.

\begin{definition}
Consider the following rings we defined in the previous section:
\begin{align}
\Pi_{[s_I,r_I],I,\breve{J},I\backslash J,A},\\	
\Pi_{[s_I,r_I],I,\widetilde{J},I\backslash J,A},\\
\Pi_{[s_I,r_I],I,J,\breve{I\backslash J},A},\\	
\Pi_{[s_I,r_I],I,\breve{J},\breve{I\backslash J},A},\\	
\Pi_{[s_I,r_I],I,\widetilde{J},\breve{I\backslash J},A},\\
\Pi_{[s_I,r_I],I,J,\widetilde{I\backslash J},A},\\	
\Pi_{[s_I,r_I],I,\breve{J},\widetilde{I\backslash J},A},\\	
\Pi_{[s_I,r_I],I,\widetilde{J},\widetilde{I\backslash J},A}.	
\end{align}
We then take the corresponding derived spectrum from Bambozzi-Kremnizer to defined the following $\infty$-analytic stacks:
\begin{align}
\mathrm{Spa}^h\Pi_{[s_I,r_I],I,\breve{J},I\backslash J,A},\\	
\mathrm{Spa}^h\Pi_{[s_I,r_I],I,\widetilde{J},I\backslash J,A},\\
\mathrm{Spa}^h\Pi_{[s_I,r_I],I,J,\breve{I\backslash J},A},\\	
\mathrm{Spa}^h\Pi_{[s_I,r_I],I,\breve{J},\breve{I\backslash J},A},\\	
\mathrm{Spa}^h\Pi_{[s_I,r_I],I,\widetilde{J},\breve{I\backslash J},A},\\
\mathrm{Spa}^h\Pi_{[s_I,r_I],I,J,\widetilde{I\backslash J},A},\\	
\mathrm{Spa}^h\Pi_{[s_I,r_I],I,\breve{J},\widetilde{I\backslash J},A},\\	
\mathrm{Spa}^h\Pi_{[s_I,r_I],I,\widetilde{J},\widetilde{I\backslash J},A}.	
\end{align}
Taking the global section we have the following ring spectra:
\begin{align}
\Pi^h_{[s_I,r_I],I,\breve{J},I\backslash J,A},\\	
\Pi^h_{[s_I,r_I],I,\widetilde{J},I\backslash J,A},\\
\Pi^h_{[s_I,r_I],I,J,\breve{I\backslash J},A},\\	
\Pi^h_{[s_I,r_I],I,\breve{J},\breve{I\backslash J},A},\\	
\Pi^h_{[s_I,r_I],I,\widetilde{J},\breve{I\backslash J},A},\\
\Pi^h_{[s_I,r_I],I,J,\widetilde{I\backslash J},A},\\	
\Pi^h_{[s_I,r_I],I,\breve{J},\widetilde{I\backslash J},A},\\	
\Pi^h_{[s_I,r_I],I,\widetilde{J},\widetilde{I\backslash J},A}.
\end{align}
\end{definition}

\indent Then we have the following definitions:

\begin{definition} \mbox{\bf{(After Kedlaya-Liu, \cite[Definition 5.2.1]{KL2})}}
\begin{align}
\Pi^h_{\mathrm{an},r_I,I,\breve{J},I\backslash J,A}:=\varprojlim_{s_I}\Pi^h_{[s_I,r_I],I,\breve{J},I\backslash J,A},\\	
\Pi^h_{\mathrm{an},r_I,I,\widetilde{J},I\backslash J,A}:=\varprojlim_{s_I} \Pi^h_{[s_I,r_I],I,\widetilde{J},I\backslash J,A},\\
\Pi^h_{\mathrm{an},r_I,I,J,\breve{I\backslash J},A}:=\varprojlim_{s_I}\Pi^h_{[s_I,r_I],I,J,\breve{I\backslash J},A},\\	
\Pi^h_{\mathrm{an},r_I,I,\breve{J},\breve{I\backslash J},A}:=\varprojlim_{s_I} \Pi^h_{[s_I,r_I],I,\breve{J},\breve{I\backslash J},A},\\	
\Pi^h_{\mathrm{an},r_I,I,\widetilde{J},\breve{I\backslash J},A}:=\varprojlim_{s_I} \Pi^h_{[s_I,r_I],I,\widetilde{J},\breve{I\backslash J},A},\\
\Pi^h_{\mathrm{an},r_I,I,J,\widetilde{I\backslash J},A}:=\varprojlim_{s_I} \Pi^h_{[s_I,r_I],I,J,\widetilde{I\backslash J},A},\\	
\Pi^h_{\mathrm{an},r_I,I,\breve{J},\widetilde{I\backslash J},A}:=\varprojlim_{s_I} \Pi^h_{[s_I,r_I],I,\breve{J},\widetilde{I\backslash J},A},\\	
\Pi^h_{\mathrm{an},r_I,I,\widetilde{J},\widetilde{I\backslash J},A}:=\varprojlim_{s_I} \Pi^h_{[s_I,r_I],I,\widetilde{J},\widetilde{I\backslash J},A}.	
\end{align}

and 

\begin{align}
\Pi^h_{\mathrm{an},\mathrm{con},I,\breve{J},I\backslash J,A}:=\varinjlim_{r_I}\varprojlim_{s_I}\Pi^h_{[s_I,r_I],I,\breve{J},I\backslash J,A},\\	
\Pi^h_{\mathrm{an},\mathrm{con},I,\widetilde{J},I\backslash J,A}:=\varinjlim_{r_I}\varprojlim_{s_I} \Pi^h_{[s_I,r_I],I,\widetilde{J},I\backslash J,A},\\
\Pi^h_{\mathrm{an},\mathrm{con},I,J,\breve{I\backslash J},A}:=\varinjlim_{r_I}\varprojlim_{s_I}\Pi^h_{[s_I,r_I],I,J,\breve{I\backslash J},A},\\	
\Pi^h_{\mathrm{an},\mathrm{con},I,\breve{J},\breve{I\backslash J},A}:=\varinjlim_{r_I}\varprojlim_{s_I} \Pi^h_{[s_I,r_I],I,\breve{J},\breve{I\backslash J},A},\\
\Pi^h_{\mathrm{an},\mathrm{con},I,\widetilde{J},\breve{I\backslash J},A}:=\varinjlim_{r_I}\varprojlim_{s_I} \Pi^h_{[s_I,r_I],I,\widetilde{J},\breve{I\backslash J},A},\\
\Pi^h_{\mathrm{an},\mathrm{con},I,J,\widetilde{I\backslash J},A}:=\varinjlim_{r_I}\varprojlim_{s_I} \Pi^h_{[s_I,r_I],I,J,\widetilde{I\backslash J},A},\\	
\Pi^h_{\mathrm{an},\mathrm{con},I,\breve{J},\widetilde{I\backslash J},A}:=\varinjlim_{r_I}\varprojlim_{s_I} \Pi^h_{[s_I,r_I],I,\breve{J},\widetilde{I\backslash J},A},\\	
\Pi^h_{\mathrm{an},\mathrm{con},I,\widetilde{J},\widetilde{I\backslash J},A}:=\varinjlim_{r_I}\varprojlim_{s_I} \Pi^h_{[s_I,r_I],I,\widetilde{J},\widetilde{I\backslash J},A}.\\	
\end{align}	
\end{definition}

\subsection{Fr\'echet Objects in the Sheafy Situation}

\begin{assumption}
Assume the followings are sheafy adic Banach uniform algebra over $\mathbb{Q}_p$:
\begin{align}
\Pi_{[s_I,r_I],I,?,?',A},?=J,\widetilde{J},\breve{J},?'=I\backslash J,\widetilde{I\backslash J},\breve{I\backslash J}.	
\end{align}	
\end{assumption}

\indent In what follow, we consider the corresponding radii living in the set of all the rational numbers.

\begin{definition} \mbox{\bf{(After KPX, \cite[Definition 2.1.3]{KPX})}}
Over $\Pi_{\mathrm{an},r_{I,0},I,\breve{J},I\backslash J,A}$ we define the corresponding stably pseudocoherent sheaves to mean a collection of stably pseudocoherent modules $(M_{[s_I,r_I]})$	over each $\Pi_{[s_I,r_I],I,\breve{J},I\backslash J,A}$ satisfying the corresponding compatibility condition and the obvious cocycle condition with respect to the family of the corresponding multi-intervals $\{[s_I,r_I]\}$.
\end{definition}

\begin{definition} \mbox{\bf{(After KPX, \cite[Definition 2.1.3]{KPX})}}
Over $\Pi_{\mathrm{an},r_{I,0},I,\widetilde{J},I\backslash J,A}$ we define the corresponding stably pseudocoherent sheaves to mean a collection of stably pseudocoherent modules $(M_{[s_I,r_I]})$	over each $\Pi_{[s_I,r_I],I,\widetilde{J},I\backslash J,A}$ satisfying the corresponding compatibility condition and the obvious cocycle condition with respect to the family of the corresponding multi-intervals $\{[s_I,r_I]\}$.
\end{definition}

\begin{definition} \mbox{\bf{(After KPX, \cite[Definition 2.1.3]{KPX})}}
Over $\Pi_{\mathrm{an},r_{I,0},I,{J},\breve{I\backslash J},A}$ we define the corresponding stably pseudocoherent sheaves to mean a collection of stably pseudocoherent modules $(M_{[s_I,r_I]})$	over each $\Pi_{[s_I,r_I],I,{J},\breve{I\backslash J},A}$ satisfying the corresponding compatibility condition and the obvious cocycle condition with respect to the family of the corresponding multi-intervals $\{[s_I,r_I]\}$.
\end{definition}

\begin{definition} \mbox{\bf{(After KPX, \cite[Definition 2.1.3]{KPX})}}
Over $\Pi_{\mathrm{an},r_{I,0},I,\breve{J},\breve{I\backslash J},A}$ we define the corresponding stably pseudocoherent sheaves to mean a collection of stably pseudocoherent modules $(M_{[s_I,r_I]})$	over each $\Pi_{[s_I,r_I],I,\breve{J},\breve{I\backslash J},A}$ satisfying the corresponding compatibility condition and the obvious cocycle condition with respect to the family of the corresponding multi-intervals $\{[s_I,r_I]\}$.
\end{definition}

\begin{definition} \mbox{\bf{(After KPX, \cite[Definition 2.1.3]{KPX})}}
Over $\Pi_{\mathrm{an},r_{I,0},I,\widetilde{J},\breve{I\backslash J},A}$ we define the corresponding stably pseudocoherent sheaves to mean a collection of stably pseudocoherent modules $(M_{[s_I,r_I]})$ over each $\Pi_{[s_I,r_I],I,\widetilde{J},\breve{I\backslash J},A}$ satisfying the corresponding compatibility condition and the obvious cocycle condition with respect to the family of the corresponding multi-intervals $\{[s_I,r_I]\}$.
\end{definition}

\begin{definition} \mbox{\bf{(After KPX, \cite[Definition 2.1.3]{KPX})}}
Over $\Pi_{\mathrm{an},r_{I,0},I,{J},\widetilde{I\backslash J},A}$ we define the corresponding stably pseudocoherent sheaves to mean a collection of stably pseudocoherent modules $(M_{[s_I,r_I]})$	over each $\Pi_{[s_I,r_I],I,{J},\widetilde{I\backslash J},A}$ satisfying the corresponding compatibility condition and the obvious cocycle condition with respect to the family of the corresponding multi-intervals $\{[s_I,r_I]\}$.
\end{definition}

\begin{definition} \mbox{\bf{(After KPX, \cite[Definition 2.1.3]{KPX})}}
Over $\Pi_{\mathrm{an},r_{I,0},I,\breve{J},\widetilde{I\backslash J},A}$ we define the corresponding stably pseudocoherent sheaves to mean a collection of stably pseudocoherent modules $(M_{[s_I,r_I]})$ over each $\Pi_{[s_I,r_I],I,\breve{J},\widetilde{I\backslash J},A}$ satisfying the corresponding compatibility condition and the obvious cocycle condition with respect to the family of the corresponding multi-intervals $\{[s_I,r_I]\}$.
\end{definition}

\begin{definition} \mbox{\bf{(After KPX, \cite[Definition 2.1.3]{KPX})}}
Over $\Pi_{\mathrm{an},r_{I,0},I,\widetilde{J},\widetilde{I\backslash J},A}$ we define the corresponding stably pseudocoherent sheaves to mean a collection of stably pseudocoherent modules $(M_{[s_I,r_I]})$ over each $\Pi_{[s_I,r_I],I,\widetilde{J},\widetilde{I\backslash J},A}$ satisfying the corresponding compatibility condition and the obvious cocycle condition with respect to the family of the corresponding multi-intervals $\{[s_I,r_I]\}$.
\end{definition}

\begin{remark}
There is some overlap and repeating on some objects within the eight categories. Therefore in the following we are going to then work with only 1st, 2nd, 4th, 5th, 8th categories. 	
\end{remark}

\begin{proposition} \mbox{\bf{(After Kedlaya-Liu, \cite[Corollary 2.6.8]{KL2})}}
The corresponding pseudocoherent finitely projective bundles over
\begin{align}
\Pi_{\mathrm{an},r_{I,0},I,\breve{J},I\backslash J,A},\\
\Pi_{\mathrm{an},r_{I,0},I,\widetilde{J},I\backslash J,A},\\
\Pi_{\mathrm{an},r_{I,0},I,\breve{J},\breve{I\backslash J},A},\\
\Pi_{\mathrm{an},r_{I,0},I,\widetilde{J},\breve{I\backslash J},A},\\
\Pi_{\mathrm{an},r_{I,0},I,\widetilde{J},\widetilde{I\backslash J},A}
\end{align}
defined as above have global sections which are finite projective if and only if the global sections are finitely generated. 	
\end{proposition}

\begin{proof}
Just apply \cite[Corollary 2.6.8]{KL2}.	
\end{proof}

\begin{proposition} \mbox{\bf{(After Kedlaya-Liu, \cite[Proposition 2.6.17]{KL2})}}
The corresponding pseudocoherent sheaves over
\begin{align}
\Pi_{\mathrm{an},r_{I,0},I,\breve{J},I\backslash J,A},\\
\Pi_{\mathrm{an},r_{I,0},I,\widetilde{J},I\backslash J,A},\\
\Pi_{\mathrm{an},r_{I,0},I,\breve{J},\breve{I\backslash J},A},\\
\Pi_{\mathrm{an},r_{I,0},I,\widetilde{J},\breve{I\backslash J},A},\\
\Pi_{\mathrm{an},r_{I,0},I,\widetilde{J},\widetilde{I\backslash J},A}
\end{align}
defined as above have global sections which are finitely generated as long as we have the uniform bound on the rank of the bundles over quasi-compacts with respect to closed multi-intervals taking the general form of $[s_I,r_I]$. 	
\end{proposition}

\begin{proof}
This is actually a direct consequence of \cite[Proposition 2.6.17]{KL2}, where the space $(0,r_{I,0}]$ admits $2^{|I|}$-uniform covering.	
\end{proof}

\begin{proposition} \mbox{\bf{(After Kedlaya-Liu, \cite[Corollary 2.6.8, Proposition 2.6.17]{KL2})}}
The corresponding pseudocoherent sheaves over
\begin{align}
\Pi_{\mathrm{an},r_{I,0},I,\breve{J},I\backslash J,A},\\
\Pi_{\mathrm{an},r_{I,0},I,\widetilde{J},I\backslash J,A},\\
\Pi_{\mathrm{an},r_{I,0},I,\breve{J},\breve{I\backslash J},A},\\
\Pi_{\mathrm{an},r_{I,0},I,\widetilde{J},\breve{I\backslash J},A},\\
\Pi_{\mathrm{an},r_{I,0},I,\widetilde{J},\widetilde{I\backslash J},A}
\end{align}
defined as above have global sections which are finite projective as long as we have the uniform bound on the rank of the bundle over each quasi-compact with respect to each closed multi-interval $[s_I,r_I]$, and we have that the sheaves admits section actually finite projective over each quasi-compact with respect to each closed multi-interval $[s_I,r_I]$. 	
\end{proposition}

\begin{proof}
This is actually a direct corollary of the previous two propositions. 	
\end{proof}

\newpage

\section{Cyclotomic Multivariate $(\varphi_I,\Gamma_I)$-Modules over Rigid Analytic Affinoids in Mixed-characteristic Case}

\subsection{Fundamental Definitions}

\noindent In the situation where $A$ is a rigid analytic affinoid over $\mathbb{Q}_p$. Recall from \cite{T1}, suppose we have $|I|$ finite extensions of $\mathbb{Q}_p$, where we denote them as $K_1,...,K_I$. Then we have the corresponding uniformizers $\pi_{K_1},...,\pi_{K_I}$, the corresponding Frobenius operators $\varphi_1,...,\varphi_I$ and the corresponding groups $\Gamma_{K_1},...,\Gamma_{K_I}$. Recall from \cite{T1}, by adding the corresponding variables from $\pi_{K_1},...,\pi_{K_I}$ and $\Gamma_{K_1},...,\Gamma_{K_I}$ we have the following rings:

\begin{align}
\Pi_{[s_I,r_I],I,I,\emptyset,A}(\pi_{K_I}):=\Pi_{[s_I,r_I],I,A}(\pi_{K_I})\\
\Pi_{\mathrm{an},r_I,I,I,\emptyset,A}(\pi_{K_I}):=\Pi_{\mathrm{an},r_I,I,A}(\pi_{K_I})\\
\Pi_{\mathrm{an,con},I,I,\emptyset,A}(\pi_{K_I}):=\Pi_{\mathrm{an,con},I,\emptyset,A}(\pi_{K_I})	
\end{align}
and
\begin{align}
\Pi_{[s_I,r_I],I,I,\emptyset,A}(\Gamma_{K_I}):=\Pi_{[s_I,r_I],I,A}(\Gamma_{K_I})\\
\Pi_{\mathrm{an},r_I,I,I,\emptyset,A}(\Gamma_{K_I}):=\Pi_{\mathrm{an},r_I,I,A}(\Gamma_{K_I})\\
\Pi_{\mathrm{an,con},I,I,\emptyset,A}(\Gamma_{K_I}):=\Pi_{\mathrm{an,con},I,\emptyset,A}(\Gamma_{K_I}).	
\end{align}

\begin{definition}
By taking the direct base change we have the following rings in mixed-characteristic situation:
\begin{align}
\Pi_{[s_I,r_I],I,\breve{J},I\backslash J,A}(\pi_{K_I}),\\	
\Pi_{[s_I,r_I],I,\widetilde{J},I\backslash J,A}(\pi_{K_I}),\\
\Pi_{[s_I,r_I],I,J,\breve{I\backslash J},A}(\pi_{K_I}),\\	
\Pi_{[s_I,r_I],I,\breve{J},\breve{I\backslash J},A}(\pi_{K_I}),\\	
\Pi_{[s_I,r_I],I,\widetilde{J},\breve{I\backslash J},A}(\pi_{K_I}),\\
\Pi_{[s_I,r_I],I,J,\widetilde{I\backslash J},A}(\pi_{K_I}),\\	
\Pi_{[s_I,r_I],I,\breve{J},\widetilde{I\backslash J},A}(\pi_{K_I}),\\	
\Pi_{[s_I,r_I],I,\widetilde{J},\widetilde{I\backslash J},A}(\pi_{K_I}).	
\end{align}
By taking the direct base change we have the following rings in mixed-characteristic situation:
\begin{align}
\Pi_{[s_I,r_I],I,\breve{J},I\backslash J,A}(\Gamma_{K_I}),\\	
\Pi_{[s_I,r_I],I,\widetilde{J},I\backslash J,A}(\Gamma_{K_I}),\\
\Pi_{[s_I,r_I],I,J,\breve{I\backslash J},A}(\Gamma_{K_I}),\\	
\Pi_{[s_I,r_I],I,\breve{J},\breve{I\backslash J},A}(\Gamma_{K_I}),\\
\Pi_{[s_I,r_I],I,\widetilde{J},\breve{I\backslash J},A}(\Gamma_{K_I}),\\
\Pi_{[s_I,r_I],I,J,\widetilde{I\backslash J},A}(\Gamma_{K_I}),\\	
\Pi_{[s_I,r_I],I,\breve{J},\widetilde{I\backslash J},A}(\Gamma_{K_I}),\\	
\Pi_{[s_I,r_I],I,\widetilde{J},\widetilde{I\backslash J},A}(\Gamma_{K_I}).	
\end{align}	
\end{definition}

\begin{definition} \mbox{\bf{(After Kedlaya-Liu, \cite[Definition 5.2.1]{KL2})}}
\begin{align}
\Pi_{\mathrm{an},r_I,I,\breve{J},I\backslash J,A}(\pi_{K_I}):=\varprojlim_{s_I}\Pi_{[s_I,r_I],I,\breve{J},I\backslash J,A}(\pi_{K_I}),\\	
\Pi_{\mathrm{an},r_I,I,\widetilde{J},I\backslash J,A}(\pi_{K_I}):=\varprojlim_{s_I} \Pi_{[s_I,r_I],I,\widetilde{J},I\backslash J,A}(\pi_{K_I}),\\
\Pi_{\mathrm{an},r_I,I,J,\breve{I\backslash J},A}(\pi_{K_I}):=\varprojlim_{s_I}\Pi_{[s_I,r_I],I,J,\breve{I\backslash J},A}(\pi_{K_I}),\\	
\Pi_{\mathrm{an},r_I,I,\breve{J},\breve{I\backslash J},A}(\pi_{K_I}):=\varprojlim_{s_I} \Pi_{[s_I,r_I],I,\breve{J},\breve{I\backslash J},A}(\pi_{K_I}),\\	
\Pi_{\mathrm{an},r_I,I,\widetilde{J},\breve{I\backslash J},A}(\pi_{K_I}):=\varprojlim_{s_I} \Pi_{[s_I,r_I],I,\widetilde{J},\breve{I\backslash J},A}(\pi_{K_I}),\\
\Pi_{\mathrm{an},r_I,I,J,\widetilde{I\backslash J},A}(\pi_{K_I}):=\varprojlim_{s_I} \Pi_{[s_I,r_I],I,J,\widetilde{I\backslash J},A}(\pi_{K_I}),\\	
\Pi_{\mathrm{an},r_I,I,\breve{J},\widetilde{I\backslash J},A}(\pi_{K_I}):=\varprojlim_{s_I} \Pi_{[s_I,r_I],I,\breve{J},\widetilde{I\backslash J},A}(\pi_{K_I}),\\	
\Pi_{\mathrm{an},r_I,I,\widetilde{J},\widetilde{I\backslash J},A}(\pi_{K_I}):=\varprojlim_{s_I} \Pi_{[s_I,r_I],I,\widetilde{J},\widetilde{I\backslash J},A}(\pi_{K_I}).	
\end{align}

and 

\begin{align}
\Pi_{\mathrm{an},\mathrm{con},I,\breve{J},I\backslash J,A}(\pi_{K_I}):=\varinjlim_{r_I}\varprojlim_{s_I}\Pi_{[s_I,r_I],I,\breve{J},I\backslash J,A}(\pi_{K_I}),\\	
\Pi_{\mathrm{an},\mathrm{con},I,\widetilde{J},I\backslash J,A}(\pi_{K_I}):=\varinjlim_{r_I}\varprojlim_{s_I} \Pi_{[s_I,r_I],I,\widetilde{J},I\backslash J,A}(\pi_{K_I}),\\
\Pi_{\mathrm{an},\mathrm{con},I,J,\breve{I\backslash J},A}(\pi_{K_I}):=\varinjlim_{r_I}\varprojlim_{s_I}\Pi_{[s_I,r_I],I,J,\breve{I\backslash J},A}(\pi_{K_I}),\\	
\Pi_{\mathrm{an},\mathrm{con},I,\breve{J},\breve{I\backslash J},A}(\pi_{K_I}):=\varinjlim_{r_I}\varprojlim_{s_I} \Pi_{[s_I,r_I],I,\breve{J},\breve{I\backslash J},A}(\pi_{K_I}),\\
\Pi_{\mathrm{an},\mathrm{con},I,\widetilde{J},\breve{I\backslash J},A}(\pi_{K_I}):=\varinjlim_{r_I}\varprojlim_{s_I} \Pi_{[s_I,r_I],I,\widetilde{J},\breve{I\backslash J},A}(\pi_{K_I}),\\
\Pi_{\mathrm{an},\mathrm{con},I,J,\widetilde{I\backslash J},A}(\pi_{K_I}):=\varinjlim_{r_I}\varprojlim_{s_I} \Pi_{[s_I,r_I],I,J,\widetilde{I\backslash J},A}(\pi_{K_I}),\\	
\Pi_{\mathrm{an},\mathrm{con},I,\breve{J},\widetilde{I\backslash J},A}(\pi_{K_I}):=\varinjlim_{r_I}\varprojlim_{s_I} \Pi_{[s_I,r_I],I,\breve{J},\widetilde{I\backslash J},A}(\pi_{K_I}),\\	
\Pi_{\mathrm{an},\mathrm{con},I,\widetilde{J},\widetilde{I\backslash J},A}(\pi_{K_I}):=\varinjlim_{r_I}\varprojlim_{s_I} \Pi_{[s_I,r_I],I,\widetilde{J},\widetilde{I\backslash J},A}(\pi_{K_I}).\\	
\end{align}	
\end{definition}

\begin{definition} \mbox{\bf{(After Kedlaya-Liu, \cite[Definition 5.2.1]{KL2})}}
\begin{align}
\Pi_{\mathrm{an},r_I,I,\breve{J},I\backslash J,A}(\Gamma_{K_I}):=\varprojlim_{s_I}\Pi_{[s_I,r_I],I,\breve{J},I\backslash J,A}(\Gamma_{K_I}),\\	
\Pi_{\mathrm{an},r_I,I,\widetilde{J},I\backslash J,A}(\Gamma_{K_I}):=\varprojlim_{s_I} \Pi_{[s_I,r_I],I,\widetilde{J},I\backslash J,A}(\Gamma_{K_I}),\\
\Pi_{\mathrm{an},r_I,I,J,\breve{I\backslash J},A}(\Gamma_{K_I}):=\varprojlim_{s_I}\Pi_{[s_I,r_I],I,J,\breve{I\backslash J},A}(\Gamma_{K_I}),\\	
\Pi_{\mathrm{an},r_I,I,\breve{J},\breve{I\backslash J},A}(\Gamma_{K_I}):=\varprojlim_{s_I} \Pi_{[s_I,r_I],I,\breve{J},\breve{I\backslash J},A}(\Gamma_{K_I}),\\	
\Pi_{\mathrm{an},r_I,I,\widetilde{J},\breve{I\backslash J},A}(\Gamma_{K_I}):=\varprojlim_{s_I} \Pi_{[s_I,r_I],I,\widetilde{J},\breve{I\backslash J},A}(\Gamma_{K_I}),\\
\Pi_{\mathrm{an},r_I,I,J,\widetilde{I\backslash J},A}(\Gamma_{K_I}):=\varprojlim_{s_I} \Pi_{[s_I,r_I],I,J,\widetilde{I\backslash J},A}(\Gamma_{K_I}),\\	
\Pi_{\mathrm{an},r_I,I,\breve{J},\widetilde{I\backslash J},A}(\Gamma_{K_I}):=\varprojlim_{s_I} \Pi_{[s_I,r_I],I,\breve{J},\widetilde{I\backslash J},A}(\Gamma_{K_I}),\\	
\Pi_{\mathrm{an},r_I,I,\widetilde{J},\widetilde{I\backslash J},A}(\Gamma_{K_I}):=\varprojlim_{s_I} \Pi_{[s_I,r_I],I,\widetilde{J},\widetilde{I\backslash J},A}(\Gamma_{K_I}).	
\end{align}

and 

\begin{align}
\Pi_{\mathrm{an},\mathrm{con},I,\breve{J},I\backslash J,A}(\Gamma_{K_I}):=\varinjlim_{r_I}\varprojlim_{s_I}\Pi_{[s_I,r_I],I,\breve{J},I\backslash J,A}(\Gamma_{K_I}),\\	
\Pi_{\mathrm{an},\mathrm{con},I,\widetilde{J},I\backslash J,A}(\Gamma_{K_I}):=\varinjlim_{r_I}\varprojlim_{s_I} \Pi_{[s_I,r_I],I,\widetilde{J},I\backslash J,A}(\Gamma_{K_I}),\\
\Pi_{\mathrm{an},\mathrm{con},I,J,\breve{I\backslash J},A}(\Gamma_{K_I}):=\varinjlim_{r_I}\varprojlim_{s_I}\Pi_{[s_I,r_I],I,J,\breve{I\backslash J},A}(\Gamma_{K_I}),\\	
\Pi_{\mathrm{an},\mathrm{con},I,\breve{J},\breve{I\backslash J},A}(\Gamma_{K_I}):=\varinjlim_{r_I}\varprojlim_{s_I} \Pi_{[s_I,r_I],I,\breve{J},\breve{I\backslash J},A}(\Gamma_{K_I}),\\
\Pi_{\mathrm{an},\mathrm{con},I,\widetilde{J},\breve{I\backslash J},A}(\Gamma_{K_I}):=\varinjlim_{r_I}\varprojlim_{s_I} \Pi_{[s_I,r_I],I,\widetilde{J},\breve{I\backslash J},A}(\Gamma_{K_I}),\\
\Pi_{\mathrm{an},\mathrm{con},I,J,\widetilde{I\backslash J},A}(\Gamma_{K_I}):=\varinjlim_{r_I}\varprojlim_{s_I} \Pi_{[s_I,r_I],I,J,\widetilde{I\backslash J},A}(\Gamma_{K_I}),\\	
\Pi_{\mathrm{an},\mathrm{con},I,\breve{J},\widetilde{I\backslash J},A}(\Gamma_{K_I}):=\varinjlim_{r_I}\varprojlim_{s_I} \Pi_{[s_I,r_I],I,\breve{J},\widetilde{I\backslash J},A}(\Gamma_{K_I}),\\	
\Pi_{\mathrm{an},\mathrm{con},I,\widetilde{J},\widetilde{I\backslash J},A}(\Gamma_{K_I}):=\varinjlim_{r_I}\varprojlim_{s_I} \Pi_{[s_I,r_I],I,\widetilde{J},\widetilde{I\backslash J},A}(\Gamma_{K_I}).	
\end{align}	
\end{definition}

\indent Now we consider the corresponding definition of the $(\varphi_I,\Gamma_I)$-modules over the corresponding period rings  defined above.

\begin{setting}
For the corresponding $\varphi_I$-modules over the rings:
\begin{align}
\Pi_{\mathrm{an},r_I,I,\breve{J},I\backslash J,A}(\pi_{K_I}):=\varprojlim_{s_I}\Pi_{[s_I,r_I],I,\breve{J},I\backslash J,A}(\pi_{K_I}),\\	
\Pi_{\mathrm{an},r_I,I,\widetilde{J},I\backslash J,A}(\pi_{K_I}):=\varprojlim_{s_I} \Pi_{[s_I,r_I],I,\widetilde{J},I\backslash J,A}(\pi_{K_I}),\\
\Pi_{\mathrm{an},r_I,I,J,\breve{I\backslash J},A}(\pi_{K_I}):=\varprojlim_{s_I}\Pi_{[s_I,r_I],I,J,\breve{I\backslash J},A}(\pi_{K_I}),\\	
\Pi_{\mathrm{an},r_I,I,\breve{J},\breve{I\backslash J},A}(\pi_{K_I}):=\varprojlim_{s_I} \Pi_{[s_I,r_I],I,\breve{J},\breve{I\backslash J},A}(\pi_{K_I}),\\	
\Pi_{\mathrm{an},r_I,I,\widetilde{J},\breve{I\backslash J},A}(\pi_{K_I}):=\varprojlim_{s_I} \Pi_{[s_I,r_I],I,\widetilde{J},\breve{I\backslash J},A}(\pi_{K_I}),\\
\Pi_{\mathrm{an},r_I,I,J,\widetilde{I\backslash J},A}(\pi_{K_I}):=\varprojlim_{s_I} \Pi_{[s_I,r_I],I,J,\widetilde{I\backslash J},A}(\pi_{K_I}),\\	
\Pi_{\mathrm{an},r_I,I,\breve{J},\widetilde{I\backslash J},A}(\pi_{K_I}):=\varprojlim_{s_I} \Pi_{[s_I,r_I],I,\breve{J},\widetilde{I\backslash J},A}(\pi_{K_I}),\\	
\Pi_{\mathrm{an},r_I,I,\widetilde{J},\widetilde{I\backslash J},A}(\pi_{K_I}):=\varprojlim_{s_I} \Pi_{[s_I,r_I],I,\widetilde{J},\widetilde{I\backslash J},A}(\pi_{K_I})	
\end{align}	
we assume we have the sufficiently small radius as in \cite[Definition 2.2.6]{KPX}. We will keep this assumption in all similar situation involving the corresponding $\varphi_I$-modules.
\end{setting}

\begin{definition} \mbox{\bf{(After KPX \cite[Definition 2.2.6]{KPX})}}
We define in the following way the corresponding $\varphi_I$-modules over the following period rings:
\begin{align}
\Pi_{\mathrm{an},\mathrm{con},I,\breve{J},I\backslash J,A}(\pi_{K_I}):=\varinjlim_{r_I}\varprojlim_{s_I}\Pi_{[s_I,r_I],I,\breve{J},I\backslash J,A}(\pi_{K_I}),\\	
\Pi_{\mathrm{an},\mathrm{con},I,\widetilde{J},I\backslash J,A}(\pi_{K_I}):=\varinjlim_{r_I}\varprojlim_{s_I} \Pi_{[s_I,r_I],I,\widetilde{J},I\backslash J,A}(\pi_{K_I}),\\
\Pi_{\mathrm{an},\mathrm{con},I,J,\breve{I\backslash J},A}(\pi_{K_I}):=\varinjlim_{r_I}\varprojlim_{s_I}\Pi_{[s_I,r_I],I,J,\breve{I\backslash J},A}(\pi_{K_I}),\\	
\Pi_{\mathrm{an},\mathrm{con},I,\breve{J},\breve{I\backslash J},A}(\pi_{K_I}):=\varinjlim_{r_I}\varprojlim_{s_I} \Pi_{[s_I,r_I],I,\breve{J},\breve{I\backslash J},A}(\pi_{K_I}),\\
\Pi_{\mathrm{an},\mathrm{con},I,\widetilde{J},\breve{I\backslash J},A}(\pi_{K_I}):=\varinjlim_{r_I}\varprojlim_{s_I} \Pi_{[s_I,r_I],I,\widetilde{J},\breve{I\backslash J},A}(\pi_{K_I}),\\
\Pi_{\mathrm{an},\mathrm{con},I,J,\widetilde{I\backslash J},A}(\pi_{K_I}):=\varinjlim_{r_I}\varprojlim_{s_I} \Pi_{[s_I,r_I],I,J,\widetilde{I\backslash J},A}(\pi_{K_I}),\\	
\Pi_{\mathrm{an},\mathrm{con},I,\breve{J},\widetilde{I\backslash J},A}(\pi_{K_I}):=\varinjlim_{r_I}\varprojlim_{s_I} \Pi_{[s_I,r_I],I,\breve{J},\widetilde{I\backslash J},A}(\pi_{K_I}),\\	
\Pi_{\mathrm{an},\mathrm{con},I,\widetilde{J},\widetilde{I\backslash J},A}(\pi_{K_I}):=\varinjlim_{r_I}\varprojlim_{s_I} \Pi_{[s_I,r_I],I,\widetilde{J},\widetilde{I\backslash J},A}(\pi_{K_I}).	
\end{align}	
These are the corresponding base change of the corresponding $\varphi_I$-modules coming from the the ones over the following rings:
\begin{align}
\Pi_{\mathrm{an},r_I,I,\breve{J},I\backslash J,A}(\pi_{K_I}):=\varprojlim_{s_I}\Pi_{[s_I,r_I],I,\breve{J},I\backslash J,A}(\pi_{K_I}),\\	
\Pi_{\mathrm{an},r_I,I,\widetilde{J},I\backslash J,A}(\pi_{K_I}):=\varprojlim_{s_I} \Pi_{[s_I,r_I],I,\widetilde{J},I\backslash J,A}(\pi_{K_I}),\\
\Pi_{\mathrm{an},r_I,I,J,\breve{I\backslash J},A}(\pi_{K_I}):=\varprojlim_{s_I}\Pi_{[s_I,r_I],I,J,\breve{I\backslash J},A}(\pi_{K_I}),\\	
\Pi_{\mathrm{an},r_I,I,\breve{J},\breve{I\backslash J},A}(\pi_{K_I}):=\varprojlim_{s_I} \Pi_{[s_I,r_I],I,\breve{J},\breve{I\backslash J},A}(\pi_{K_I}),\\	
\Pi_{\mathrm{an},r_I,I,\widetilde{J},\breve{I\backslash J},A}(\pi_{K_I}):=\varprojlim_{s_I} \Pi_{[s_I,r_I],I,\widetilde{J},\breve{I\backslash J},A}(\pi_{K_I}),\\
\Pi_{\mathrm{an},r_I,I,J,\widetilde{I\backslash J},A}(\pi_{K_I}):=\varprojlim_{s_I} \Pi_{[s_I,r_I],I,J,\widetilde{I\backslash J},A}(\pi_{K_I}),\\	
\Pi_{\mathrm{an},r_I,I,\breve{J},\widetilde{I\backslash J},A}(\pi_{K_I}):=\varprojlim_{s_I} \Pi_{[s_I,r_I],I,\breve{J},\widetilde{I\backslash J},A}(\pi_{K_I}),\\	
\Pi_{\mathrm{an},r_I,I,\widetilde{J},\widetilde{I\backslash J},A}(\pi_{K_I}):=\varprojlim_{s_I} \Pi_{[s_I,r_I],I,\widetilde{J},\widetilde{I\backslash J},A}(\pi_{K_I}).	
\end{align}
For this latter group of period rings, we define a corresponding pseudocoherent or finite projective $\varphi_I$-module to be a corresponding pseudocoherent or finite projective module $M$ over this latter group of period rings carrying the corresponding semilinear partial Frobenius action coming from each Frobenius operator $\varphi_\alpha,\alpha\in I$ such that for each $\alpha$ we have 
\begin{align}
\varphi_\alpha^*M\otimes_{\Pi_{\mathrm{an},\{...,r_\alpha/p,...\},I,*,*,A}(\pi_{K_I})} &\Pi_{\mathrm{an},\{...,r_\alpha/p,...\},I,*,*,A}(\pi_{K_I})\\
&\overset{\sim}{\longrightarrow} M\otimes_{\Pi_{\mathrm{an},\{...,r_\alpha,...\},I,*,*,A}(\pi_{K_I})}\Pi_{\mathrm{an},\{...,r_\alpha/p,...\},I,*,*,A}(\pi_{K_I}).	
\end{align}
And we assume that altogether the partial Frobenius operators are commuting with each other. We assume all the modules involved are complete for the natural topology (mainly in the pseudocoherent situation) whose base changes to the following rings:
\begin{align}
\Pi_{[s_I,r_I],I,\breve{J},I\backslash J,A}(\pi_{K_I}),\\	
\Pi_{[s_I,r_I],I,\widetilde{J},I\backslash J,A}(\pi_{K_I}),\\
\Pi_{[s_I,r_I],I,J,\breve{I\backslash J},A}(\pi_{K_I}),\\	
\Pi_{[s_I,r_I],I,\breve{J},\breve{I\backslash J},A}(\pi_{K_I}),\\
\Pi_{[s_I,r_I],I,\widetilde{J},\breve{I\backslash J},A}(\pi_{K_I}),\\
\Pi_{[s_I,r_I],I,J,\widetilde{I\backslash J},A}(\pi_{K_I}),\\	
\Pi_{[s_I,r_I],I,\breve{J},\widetilde{I\backslash J},A}(\pi_{K_I}),\\
\Pi_{[s_I,r_I],I,\widetilde{J},\widetilde{I\backslash J},A}(\pi_{K_I})
\end{align}	
give rise to the corresponding $\varphi_I$-modules defined over these rings which defined in the following way.	
\noindent We define a corresponding pseudocoherent or finite projective $\varphi_I$-module over:
\begin{align}
\Pi_{[s_I,r_I],I,\breve{J},I\backslash J,A}(\pi_{K_I}),\\	
\Pi_{[s_I,r_I],I,\widetilde{J},I\backslash J,A}(\pi_{K_I}),\\
\Pi_{[s_I,r_I],I,J,\breve{I\backslash J},A}(\pi_{K_I}),\\	
\Pi_{[s_I,r_I],I,\breve{J},\breve{I\backslash J},A}(\pi_{K_I}),\\
\Pi_{[s_I,r_I],I,\widetilde{J},\breve{I\backslash J},A}(\pi_{K_I}),\\
\Pi_{[s_I,r_I],I,J,\widetilde{I\backslash J},A}(\pi_{K_I}),\\	
\Pi_{[s_I,r_I],I,\breve{J},\widetilde{I\backslash J},A}(\pi_{K_I}),\\
\Pi_{[s_I,r_I],I,\widetilde{J},\widetilde{I\backslash J},A}(\pi_{K_I})	
\end{align}
to be a corresponding stably-pseudocoherent or finite projective module $M$ over this latter group of period rings carrying the corresponding semilinear partial Frobenius action coming from each Frobenius operator $\varphi_\alpha,\alpha\in I$ such that for each $\alpha$ we have 
\begin{align}
\varphi_\alpha^*M\otimes_{\Pi_{\mathrm{an},...,[s_\alpha/p,r_\alpha/p],...,I,*,*,A}(\pi_{K_I})} &\Pi_{\mathrm{an},...,[s_\alpha,r_\alpha/p],...,I,*,*,A}(\pi_{K_I})\\
&\overset{\sim}{\longrightarrow} M\otimes_{\Pi_{\mathrm{an},...,[s_\alpha,r_\alpha],...,I,*,*,A}(\pi_{K_I})}\Pi_{\mathrm{an},...,[s_\alpha,r_\alpha/p],...,I,*,*,A}(\pi_{K_I}).	
\end{align}
And we assume that altogether the partial Frobenius operators are commuting with each other. We assume all the modules involved are complete for the natural topology (mainly in the pseudocoherent situation).\\
\noindent Finally we define the pseudocoherent or finite projective $\varphi_I$-modules over the corresponding period rings:
\begin{align}
\Pi_{\mathrm{an},\mathrm{con},I,\breve{J},I\backslash J,A}(\pi_{K_I}):=\varinjlim_{r_I}\varprojlim_{s_I}\Pi_{[s_I,r_I],I,\breve{J},I\backslash J,A}(\pi_{K_I}),\\	
\Pi_{\mathrm{an},\mathrm{con},I,\widetilde{J},I\backslash J,A}(\pi_{K_I}):=\varinjlim_{r_I}\varprojlim_{s_I} \Pi_{[s_I,r_I],I,\widetilde{J},I\backslash J,A}(\pi_{K_I}),\\
\Pi_{\mathrm{an},\mathrm{con},I,J,\breve{I\backslash J},A}(\pi_{K_I}):=\varinjlim_{r_I}\varprojlim_{s_I}\Pi_{[s_I,r_I],I,J,\breve{I\backslash J},A}(\pi_{K_I}),\\	
\Pi_{\mathrm{an},\mathrm{con},I,\breve{J},\breve{I\backslash J},A}(\pi_{K_I}):=\varinjlim_{r_I}\varprojlim_{s_I} \Pi_{[s_I,r_I],I,\breve{J},\breve{I\backslash J},A}(\pi_{K_I}),\\
\Pi_{\mathrm{an},\mathrm{con},I,\widetilde{J},\breve{I\backslash J},A}(\pi_{K_I}):=\varinjlim_{r_I}\varprojlim_{s_I} \Pi_{[s_I,r_I],I,\widetilde{J},\breve{I\backslash J},A}(\pi_{K_I}),\\
\Pi_{\mathrm{an},\mathrm{con},I,J,\widetilde{I\backslash J},A}(\pi_{K_I}):=\varinjlim_{r_I}\varprojlim_{s_I} \Pi_{[s_I,r_I],I,J,\widetilde{I\backslash J},A}(\pi_{K_I}),\\	
\Pi_{\mathrm{an},\mathrm{con},I,\breve{J},\widetilde{I\backslash J},A}(\pi_{K_I}):=\varinjlim_{r_I}\varprojlim_{s_I} \Pi_{[s_I,r_I],I,\breve{J},\widetilde{I\backslash J},A}(\pi_{K_I}),\\	
\Pi_{\mathrm{an},\mathrm{con},I,\widetilde{J},\widetilde{I\backslash J},A}(\pi_{K_I}):=\varinjlim_{r_I}\varprojlim_{s_I} \Pi_{[s_I,r_I],I,\widetilde{J},\widetilde{I\backslash J},A}(\pi_{K_I})
\end{align}
to be the corresponding base changes of some $\varphi_I$-modules over the rings:
\begin{align}
\Pi_{\mathrm{an},r_I,I,\breve{J},I\backslash J,A}(\pi_{K_I}):=\varprojlim_{s_I}\Pi_{[s_I,r_I],I,\breve{J},I\backslash J,A}(\pi_{K_I}),\\	
\Pi_{\mathrm{an},r_I,I,\widetilde{J},I\backslash J,A}(\pi_{K_I}):=\varprojlim_{s_I} \Pi_{[s_I,r_I],I,\widetilde{J},I\backslash J,A}(\pi_{K_I}),\\
\Pi_{\mathrm{an},r_I,I,J,\breve{I\backslash J},A}(\pi_{K_I}):=\varprojlim_{s_I}\Pi_{[s_I,r_I],I,J,\breve{I\backslash J},A}(\pi_{K_I}),\\	
\Pi_{\mathrm{an},r_I,I,\breve{J},\breve{I\backslash J},A}(\pi_{K_I}):=\varprojlim_{s_I} \Pi_{[s_I,r_I],I,\breve{J},\breve{I\backslash J},A}(\pi_{K_I}),\\	
\Pi_{\mathrm{an},r_I,I,\widetilde{J},\breve{I\backslash J},A}(\pi_{K_I}):=\varprojlim_{s_I} \Pi_{[s_I,r_I],I,\widetilde{J},\breve{I\backslash J},A}(\pi_{K_I}),\\
\Pi_{\mathrm{an},r_I,I,J,\widetilde{I\backslash J},A}(\pi_{K_I}):=\varprojlim_{s_I} \Pi_{[s_I,r_I],I,J,\widetilde{I\backslash J},A}(\pi_{K_I}),\\	
\Pi_{\mathrm{an},r_I,I,\breve{J},\widetilde{I\backslash J},A}(\pi_{K_I}):=\varprojlim_{s_I} \Pi_{[s_I,r_I],I,\breve{J},\widetilde{I\backslash J},A}(\pi_{K_I}),\\	
\Pi_{\mathrm{an},r_I,I,\widetilde{J},\widetilde{I\backslash J},A}(\pi_{K_I}):=\varprojlim_{s_I} \Pi_{[s_I,r_I],I,\widetilde{J},\widetilde{I\backslash J},A}(\pi_{K_I}),	
\end{align}
with the corresponding requirement that they are basically complete with respect to the natural topology and the partial Frobenius operators are commuting with each other.

\end{definition}

\begin{definition} \mbox{\bf{(After KPX \cite[Definition 2.2.6]{KPX})}}
Then we define a corresponding pseudocoherent or finite projective $\varphi_I$-sheaf $F$ over one of the following period rings:
\begin{align}
\Pi_{\mathrm{an},r_{I,0},I,\breve{J},I\backslash J,A}(\pi_{K_I}):=\varprojlim_{s_I}\Pi_{[s_I,r_I],I,\breve{J},I\backslash J,A}(\pi_{K_I}),\\	
\Pi_{\mathrm{an},r_{I,0},I,\widetilde{J},I\backslash J,A}(\pi_{K_I}):=\varprojlim_{s_I} \Pi_{[s_I,r_I],I,\widetilde{J},I\backslash J,A}(\pi_{K_I}),\\
\Pi_{\mathrm{an},r_{I,0},I,J,\breve{I\backslash J},A}(\pi_{K_I}):=\varprojlim_{s_I}\Pi_{[s_I,r_I],I,J,\breve{I\backslash J},A}(\pi_{K_I}),\\	
\Pi_{\mathrm{an},r_{I,0},I,\breve{J},\breve{I\backslash J},A}(\pi_{K_I}):=\varprojlim_{s_I} \Pi_{[s_I,r_I],I,\breve{J},\breve{I\backslash J},A}(\pi_{K_I}),\\	
\Pi_{\mathrm{an},r_{I,0},I,\widetilde{J},\breve{I\backslash J},A}(\pi_{K_I}):=\varprojlim_{s_I} \Pi_{[s_I,r_I],I,\widetilde{J},\breve{I\backslash J},A}(\pi_{K_I}),\\
\Pi_{\mathrm{an},r_{I,0},I,J,\widetilde{I\backslash J},A}(\pi_{K_I}):=\varprojlim_{s_I} \Pi_{[s_I,r_I],I,J,\widetilde{I\backslash J},A}(\pi_{K_I}),\\	
\Pi_{\mathrm{an},r_{I,0},I,\breve{J},\widetilde{I\backslash J},A}(\pi_{K_I}):=\varprojlim_{s_I} \Pi_{[s_I,r_I],I,\breve{J},\widetilde{I\backslash J},A}(\pi_{K_I}),\\	
\Pi_{\mathrm{an},r_{I,0},I,\widetilde{J},\widetilde{I\backslash J},A}(\pi_{K_I}):=\varprojlim_{s_I} \Pi_{[s_I,r_I],I,\widetilde{J},\widetilde{I\backslash J},A}(\pi_{K_I})	
\end{align}	
to be the corresponding compatible family of the $\varphi_I$-modules over any one $\Pi_{[s_I,r_I],I,*,*,A}(\pi_{K_I})$ of the following rings:
\begin{align}
\Pi_{[s_I,r_I],I,\breve{J},I\backslash J,A}(\pi_{K_I}),\\	
\Pi_{[s_I,r_I],I,\widetilde{J},I\backslash J,A}(\pi_{K_I}),\\
\Pi_{[s_I,r_I],I,J,\breve{I\backslash J},A}(\pi_{K_I}),\\	
\Pi_{[s_I,r_I],I,\breve{J},\breve{I\backslash J},A}(\pi_{K_I}),\\
\Pi_{[s_I,r_I],I,\widetilde{J},\breve{I\backslash J},A}(\pi_{K_I}),\\
\Pi_{[s_I,r_I],I,J,\widetilde{I\backslash J},A}(\pi_{K_I}),\\	
\Pi_{[s_I,r_I],I,\breve{J},\widetilde{I\backslash J},A}(\pi_{K_I}),\\
\Pi_{[s_I,r_I],I,\widetilde{J},\widetilde{I\backslash J},A}(\pi_{K_I})	
\end{align}
satisfying the corresponding restriction requirement and the corresponding cocycle condition as in \cite[Definition 2.2.6]{KPX}, such that $[s_I,r_I]\subset (0,r_{I,0}]$.
\end{definition}

\begin{definition}  \mbox{\bf{(After KPX \cite[Definition 2.2.12]{KPX})}}
As in \cite[Definition 2.2.12]{KPX} we impose the corresponding $\Gamma_I$-structure by adding the corresponding semilinear continuous action of $\Gamma_I$ on the modules induced from that on the period rings, which are assumed to be commuting with the action from $\varphi_I$.	
\end{definition}

\subsection{The Comparison Theorems}

\indent Now we establish some results on the comparison on the corresponding $\varphi_I$-modules we defined above.

\begin{theorem}\mbox{\bf{(After KPX \cite[Proposition 2.2.7]{KPX})}} 
Consider the following categories:\\
1. The category of all the finite projective $\varphi_I$-modules over the ring $\Pi_{\mathrm{an},r_{I,0},I,?,?,A}(\pi_{K_I})$;\\
2. The category of all the finite projective $\varphi_I$-sheaves over the ring $\Pi_{\mathrm{an},r_{I,0},I,?,?,A}(\pi_{K_I})$.\\
Then we have that the two categories are equivalent.	
\end{theorem}

\begin{proof}
The base change gives rise to the corresponding fully faithful functor from the first category to the second one, while to show the corresponding essential surjectivity, consider the corresponding multi-interval $[r_{1,0}/p,r_{1,0}]\times...\times [r_{I,0}/p,r_{I,0}]$ and use the corresponding Frobenius to reach all the corresponding intervals taking the general form of:
\begin{align}
[r_{1,0}/p^{k_1},r_{1,0}/p^{k_1-1}]\times...\times [r_{I,0}/p^{k_I},r_{I,0}/p^{k_I-1}],k_\alpha=1,2,...,\forall\alpha\in I.	
\end{align}
This forms a $2^{|I|}$-uniform covering of the whole space. And the corresponding uniform finiteness of the modules over each 
\begin{align}
[r_{1,0}/p^{k_1},r_{1,0}/p^{k_1-1}]\times...\times [r_{I,0}/p^{k_I},r_{I,0}/p^{k_I-1}],k_\alpha=1,2,...,\forall\alpha\in I.	
\end{align}	
could be achieved by using the corresponding partial Frobenius actions. Then we are done by applying \cref{proposition2.19}.
\end{proof}

\begin{theorem}\mbox{\bf{(After KPX \cite[Proposition 2.2.7]{KPX})}} 
Consider the following categories:\\
1. The category of all the pseudocoherent $\varphi_I$-modules over the ring $\Pi_{\mathrm{an},r_{I,0},I,?,?,A}(\pi_{K_I})$;\\
2. The category of all the pseudocoherent $\varphi_I$-sheaves over the ring $\Pi_{\mathrm{an},r_{I,0},I,?,?,A}(\pi_{K_I})$.\\
Then we have that the two categories are equivalent.	
\end{theorem}

\begin{proof}
The base change gives rise to the corresponding fully faithful functor from the first category to the second one, while to show the corresponding essential surjectivity, consider the corresponding multi-interval $[r_{1,0}/p,r_{1,0}]\times...\times [r_{I,0}/p,r_{I,0}]$ and use the corresponding Frobenius to reach all the corresponding intervals taking the general form of:
\begin{align}
[r_{1,0}/p^{k_1},r_{1,0}/p^{k_1-1}]\times...\times [r_{I,0}/p^{k_I},r_{I,0}/p^{k_I-1}],k_\alpha=1,2,...,\forall\alpha\in I.	
\end{align}
This forms a $2^{|I|}$-uniform covering of the whole space. And the corresponding uniform finiteness of the modules over each 
\begin{align}
[r_{1,0}/p^{k_1},r_{1,0}/p^{k_1-1}]\times...\times [r_{I,0}/p^{k_I},r_{I,0}/p^{k_I-1}],k_\alpha=1,2,...,\forall\alpha\in I.	
\end{align}	
could be achieved by using the corresponding partial Frobenius actions. Then we are done by applying \cref{proposition2.18}. Then one choose finite free covering to promote the finiteness to pseudocoherence as in \cite[Theorem 4.6.1, Lemma 5.4.11]{KL2}.
\end{proof}


\indent We now consider the corresponding vertical comparison for $I=\{1,2\}$:

\begin{theorem}\mbox{\bf{(After Kedlaya-Liu \cite[Theorem 5.7.5]{KL2})}} 
Consider the following categories:\\
1. The category of all the finite projective $(\varphi_I,\Gamma_I)$-modules over the ring $\Pi_{[s_I,r_{I}],I,J,I\backslash J,A}(\pi_{K_I})$;\\
2. The category of all the finite projective $(\varphi_I,\Gamma_I)$-modules over the ring $\Pi_{[s_I,r_{I}],I,J,\breve{I\backslash J},A}(\pi_{K_I})$;\\
3. The category of all the finite projective $(\varphi_I,\Gamma_I)$-modules over the ring $\Pi_{[s_I,r_{I}],I,J,\widetilde{I\backslash J},A}(\pi_{K_I})$.\\
Then we have that these categories are equivalent. Here $0< s_\alpha<r_\alpha<\infty$ for any $\alpha\in I$.
\end{theorem}

\indent This is the consequence of the following theorem:

\begin{theorem} \mbox{\bf{(After Kedlaya-Liu \cite[Theorem 5.7.5]{KL2})}}
Consider the following categories:\\
1. The category of all the finite projective $(\varphi_2,\Gamma_2)$-modules over the ring $\Pi_{[s_I,r_{I}],I,{1},I\backslash J,A}(\pi_{K_I})$;\\
2. The category of all the finite projective $(\varphi_2,\Gamma_2)$-modules over the ring $\Pi_{[s_I,r_{I}],I,{1},\breve{I\backslash J},A}(\pi_{K_I})$;\\
3. The category of all the finite projective $(\varphi_2,\Gamma_2)$-modules over the ring $\Pi_{[s_I,r_{I}],I,{1},\widetilde{I\backslash J},A}(\pi_{K_I})$.\\
Then we have that these categories are equivalent. Here $0< s_\alpha<r_\alpha<\infty$ for any $\alpha\in I$.
\end{theorem}

\begin{proof}
In this situation this is just the corresponding relative comparison for finite projective $(\varphi,\Gamma)$ modules, which is \cite[Theorem 4.4]{KP}.	
\end{proof}

\indent We now consider the corresponding vertical comparison in the following context for $I=\{1,2\}$:

\begin{theorem}\mbox{\bf{(After Kedlaya-Liu \cite[Theorem 5.7.5]{KL2})}}
Consider the following categories:\\
1. The category of all the pseudocoherent $(\varphi_I,\Gamma_I)$-modules over the ring $\Pi_{[s_I,r_{I}],I,J,I\backslash J,A}(\pi_{K_I})$;\\
2. The category of all the pseudocoherent $(\varphi_I,\Gamma_I)$-modules over the ring $\Pi_{[s_I,r_{I}],I,J,\breve{I\backslash J},A}(\pi_{K_I})$;\\
3. The category of all the pseudocoherent $(\varphi_I,\Gamma_I)$-modules over the ring $\Pi_{[s_I,r_{I}],I,J,\widetilde{I\backslash J},A}(\pi_{K_I})$.\\
Then we have that these categories are equivalent. Here $0< s_\alpha<r_\alpha<\infty$ for any $\alpha\in I$.
\end{theorem}

\indent This is the consequence of the following theorem:

\begin{theorem} \mbox{\bf{(After Kedlaya-Liu \cite[Theorem 5.7.5]{KL2})}}
Consider the following categories:\\
1. The category of all the pseudocoherent $(\varphi_2,\Gamma_2)$-modules over the ring $\Pi_{[s_I,r_{I}],I,{1},I\backslash J,A}(\pi_{K_I})$;\\
2. The category of all the pseudocoherent $(\varphi_2,\Gamma_2)$-modules over the ring $\Pi_{[s_I,r_{I}],I,{1},\breve{I\backslash J},A}(\pi_{K_I})$;\\
3. The category of all the pseudocoherent $(\varphi_2,\Gamma_2)$-modules over the ring $\Pi_{[s_I,r_{I}],I,{1},\widetilde{I\backslash J},A}(\pi_{K_I})$.\\
Then we have that these categories are equivalent. Here $0< s_\alpha<r_\alpha<\infty$ for any $\alpha\in I$.
\end{theorem}

\begin{proof}
In this situation this is just the corresponding relative comparison for finite projective $(\varphi,\Gamma)$ modules, which is \cite[Proposition 5.44, Proposition 5.51]{T2}.	
\end{proof}

\begin{corollary}
Assume $I=\{1,2\}$. Consider the following categories:\\
1. The category of all the pseudocoherent $(\varphi_I,\Gamma_I)$-modules over the ring $\Pi_{[s_I,r_{I}],I,J,I\backslash J,A}(\pi_{K_I})$;\\
2. The category of all the pseudocoherent $(\varphi_I,\Gamma_I)$-modules over the ring $\Pi_{[s_I,r_{I}],I,J,\breve{I\backslash J},A}(\pi_{K_I})$;\\
3. The category of all the pseudocoherent $(\varphi_I,\Gamma_I)$-modules over the ring $\Pi_{[s_I,r_{I}],I,J,\widetilde{I\backslash J},A}(\pi_{K_I})$.\\
Then we have that these categories are equivalent. Here $0< s_\alpha\leq r_\alpha/p<\infty$ for any $\alpha\in I$. \\
Consider the following categories:\\
1. The category of all the finite projective $(\varphi_I,\Gamma_I)$-modules over the ring $\Pi_{[s_I,r_{I}],I,J,I\backslash J,A}(\pi_{K_I})$;\\
2. The category of all the finite projective $(\varphi_I,\Gamma_I)$-modules over the ring $\Pi_{[s_I,r_{I}],I,J,\breve{I\backslash J},A}(\pi_{K_I})$;\\
3. The category of all the finite projective $(\varphi_I,\Gamma_I)$-modules over the ring $\Pi_{[s_I,r_{I}],I,J,\widetilde{I\backslash J},A}(\pi_{K_I})$.\\
Then we have that these categories are equivalent. Here $0< s_\alpha\leq r_\alpha/p<\infty$ for any $\alpha\in I$.	
\end{corollary}

\newpage

\section{$B_{I}$-Pairs and Intermediate Objects in Rigid Family}

\subsection{Mixed-type Hodge Structures}

\noindent Now we work with the corresponding $B$-pairs and some mixed-type objects as in \cite{Ber1} and \cite{Nak1}.

\begin{definition}
Define $B^+_{\mathrm{dR},I}:=\mathbb{C}_p[[t_1,...,t_I]]$, and define  $B_{\mathrm{dR},I}:=\mathbb{C}_p[[t_1,...,t_I]][t_1^{-1},...,t_I^{-1}]$, and similarly we have the obvious higher dimensional analog $B_{e,I}$ of the corresponding period ring $B_{e}$ which could be defined as:
\begin{align}
\varinjlim_{k_1}...	\varinjlim_{k_I}t_1^{-k_1}...t_I^{-k_I}\bigcap_{i=1,...,|I|}(\varprojlim_{n\rightarrow \infty}B_\mathrm{max,1}^+/p^n\widehat{\otimes}_{\mathbb{Z}_p}...\widehat{\otimes}_{\mathbb{Z}_p}B_{\mathrm{max},I}^+/p^n)^{\varphi_i-p^{k_i}}.
\end{align}

.	
\end{definition}

\indent Then after Bloch-Kato \cite{BK1} we have the following fundamental sequence:

\begin{proposition}
We have the higher dimensional generalization of the corresponding Bloch-Kato fundamental sequence:
\[
\xymatrix@C+0pc@R+0pc{
0 \ar[r] \ar[r] \ar[r] &\mathbb{Q}_p \ar[r] \ar[r] \ar[r]  &B_{e,I}\bigoplus B^+_{\mathrm{dR},I} \ar[r] \ar[r] \ar[r] &B_{\mathrm{dR},I} \ar[r] \ar[r] \ar[r] &0
}
\]
induced by the corresponding short exact sequence:
\[
\xymatrix@C+0pc@R+0pc{
0 \ar[r] \ar[r] \ar[r] &\mathbb{Q}_p \ar[r] \ar[r] \ar[r]  &B_{e,I}\bigoplus \mathbb{C}_p[[t_1,...,t_I]] \ar[r] \ar[r] \ar[r] &\mathbb{C}_p[[t_1,...,t_I]][t_1^{-1},...,t_I^{-1}] \ar[r] \ar[r] \ar[r] &0.
}
\]	
\end{proposition}

\begin{setting}
In what follows, we assume that the corresponding $A$ to be a rigid affinoid in rigid analytic geometry over $\mathbb{Q}_p$.	
\end{setting}

\begin{definition}
We now consider the following rings:
\begin{align}
B^+_{\mathrm{dR},I'}	\widehat{\otimes}\Pi_{[s_I,r_I],I,J,I\backslash J,A},\\	
B^+_{\mathrm{dR},I'}	\widehat{\otimes}\Pi_{[s_I,r_I],I,\breve{J},I\backslash J,A},\\	
B^+_{\mathrm{dR},I'}	\widehat{\otimes}\Pi_{[s_I,r_I],I,\widetilde{J},I\backslash J,A},\\
B^+_{\mathrm{dR},I'}	\widehat{\otimes}\Pi_{[s_I,r_I],I,J,\breve{I\backslash J},A},\\	
B^+_{\mathrm{dR},I'}	\widehat{\otimes}\Pi_{[s_I,r_I],I,\breve{J},\breve{I\backslash J},A},\\
B^+_{\mathrm{dR},I'}	\widehat{\otimes}\Pi_{[s_I,r_I],I,\widetilde{J},\breve{I\backslash J},A},\\
B^+_{\mathrm{dR},I'}	\widehat{\otimes}\Pi_{[s_I,r_I],I,J,\widetilde{I\backslash J},A},\\	
B^+_{\mathrm{dR},I'}	\widehat{\otimes}\Pi_{[s_I,r_I],I,\breve{J},\widetilde{I\backslash J},A},\\
B^+_{\mathrm{dR},I'}	\widehat{\otimes}\Pi_{[s_I,r_I],I,\widetilde{J},\widetilde{I\backslash J},A}
\end{align}

with

\begin{align}
B^+_{\mathrm{dR},I'}	\widehat{\otimes}\Pi_{[s_I,r_I],I,J,I\backslash J,A}[t_1^{-1},...,t_{I'}^{-1}],\\	
B^+_{\mathrm{dR},I'}	\widehat{\otimes}\Pi_{[s_I,r_I],I,\breve{J},I\backslash J,A}[t_1^{-1},...,t_{I'}^{-1}],\\	
B^+_{\mathrm{dR},I'}	\widehat{\otimes}\Pi_{[s_I,r_I],I,\widetilde{J},I\backslash J,A}[t_1^{-1},...,t_{I'}^{-1}],\\
B^+_{\mathrm{dR},I'}	\widehat{\otimes}\Pi_{[s_I,r_I],I,J,\breve{I\backslash J},A}[t_1^{-1},...,t_{I'}^{-1}],\\	
B^+_{\mathrm{dR},I'}	\widehat{\otimes}\Pi_{[s_I,r_I],I,\breve{J},\breve{I\backslash J},A}[t_1^{-1},...,t_{I'}^{-1}],\\
B^+_{\mathrm{dR},I'}	\widehat{\otimes}\Pi_{[s_I,r_I],I,\widetilde{J},\breve{I\backslash J},A}[t_1^{-1},...,t_{I'}^{-1}],\\
B^+_{\mathrm{dR},I'}	\widehat{\otimes}\Pi_{[s_I,r_I],I,J,\widetilde{I\backslash J},A}[t_1^{-1},...,t_{I'}^{-1}],\\	
B^+_{\mathrm{dR},I'}	\widehat{\otimes}\Pi_{[s_I,r_I],I,\breve{J},\widetilde{I\backslash J},A}[t_1^{-1},...,t_{I'}^{-1}],\\
B^+_{\mathrm{dR},I'}	\widehat{\otimes}\Pi_{[s_I,r_I],I,\widetilde{J},\widetilde{I\backslash J},A}[t_1^{-1},...,t_{I'}^{-1}],
\end{align}

with 

\begin{align}
B_{e,I'}	\widehat{\otimes}\Pi_{[s_I,r_I],I,J,I\backslash J,A},\\	
B_{e,I'}	\widehat{\otimes}\Pi_{[s_I,r_I],I,\breve{J},I\backslash J,A},\\	
B_{e,I'}	\widehat{\otimes}\Pi_{[s_I,r_I],I,\widetilde{J},I\backslash J,A},\\
B_{e,I'}	\widehat{\otimes}\Pi_{[s_I,r_I],I,J,\breve{I\backslash J},A},\\	
B_{e,I'}	\widehat{\otimes}\Pi_{[s_I,r_I],I,\breve{J},\breve{I\backslash J},A},\\
B_{e,I'}	\widehat{\otimes}\Pi_{[s_I,r_I],I,\widetilde{J},\breve{I\backslash J},A},\\
B_{e,I'}	\widehat{\otimes}\Pi_{[s_I,r_I],I,J,\widetilde{I\backslash J},A},\\	
B_{e,I'}	\widehat{\otimes}\Pi_{[s_I,r_I],I,\breve{J},\widetilde{I\backslash J},A},\\
B_{e,I'} \widehat{\otimes}\Pi_{[s_I,r_I],I,\widetilde{J},\widetilde{I\backslash J},A}
\end{align}

with

\begin{align}
\varprojlim_{s_I}  B^+_{\mathrm{dR},I'}	\widehat{\otimes}\Pi_{[s_I,r_I],I,J,I\backslash J,A},\\	
\varprojlim_{s_I}  B^+_{\mathrm{dR},I'}	\widehat{\otimes}\Pi_{[s_I,r_I],I,\breve{J},I\backslash J,A},\\	
\varprojlim_{s_I} B^+_{\mathrm{dR},I'}	\widehat{\otimes}\Pi_{[s_I,r_I],I,\widetilde{J},I\backslash J,A},\\
\varprojlim_{s_I} B^+_{\mathrm{dR},I'}	\widehat{\otimes}\Pi_{[s_I,r_I],I,J,\breve{I\backslash J},A},\\	
\varprojlim_{s_I} B^+_{\mathrm{dR},I'}	\widehat{\otimes}\Pi_{[s_I,r_I],I,\breve{J},\breve{I\backslash J},A},\\
\varprojlim_{s_I} B^+_{\mathrm{dR},I'}	\widehat{\otimes}\Pi_{[s_I,r_I],I,\widetilde{J},\breve{I\backslash J},A},\\
\varprojlim_{s_I} B^+_{\mathrm{dR},I'}	\widehat{\otimes}\Pi_{[s_I,r_I],I,J,\widetilde{I\backslash J},A},\\	
\varprojlim_{s_I} B^+_{\mathrm{dR},I'}	\widehat{\otimes}\Pi_{[s_I,r_I],I,\breve{J},\widetilde{I\backslash J},A},\\
\varprojlim_{s_I} B^+_{\mathrm{dR},I'}	\widehat{\otimes}\Pi_{[s_I,r_I],I,\widetilde{J},\widetilde{I\backslash J},A}
\end{align}

with

\begin{align}
\varprojlim_{s_I} B^+_{\mathrm{dR},I'}	\widehat{\otimes}\Pi_{[s_I,r_I],I,J,I\backslash J,A}[t_1^{-1},...,t_{I'}^{-1}],\\	
\varprojlim_{s_I} B^+_{\mathrm{dR},I'}	\widehat{\otimes}\Pi_{[s_I,r_I],I,\breve{J},I\backslash J,A}[t_1^{-1},...,t_{I'}^{-1}],\\	
\varprojlim_{s_I} B^+_{\mathrm{dR},I'}	\widehat{\otimes}\Pi_{[s_I,r_I],I,\widetilde{J},I\backslash J,A}[t_1^{-1},...,t_{I'}^{-1}],\\
\varprojlim_{s_I} B^+_{\mathrm{dR},I'}	\widehat{\otimes}\Pi_{[s_I,r_I],I,J,\breve{I\backslash J},A}[t_1^{-1},...,t_{I'}^{-1}],\\	
\varprojlim_{s_I} B^+_{\mathrm{dR},I'}	\widehat{\otimes}\Pi_{[s_I,r_I],I,\breve{J},\breve{I\backslash J},A}[t_1^{-1},...,t_{I'}^{-1}],\\
\varprojlim_{s_I} B^+_{\mathrm{dR},I'}	\widehat{\otimes}\Pi_{[s_I,r_I],I,\widetilde{J},\breve{I\backslash J},A}[t_1^{-1},...,t_{I'}^{-1}],\\
\varprojlim_{s_I} B^+_{\mathrm{dR},I'}	\widehat{\otimes}\Pi_{[s_I,r_I],I,J,\widetilde{I\backslash J},A}[t_1^{-1},...,t_{I'}^{-1}],\\	
\varprojlim_{s_I} B^+_{\mathrm{dR},I'}	\widehat{\otimes}\Pi_{[s_I,r_I],I,\breve{J},\widetilde{I\backslash J},A}[t_1^{-1},...,t_{I'}^{-1}],\\
\varprojlim_{s_I} B^+_{\mathrm{dR},I'}	\widehat{\otimes}\Pi_{[s_I,r_I],I,\widetilde{J},\widetilde{I\backslash J},A}[t_1^{-1},...,t_{I'}^{-1}],
\end{align}

with 

\begin{align}
\varprojlim_{s_I} B_{e,I'}	\widehat{\otimes}\Pi_{[s_I,r_I],I,J,I\backslash J,A},\\	
\varprojlim_{s_I} B_{e,I'}	\widehat{\otimes}\Pi_{[s_I,r_I],I,\breve{J},I\backslash J,A},\\	
\varprojlim_{s_I} B_{e,I'}	\widehat{\otimes}\Pi_{[s_I,r_I],I,\widetilde{J},I\backslash J,A},\\
\varprojlim_{s_I} B_{e,I'}	\widehat{\otimes}\Pi_{[s_I,r_I],I,J,\breve{I\backslash J},A},\\	
\varprojlim_{s_I} B_{e,I'}	\widehat{\otimes}\Pi_{[s_I,r_I],I,\breve{J},\breve{I\backslash J},A},\\
\varprojlim_{s_I} B_{e,I'}	\widehat{\otimes}\Pi_{[s_I,r_I],I,\widetilde{J},\breve{I\backslash J},A},\\
\varprojlim_{s_I} B_{e,I'}	\widehat{\otimes}\Pi_{[s_I,r_I],I,J,\widetilde{I\backslash J},A},\\	
\varprojlim_{s_I} B_{e,I'}	\widehat{\otimes}\Pi_{[s_I,r_I],I,\breve{J},\widetilde{I\backslash J},A},\\
\varprojlim_{s_I} B_{e,I'} \widehat{\otimes}\Pi_{[s_I,r_I],I,\widetilde{J},\widetilde{I\backslash J},A}
\end{align}

with

\begin{align}
\varinjlim_{r_I}\varprojlim_{s_I}  B^+_{\mathrm{dR},I'}	\widehat{\otimes}\Pi_{[s_I,r_I],I,J,I\backslash J,A},\\	
\varinjlim_{r_I}\varprojlim_{s_I}  B^+_{\mathrm{dR},I'}	\widehat{\otimes}\Pi_{[s_I,r_I],I,\breve{J},I\backslash J,A},\\	
\varinjlim_{r_I}\varprojlim_{s_I} B^+_{\mathrm{dR},I'}	\widehat{\otimes}\Pi_{[s_I,r_I],I,\widetilde{J},I\backslash J,A},\\
\varinjlim_{r_I}\varprojlim_{s_I} B^+_{\mathrm{dR},I'}	\widehat{\otimes}\Pi_{[s_I,r_I],I,J,\breve{I\backslash J},A},\\	
\varinjlim_{r_I}\varprojlim_{s_I} B^+_{\mathrm{dR},I'}	\widehat{\otimes}\Pi_{[s_I,r_I],I,\breve{J},\breve{I\backslash J},A},\\
\varinjlim_{r_I}\varprojlim_{s_I} B^+_{\mathrm{dR},I'}	\widehat{\otimes}\Pi_{[s_I,r_I],I,\widetilde{J},\breve{I\backslash J},A},\\
\varinjlim_{r_I}\varprojlim_{s_I} B^+_{\mathrm{dR},I'}	\widehat{\otimes}\Pi_{[s_I,r_I],I,J,\widetilde{I\backslash J},A},\\	
\varinjlim_{r_I}\varprojlim_{s_I} B^+_{\mathrm{dR},I'}	\widehat{\otimes}\Pi_{[s_I,r_I],I,\breve{J},\widetilde{I\backslash J},A},\\
\varinjlim_{r_I}\varprojlim_{s_I} B^+_{\mathrm{dR},I'}	\widehat{\otimes}\Pi_{[s_I,r_I],I,\widetilde{J},\widetilde{I\backslash J},A}
\end{align}

with

\begin{align}
\varinjlim_{r_I}\varprojlim_{s_I} B^+_{\mathrm{dR},I'}	\widehat{\otimes}\Pi_{[s_I,r_I],I,J,I\backslash J,A}[t_1^{-1},...,t_{I'}^{-1}],\\	
\varinjlim_{r_I}\varprojlim_{s_I} B^+_{\mathrm{dR},I'}	\widehat{\otimes}\Pi_{[s_I,r_I],I,\breve{J},I\backslash J,A}[t_1^{-1},...,t_{I'}^{-1}],\\	
\varinjlim_{r_I}\varprojlim_{s_I} B^+_{\mathrm{dR},I'}	\widehat{\otimes}\Pi_{[s_I,r_I],I,\widetilde{J},I\backslash J,A}[t_1^{-1},...,t_{I'}^{-1}],\\
\varinjlim_{r_I}\varprojlim_{s_I} B^+_{\mathrm{dR},I'}	\widehat{\otimes}\Pi_{[s_I,r_I],I,J,\breve{I\backslash J},A}[t_1^{-1},...,t_{I'}^{-1}],\\	
\varinjlim_{r_I}\varprojlim_{s_I} B^+_{\mathrm{dR},I'}	\widehat{\otimes}\Pi_{[s_I,r_I],I,\breve{J},\breve{I\backslash J},A}[t_1^{-1},...,t_{I'}^{-1}],\\
\varinjlim_{r_I}\varprojlim_{s_I} B^+_{\mathrm{dR},I'}	\widehat{\otimes}\Pi_{[s_I,r_I],I,\widetilde{J},\breve{I\backslash J},A}[t_1^{-1},...,t_{I'}^{-1}],\\
\varinjlim_{r_I}\varprojlim_{s_I} B^+_{\mathrm{dR},I'}	\widehat{\otimes}\Pi_{[s_I,r_I],I,J,\widetilde{I\backslash J},A}[t_1^{-1},...,t_{I'}^{-1}],\\	
\varinjlim_{r_I}\varprojlim_{s_I} B^+_{\mathrm{dR},I'}	\widehat{\otimes}\Pi_{[s_I,r_I],I,\breve{J},\widetilde{I\backslash J},A}[t_1^{-1},...,t_{I'}^{-1}],\\
\varinjlim_{r_I}\varprojlim_{s_I} B^+_{\mathrm{dR},I'}	\widehat{\otimes}\Pi_{[s_I,r_I],I,\widetilde{J},\widetilde{I\backslash J},A}[t_1^{-1},...,t_{I'}^{-1}],
\end{align}

with 

\begin{align}
\varinjlim_{r_I}\varprojlim_{s_I} B_{e,I'}	\widehat{\otimes}\Pi_{[s_I,r_I],I,J,I\backslash J,A},\\	
\varinjlim_{r_I}\varprojlim_{s_I} B_{e,I'}	\widehat{\otimes}\Pi_{[s_I,r_I],I,\breve{J},I\backslash J,A},\\	
\varinjlim_{r_I}\varprojlim_{s_I} B_{e,I'}	\widehat{\otimes}\Pi_{[s_I,r_I],I,\widetilde{J},I\backslash J,A},\\
\varinjlim_{r_I}\varprojlim_{s_I} B_{e,I'}	\widehat{\otimes}\Pi_{[s_I,r_I],I,J,\breve{I\backslash J},A},\\	
\varinjlim_{r_I}\varprojlim_{s_I} B_{e,I'}	\widehat{\otimes}\Pi_{[s_I,r_I],I,\breve{J},\breve{I\backslash J},A},\\
\varinjlim_{r_I}\varprojlim_{s_I} B_{e,I'}	\widehat{\otimes}\Pi_{[s_I,r_I],I,\widetilde{J},\breve{I\backslash J},A},\\
\varinjlim_{r_I}\varprojlim_{s_I} B_{e,I'}	\widehat{\otimes}\Pi_{[s_I,r_I],I,J,\widetilde{I\backslash J},A},\\	
\varinjlim_{r_I}\varprojlim_{s_I} B_{e,I'}	\widehat{\otimes}\Pi_{[s_I,r_I],I,\breve{J},\widetilde{I\backslash J},A},\\
\varinjlim_{r_I}\varprojlim_{s_I} B_{e,I'} \widehat{\otimes}\Pi_{[s_I,r_I],I,\widetilde{J},\widetilde{I\backslash J},A}.
\end{align}\\
	
\end{definition}

\indent Now we combine the construction in the following coherent way following \cite[Section 2]{Ber1}, \cite[Definition 2.2]{Nak1} and \cite[Definition 2.2.6]{KPX}.

\begin{definition} 
We define a $B_{I'}$-$(\varphi_I,\Gamma_I)$-module over
\begin{center}
$(B^+_{\mathrm{dR},I'}	\widehat{\otimes}\Pi_{[s_I,r_I],I,J,I\backslash J,A},B^+_{\mathrm{dR},I'}	\widehat{\otimes}\Pi_{[s_I,r_I],I,*,*,A}[t_1^{-1},...,t_I'^{-1}],B_{e,I'}	\widehat{\otimes}\Pi_{[s_I,r_I],I,*,*,A})$	
\end{center}
to be a triplet of finite projective modules:
\begin{align}
(M_e,M_{\mathrm{dR}},M^+_{\mathrm{dR}})	
\end{align}
over 
\begin{center}
$(B^+_{\mathrm{dR},I'}	\widehat{\otimes}\Pi_{[s_I,r_I],I,*,*,A},B^+_{\mathrm{dR},I'}	\widehat{\otimes}\Pi_{[s_I,r_I],I,*,*,A}[t_1^{-1},...,t_I'^{-1}],B_{e,I'}	\widehat{\otimes}\Pi_{[s_I,r_I],I,*,*,A})$	
\end{center}
such that we have glueing datum along :
\begin{align}
B^+_{\mathrm{dR},I'}	\widehat{\otimes}\Pi_{[s_I,r_I],I,*,*,A}\rightarrow B^+_{\mathrm{dR},I'}	\widehat{\otimes}\Pi_{[s_I,r_I],I,*,*,A}[t_1^{-1},...,t_I'^{-1}] \leftarrow B_{e,I'}	\widehat{\otimes}\Pi_{[s_I,r_I],I,*,*,A}.	
\end{align}
And this carry the corresponding relative Galois action of:
\begin{align}
\mathrm{Gal}_{\mathbb{Q}_p,1}\times...\times \mathrm{Gal}_{\mathbb{Q}_p,I'}	
\end{align}
on the multi de Rham period rings which is semilinear. And we have that the three modules involved are relative $(\varphi_I,\Gamma_I)$-modules relative to
\begin{align}
B^+_{\mathrm{dR},I'},B_{\mathrm{dR},I'},B_{e,I'}.	
\end{align}

\end{definition}

\begin{definition}
We define a pseudocoherent $B_{I'}$-$(\varphi_I,\Gamma_I)$-module over
\begin{center}
$(B^+_{\mathrm{dR},I'}	\widehat{\otimes}\Pi_{[s_I,r_I],I,J,I\backslash J,A},B^+_{\mathrm{dR},I'}	\widehat{\otimes}\Pi_{[s_I,r_I],I,*,*,A}[t_1^{-1},...,t_I'^{-1}],B_{e,I'}	\widehat{\otimes}\Pi_{[s_I,r_I],I,*,*,A})$	
\end{center}
to be a triplet of stably-pseudocoherent modules:
\begin{align}
(M_e,M_{\mathrm{dR}},M^+_{\mathrm{dR}})	
\end{align}
over 
\begin{center}
$(B^+_{\mathrm{dR},I'}	\widehat{\otimes}\Pi_{[s_I,r_I],I,*,*,A},B^+_{\mathrm{dR},I'}	\widehat{\otimes}\Pi_{[s_I,r_I],I,*,*,A}[t_1^{-1},...,t_I'^{-1}],B_{e,I'}	\widehat{\otimes}\Pi_{[s_I,r_I],I,*,*,A})$	
\end{center}
such that this is glueing datum along :
\begin{align}
B^+_{\mathrm{dR},I'}	\widehat{\otimes}\Pi_{[s_I,r_I],I,*,*,A}\rightarrow B^+_{\mathrm{dR},I'}	\widehat{\otimes}\Pi_{[s_I,r_I],I,*,*,A}[t_1^{-1},...,t_I'^{-1}] \leftarrow B_{e,I'}	\widehat{\otimes}\Pi_{[s_I,r_I],I,*,*,A}.	
\end{align}
And this carry the corresponding relative Galois action of:
\begin{align}
\mathrm{Gal}_{\mathbb{Q}_p,1}\times...\times \mathrm{Gal}_{\mathbb{Q}_p,I'}		
\end{align}
on the multi de Rham period rings which is semilinear. And we have that the three modules involved are relative pseudocoherent $(\varphi_I,\Gamma_I)$-modules relative to
\begin{align}
B^+_{\mathrm{dR},I'},B_{\mathrm{dR},I'},B_{e,I'}.	
\end{align}

\end{definition}

\begin{definition}
We define a $B_{I'}$-$(\varphi_I,\Gamma_I)$-module over
\begin{center}
$(B^+_{\mathrm{dR},I'}	\widehat{\otimes}\Pi_{\mathrm{an},r_{I},I,J,I\backslash J,A},B^+_{\mathrm{dR},I'}	\widehat{\otimes}\Pi_{[s_I,r_I],I,*,*,A}[t_1^{-1},...,t_I'^{-1}],B_{e,I'}	\widehat{\otimes}\Pi_{\mathrm{an},r_{I},I,*,*,A})$	
\end{center}
to be a triplet of finite projective modules:
\begin{align}
(M_e,M_{\mathrm{dR}},M^+_{\mathrm{dR}})	
\end{align}
over 
\begin{center}
$(B^+_{\mathrm{dR},I'}	\widehat{\otimes}\Pi_{\mathrm{an},r_{I},I,*,*,A},B^+_{\mathrm{dR},I'}	\widehat{\otimes}\Pi_{\mathrm{an},r_{I},I,*,*,A}[t_1^{-1},...,t_I'^{-1}],B_{e,I'}	\widehat{\otimes}\Pi_{\mathrm{an},r_{I},I,*,*,A})$	
\end{center}
such that this is glueing datum along :
\begin{align}
B^+_{\mathrm{dR},I'}	\widehat{\otimes}\Pi_{\mathrm{an},r_{I},I,*,*,A}\rightarrow B^+_{\mathrm{dR},I'}	\widehat{\otimes}\Pi_{\mathrm{an},r_{I},I,*,*,A}[t_1^{-1},...,t_I'^{-1}] \leftarrow B_{e,I'}	\widehat{\otimes}\Pi_{\mathrm{an},r_{I},I,*,*,A}.	
\end{align}
And this carry the corresponding relative Galois action of:
\begin{align}
\mathrm{Gal}_{\mathbb{Q}_p,1}\times...\times \mathrm{Gal}_{\mathbb{Q}_p,I'}		
\end{align}
on the multi de Rham period rings which is semilinear. And we have that the three modules involved are relative $(\varphi_I,\Gamma_I)$-modules relative to
\begin{align}
B^+_{\mathrm{dR},I'},B_{\mathrm{dR},I'},B_{e,I'}.	
\end{align}

\end{definition}

\begin{definition}
We define a pseudocoherent $B_{I'}$-$(\varphi_I,\Gamma_I)$-module over
\begin{center}
$(B^+_{\mathrm{dR},I'}	\widehat{\otimes}\Pi_{\mathrm{an},r_{I},I,J,I\backslash J,A},B^+_{\mathrm{dR},I'}	\widehat{\otimes}\Pi_{[s_I,r_I],I,*,*,A}[t_1^{-1},...,t_I'^{-1}],B_{e,I'}	\widehat{\otimes}\Pi_{\mathrm{an},r_{I},I,*,*,A})$	
\end{center}
to be a triplet of stably pseudocoherent modules:
\begin{align}
(M_e,M_{\mathrm{dR}},M^+_{\mathrm{dR}})	
\end{align}
over 
\begin{center}
$(B^+_{\mathrm{dR},I'}	\widehat{\otimes}\Pi_{\mathrm{an},r_{I},I,*,*,A},B^+_{\mathrm{dR},I'}	\widehat{\otimes}\Pi_{\mathrm{an},r_{I},I,*,*,A}[t_1^{-1},...,t_I'^{-1}],B_{e,I'}	\widehat{\otimes}\Pi_{\mathrm{an},r_{I},I,*,*,A})$	
\end{center}
such that this is glueing datum along :
\begin{align}
B^+_{\mathrm{dR},I'}	\widehat{\otimes}\Pi_{\mathrm{an},r_{I},I,*,*,A}\rightarrow B^+_{\mathrm{dR},I'}	\widehat{\otimes}\Pi_{\mathrm{an},r_{I},I,*,*,A}[t_1^{-1},...,t_I'^{-1}] \leftarrow B_{e,I'}	\widehat{\otimes}\Pi_{\mathrm{an},r_{I},I,*,*,A}.	
\end{align}
And this carry the corresponding relative Galois action of:
\begin{align}
\mathrm{Gal}_{\mathbb{Q}_p,1}\times...\times \mathrm{Gal}_{\mathbb{Q}_p,I'}		
\end{align}
on the multi de Rham period rings which is semilinear. And we have that the three modules involved are relative pseudocoherent $(\varphi_I,\Gamma_I)$-modules relative to
\begin{align}
B^+_{\mathrm{dR},I'},B_{\mathrm{dR},I'},B_{e,I'}.	
\end{align}\\

\end{definition}

\subsection{Fundamental Comparison on the Mixed-Type Objects}

\begin{proposition} \mbox{\bf{(After Berger \cite[Th\'eor\`eme A]{Ber1})}}
Let $I'$ be a set consisting of two elements and $I$ is empty, then we have that the category of all the $B_{I'}$-$(\varphi_I,\Gamma_I)$ modules over 
\begin{center}
$(B^+_{\mathrm{dR},I'}	\widehat{\otimes}\Pi_{[s_I,r_I],I,J,I\backslash J,A},B^+_{\mathrm{dR},I'}	\widehat{\otimes}\Pi_{[s_I,r_I],I,J,I\backslash J,A}[t_1^{-1},...,t_I'^{-1}],B_{e,I'}	\widehat{\otimes}\Pi_{[s_I,r_I],I,J,I\backslash J,A})$	
\end{center}
is equivalent to the category of all the $(\varphi_{I'},\Gamma_{I'})$-modules in the finite projective setting.
\end{proposition}
\begin{proof}
One has the result after the following two propositions.	
\end{proof}

\indent We first consider the following comparison:
\begin{proposition}
Let $I'=\{1,2\}$ be a set consisting of two elements and $I$ is empty, then we have that the category of all the $B_{I'}$-$(\varphi_I,\Gamma_I)$ modules over 
\begin{center}
$(B^+_{\mathrm{dR},I'}	\widehat{\otimes}\Pi_{[s_I,r_I],I,J,I\backslash J,A},B^+_{\mathrm{dR},I'}	\widehat{\otimes}\Pi_{[s_I,r_I],I,J,I\backslash J,A}[t_1^{-1},...,t_I'^{-1}],B_{e,I'}	\widehat{\otimes}\Pi_{[s_I,r_I],I,J,I\backslash J,A})$	
\end{center}
is equivalent to the category of all the $B_{\{1\}}$-$(\varphi_{\{2\}},\Gamma_{\{2\}})$ modules.
\end{proposition}
\begin{proof}
This will be the corresponding consequence of the following. Let $I'=\{1,2\}$ be a set consisting of two elements and $I$ is empty, then we have that the category of all the $B_{I'}$-$(\varphi_I,\Gamma_I)$ modules over 
\begin{center}
$(B^+_{\mathrm{dR},I'}	\widehat{\otimes}\Pi_{[s_I,r_I],I,\widetilde{J},\widetilde{I\backslash J},A},B^+_{\mathrm{dR},I'}	\widehat{\otimes}\Pi_{[s_I,r_I],I,\widetilde{J},\widetilde{I\backslash J},A}[t_1^{-1},...,t_I'^{-1}],B_{e,I'}	\widehat{\otimes}\Pi_{[s_I,r_I],I,\widetilde{J},\widetilde{I\backslash J},A})$	
\end{center}
is equivalent to the category of all the $B_{\{1\}}$-$(\varphi_{\{2\}},\Gamma_{\{2\}})$ modules over the corresponding perfected rings with $\widetilde{.}$ accent. However this could be proved as in \cite[Theorem 2.18]{KP} as long as one works with mod $t^k,k\in \mathbb{Z}$ coefficients (also see the corresponding proof of \cite[Proposition 3.8]{T3}). To be more precise first we consider the corresponding base change of any $B_{I'}$-$(\varphi_I,\Gamma_I)$ module over 
\begin{center}
$(B^+_{\mathrm{dR},I'}	\widehat{\otimes}\Pi_{[s_I,r_I],I,\widetilde{J},\widetilde{I\backslash J},A},B^+_{\mathrm{dR},I'}	\widehat{\otimes}\Pi_{[s_I,r_I],I,\widetilde{J},\widetilde{I\backslash J},A}[t_1^{-1},...,t_I'^{-1}],B_{e,I'}	\widehat{\otimes}\Pi_{[s_I,r_I],I,\widetilde{J},\widetilde{I\backslash J},A})$	
\end{center}
to $B_\mathrm{dR,\{1\}}$, which then by the strategy above could be associated a $B_{\{1\}}$-$(\varphi_{\{2\}},\Gamma_{\{2\}})$ module over the corresponding perfected rings with $\widetilde{.}$ accent (see \cite[Theorem 2.18]{KP}, \cite[Proposition 3.8]{T3}). As in \cite[Theorem 2.18]{KP} we will have the situation where the category of all the $B_{I'}$-$(\varphi_I,\Gamma_I)$ modules over 
\begin{center}
$(B^+_{\mathrm{dR},I'}	\widehat{\otimes}\Pi_{[s_I,r_I],I,\widetilde{J},\widetilde{I\backslash J},A},B^+_{\mathrm{dR},I'}	\widehat{\otimes}\Pi_{[s_I,r_I],I,\widetilde{J},\widetilde{I\backslash J},A}[t_1^{-1},...,t_I'^{-1}],B_{e,I'}	\widehat{\otimes}\Pi_{[s_I,r_I],I,\widetilde{J},\widetilde{I\backslash J},A})$	
\end{center}
is equivalent to the category of all the $B_{\{1\}}$-$(\varphi_{\{2\}},\Gamma_{\{2\}})$ modules over the corresponding perfected rings with $\widetilde{.}$ accent. However by the proof of \cite[Theorem 4.4]{KP} we further have the situation where the category of all the $B_{I'}$-$(\varphi_I,\Gamma_I)$ modules over 
\begin{center}
$(B^+_{\mathrm{dR},I'}	\widehat{\otimes}\Pi_{[s_I,r_I],I,\widetilde{J},\widetilde{I\backslash J},A},B^+_{\mathrm{dR},I'}	\widehat{\otimes}\Pi_{[s_I,r_I],I,\widetilde{J},\widetilde{I\backslash J},A}[t_1^{-1},...,t_I'^{-1}],B_{e,I'}	\widehat{\otimes}\Pi_{[s_I,r_I],I,\widetilde{J},\widetilde{I\backslash J},A})$	
\end{center}
is equivalent to the category of all the $B_{\{1\}}$-$(\varphi_{\{2\}},\Gamma_{\{2\}})$ modules over the corresponding perfected rings with no accent. 
\end{proof}

\begin{proposition}
Let $I=\{1,2\}$ be a set consisting of two elements and $I'$ is empty, then we have that the category of all the $B_{I'}$-$(\varphi_I,\Gamma_I)$ modules over 
\begin{center}
$(B^+_{\mathrm{dR},I'}	\widehat{\otimes}\Pi_{[s_I,r_I],I,J,I\backslash J,A},B^+_{\mathrm{dR},I'}	\widehat{\otimes}\Pi_{[s_I,r_I],I,J,I\backslash J,A}[t_1^{-1},...,t_I'^{-1}],B_{e,I'}	\widehat{\otimes}\Pi_{[s_I,r_I],I,J,I\backslash J,A})$	
\end{center}
is equivalent to the category of all the $B_{\{1\}}$-$(\varphi_{\{2\}},\Gamma_{\{2\}})$ modules.

\end{proposition}

\begin{proof}
This will be the corresponding consequence of the following. Let $I=\{1,2\}$ be a set consisting of two elements and $I'$ is empty, then we have that the category of all the $B_{I'}$-$(\varphi_I,\Gamma_I)$ modules over 
\begin{center}
$(B^+_{\mathrm{dR},I'}	\widehat{\otimes}\Pi_{[s_I,r_I],I,\widetilde{J},\widetilde{I\backslash J},A},B^+_{\mathrm{dR},I'}	\widehat{\otimes}\Pi_{[s_I,r_I],I,\widetilde{J},\widetilde{I\backslash J},A}[t_1^{-1},...,t_I'^{-1}],B_{e,I'}	\widehat{\otimes}\Pi_{[s_I,r_I],I,\widetilde{J},\widetilde{I\backslash J},A})$	
\end{center}
is equivalent to the category of all the $B_{\{1\}}$-$(\varphi_{\{2\}},\Gamma_{\{2\}})$ modules over the corresponding perfected rings with $\widetilde{.}$ accent. However this is proved in \cite[Theorem 2.18]{KP} as in the proof of the previous proposition.

\end{proof}

%
%
%
%
%

\indent Now we combine the construction in the following coherent way following \cite[Section 2]{Ber1}, \cite[Definition 2.2]{Nak1} and \cite[Definition 2.2.6]{KPX}.

\begin{definition}
We define a $B_{I'}$-$(\varphi_I,\Gamma_I)$-bundle over
\begin{center}
$(B^+_{\mathrm{dR},I'}	\widehat{\otimes}\Pi_{\mathrm{an},r_{I},I,J,I\backslash J,A},B^+_{\mathrm{dR},I'}	\widehat{\otimes}\Pi_{[s_I,r_I],I,*,*,A}[t_1^{-1},...,t_I'^{-1}],B_{e,I'}	\widehat{\otimes}\Pi_{\mathrm{an},r_{I},I,*,*,A})$	
\end{center}
to be a compatible family (with respect to the Robba rings) of triplets of finite projective modules:
\begin{align}
(M_e,M_{\mathrm{dR}},M^+_{\mathrm{dR}})	
\end{align}
over 
\begin{center}
$(B^+_{\mathrm{dR},I'}	\widehat{\otimes}\Pi_{\mathrm{an},r_{I},I,*,*,A},B^+_{\mathrm{dR},I'}	\widehat{\otimes}\Pi_{\mathrm{an},r_{I},I,*,*,A}[t_1^{-1},...,t_I'^{-1}],B_{e,I'}	\widehat{\otimes}\Pi_{\mathrm{an},r_{I},I,*,*,A})$	
\end{center}
such that this is glueing datum along :
\begin{align}
B^+_{\mathrm{dR},I'}	\widehat{\otimes}\Pi_{\mathrm{an},r_{I},I,*,*,A}\rightarrow B^+_{\mathrm{dR},I'}	\widehat{\otimes}\Pi_{\mathrm{an},r_{I},I,*,*,A}[t_1^{-1},...,t_I'^{-1}] \leftarrow B_{e,I'}	\widehat{\otimes}\Pi_{\mathrm{an},r_{I},I,*,*,A}.	
\end{align}
And this carry the corresponding relative Galois action of:
\begin{align}
\mathrm{Gal}_{\mathbb{Q}_p,1}\times...\times \mathrm{Gal}_{\mathbb{Q}_p,I'}		
\end{align}
on the multi de Rham period rings which is semilinear. And we have that the three modules involved are relative $(\varphi_I,\Gamma_I)$-bundles relative to
\begin{align}
B^+_{\mathrm{dR},I'},B_{\mathrm{dR},I'},B_{e,I'}.	
\end{align}

\end{definition}

\begin{definition}
We define a pseudocoherent $B_{I'}$-$(\varphi_I,\Gamma_I)$-bundle over
\begin{center}
$(B^+_{\mathrm{dR},I'}	\widehat{\otimes}\Pi_{\mathrm{an},r_{I},I,J,I\backslash J,A},B^+_{\mathrm{dR},I'}	\widehat{\otimes}\Pi_{[s_I,r_I],I,*,*,A}[t_1^{-1},...,t_I'^{-1}],B_{e,I'}	\widehat{\otimes}\Pi_{\mathrm{an},r_{I},I,*,*,A})$	
\end{center}
to be a compatible family of triplets of stably pseudocoherent modules:
\begin{align}
(M_e,M_{\mathrm{dR}},M^+_{\mathrm{dR}})	
\end{align}
over 
\begin{center}
$(B^+_{\mathrm{dR},I'}	\widehat{\otimes}\Pi_{\mathrm{an},r_{I},I,*,*,A},B^+_{\mathrm{dR},I'}	\widehat{\otimes}\Pi_{\mathrm{an},r_{I},I,*,*,A}[t_1^{-1},...,t_I'^{-1}],B_{e,I'}	\widehat{\otimes}\Pi_{\mathrm{an},r_{I},I,*,*,A})$	
\end{center}
such that this is glueing datum along :
\begin{align}
B^+_{\mathrm{dR},I'}	\widehat{\otimes}\Pi_{\mathrm{an},r_{I},I,*,*,A}\rightarrow B^+_{\mathrm{dR},I'}	\widehat{\otimes}\Pi_{\mathrm{an},r_{I},I,*,*,A}[t_1^{-1},...,t_I'^{-1}] \leftarrow B_{e,I'}	\widehat{\otimes}\Pi_{\mathrm{an},r_{I},I,*,*,A}.	
\end{align}
And this carry the corresponding relative Galois action of:
\begin{align}
\mathrm{Gal}_{\mathbb{Q}_p,1}\times...\times \mathrm{Gal}_{\mathbb{Q}_p,I'}		
\end{align}
on the multi de Rham period rings which is semilinear. And we have that the three modules involved are relative pseudocoherent $(\varphi_I,\Gamma_I)$-bundles relative to
\begin{align}
B^+_{\mathrm{dR},I'},B_{\mathrm{dR},I'},B_{e,I'}.	
\end{align}

\end{definition}

\begin{proposition} \mbox{\bf{(After KPX \cite[Proposition 2.2.7]{KPX})}}
The category of all the finite projective $B_{I'}$-$(\varphi_I,\Gamma_I)$-bundle over 
\begin{center}
$(B^+_{\mathrm{dR},I'}	\widehat{\otimes}\Pi_{\mathrm{an},r_{I},I,*,*,A},B^+_{\mathrm{dR},I'}	\widehat{\otimes}\Pi_{\mathrm{an},r_{I},I,*,*,A}[t_1^{-1},...,t_I'^{-1}],B_{e,I'}	\widehat{\otimes}\Pi_{\mathrm{an},r_{I},I,*,*,A})$	
\end{center}
is equivalent to the category of all the finite projective $B_{I'}$-$(\varphi_I,\Gamma_I)$-modules over 
\begin{center}
$(B^+_{\mathrm{dR},I'}	\widehat{\otimes}\Pi_{\mathrm{an},r_{I},I,*,*,A},B^+_{\mathrm{dR},I'}	\widehat{\otimes}\Pi_{\mathrm{an},r_{I},I,*,*,A}[t_1^{-1},...,t_I'^{-1}],B_{e,I'}	\widehat{\otimes}\Pi_{\mathrm{an},r_{I},I,*,*,A})$.	
\end{center}
\end{proposition}

\begin{proof}
Without considering the corresponding Galois actions for the $B_{I'}$-pair components we could prove this as in the relative situation carrying just $A$-coefficient. To be more precise, the base change gives rise to the corresponding fully faithful functor from the first category to the second one, while to show the corresponding essential surjectivity, consider the corresponding multi-interval $[r_{1,0}/p,r_{1,0}]\times...\times [r_{I,0}/p,r_{I,0}]$ and use the corresponding Frobenius to reach all the corresponding intervals taking the general form of:
\begin{align}
[r_{1,0}/p^{k_1},r_{1,0}/p^{k_1-1}]\times...\times [r_{I,0}/p^{k_I},r_{I,0}/p^{k_I-1}],k_\alpha=1,2,...,\forall\alpha\in I.	
\end{align}
This forms a $2^{|I|}$-uniform covering of the whole space. And the corresponding uniform finiteness of the modules over each 
\begin{align}
[r_{1,0}/p^{k_1},r_{1,0}/p^{k_1-1}]\times...\times [r_{I,0}/p^{k_I},r_{I,0}/p^{k_I-1}],k_\alpha=1,2,...,\forall\alpha\in I.	
\end{align}	
could be achieved by using the corresponding partial Frobenius actions. Then we are done by applying \cref{proposition2.19}. 	
\end{proof}

\begin{proposition} \mbox{\bf{(After KPX \cite[Proposition 2.2.7]{KPX})}}
The category of all the pseudocoherent $B_{I'}$-$(\varphi_I,\Gamma_I)$-bundle over 
\begin{center}
$(B^+_{\mathrm{dR},I'}	\widehat{\otimes}\Pi_{\mathrm{an},r_{I},I,*,*,A},B^+_{\mathrm{dR},I'}	\widehat{\otimes}\Pi_{\mathrm{an},r_{I},I,*,*,A}[t_1^{-1},...,t_I'^{-1}],B_{e,I'}	\widehat{\otimes}\Pi_{\mathrm{an},r_{I},I,*,*,A})$	
\end{center}
is equivalent to the category of all the pseudocoherent $B_{I'}$-$(\varphi_I,\Gamma_I)$-modules over 
\begin{center}
$(B^+_{\mathrm{dR},I'}	\widehat{\otimes}\Pi_{\mathrm{an},r_{I},I,*,*,A},B^+_{\mathrm{dR},I'}	\widehat{\otimes}\Pi_{\mathrm{an},r_{I},I,*,*,A}[t_1^{-1},...,t_I'^{-1}],B_{e,I'}	\widehat{\otimes}\Pi_{\mathrm{an},r_{I},I,*,*,A})$.	
\end{center}
\end{proposition}

\begin{proof}
Without considering the corresponding Galois actions for the $B_{I'}$-pair components we could prove this as in the relative situation carrying just $A$-coefficient. To be more precise, the base change gives rise to the corresponding fully faithful functor from the first category to the second one, while to show the corresponding essential surjectivity, consider the corresponding multi-interval $[r_{1,0}/p,r_{1,0}]\times...\times [r_{I,0}/p,r_{I,0}]$ and use the corresponding Frobenius to reach all the corresponding intervals taking the general form of:
\begin{align}
[r_{1,0}/p^{k_1},r_{1,0}/p^{k_1-1}]\times...\times [r_{I,0}/p^{k_I},r_{I,0}/p^{k_I-1}],k_\alpha=1,2,...,\forall\alpha\in I.	
\end{align}
This forms a $2^{|I|}$-uniform covering of the whole space. And the corresponding uniform finiteness of the modules over each 
\begin{align}
[r_{1,0}/p^{k_1},r_{1,0}/p^{k_1-1}]\times...\times [r_{I,0}/p^{k_I},r_{I,0}/p^{k_I-1}],k_\alpha=1,2,...,\forall\alpha\in I.	
\end{align}	
could be achieved by using the corresponding partial Frobenius actions. Then we are done by applying \cref{proposition2.18}. 	
\end{proof}

\newpage

\section{Cohomologies of Cyclotomic Multivariate $(\varphi_I,\Gamma_I)$-Modules over Rigid Analytic Affinoids in Mixed-characteristic Case}

\noindent Now we define the corresponding cohomologies of the multivariate $(\varphi_I,\Gamma_I)$-modules over the following groups of rings:

\begin{align}
\Pi_{\mathrm{an},r_I,I,\breve{J},I\backslash J,A}(\pi_{K_I}):=\varprojlim_{s_I}\Pi_{[s_I,r_I],I,\breve{J},I\backslash J,A}(\pi_{K_I}),\\	
\Pi_{\mathrm{an},r_I,I,\widetilde{J},I\backslash J,A}(\pi_{K_I}):=\varprojlim_{s_I} \Pi_{[s_I,r_I],I,\widetilde{J},I\backslash J,A}(\pi_{K_I}),\\
\Pi_{\mathrm{an},r_I,I,J,\breve{I\backslash J},A}(\pi_{K_I}):=\varprojlim_{s_I}\Pi_{[s_I,r_I],I,J,\breve{I\backslash J},A}(\pi_{K_I}),\\	
\Pi_{\mathrm{an},r_I,I,\breve{J},\breve{I\backslash J},A}(\pi_{K_I}):=\varprojlim_{s_I} \Pi_{[s_I,r_I],I,\breve{J},\breve{I\backslash J},A}(\pi_{K_I}),\\	
\Pi_{\mathrm{an},r_I,I,\widetilde{J},\breve{I\backslash J},A}(\pi_{K_I}):=\varprojlim_{s_I} \Pi_{[s_I,r_I],I,\widetilde{J},\breve{I\backslash J},A}(\pi_{K_I}),\\
\Pi_{\mathrm{an},r_I,I,J,\widetilde{I\backslash J},A}(\pi_{K_I}):=\varprojlim_{s_I} \Pi_{[s_I,r_I],I,J,\widetilde{I\backslash J},A}(\pi_{K_I}),\\	
\Pi_{\mathrm{an},r_I,I,\breve{J},\widetilde{I\backslash J},A}(\pi_{K_I}):=\varprojlim_{s_I} \Pi_{[s_I,r_I],I,\breve{J},\widetilde{I\backslash J},A}(\pi_{K_I}),\\	
\Pi_{\mathrm{an},r_I,I,\widetilde{J},\widetilde{I\backslash J},A}(\pi_{K_I}):=\varprojlim_{s_I} \Pi_{[s_I,r_I],I,\widetilde{J},\widetilde{I\backslash J},A}(\pi_{K_I}).	
\end{align}

and

\begin{align}
\Pi_{[s_I,r_I],I,\breve{J},I\backslash J,A}(\Gamma_{K_I}),\\	
\Pi_{[s_I,r_I],I,\widetilde{J},I\backslash J,A}(\Gamma_{K_I}),\\
\Pi_{[s_I,r_I],I,J,\breve{I\backslash J},A}(\Gamma_{K_I}),\\	
\Pi_{[s_I,r_I],I,\breve{J},\breve{I\backslash J},A}(\Gamma_{K_I}),\\
\Pi_{[s_I,r_I],I,\widetilde{J},\breve{I\backslash J},A}(\Gamma_{K_I}),\\
\Pi_{[s_I,r_I],I,J,\widetilde{I\backslash J},A}(\Gamma_{K_I}),\\	
\Pi_{[s_I,r_I],I,\breve{J},\widetilde{I\backslash J},A}(\Gamma_{K_I}),\\	
\Pi_{[s_I,r_I],I,\widetilde{J},\widetilde{I\backslash J},A}(\Gamma_{K_I}).	
\end{align}

\begin{definition} \mbox{\bf{(After KPX, \cite[Definition 2.3.3]{KPX})}}
We define by induction the corresponding $\varphi_I$-complex $C^\bullet_{\varphi_I}$ of a corresponding $(\varphi_I,\Gamma_I)$-module $M$ over 
\begin{align}
\Pi_{\mathrm{an},r_I,I,\breve{J},I\backslash J,A}(\pi_{K_I}):=\varprojlim_{s_I}\Pi_{[s_I,r_I],I,\breve{J},I\backslash J,A}(\pi_{K_I}),\\	
\Pi_{\mathrm{an},r_I,I,\widetilde{J},I\backslash J,A}(\pi_{K_I}):=\varprojlim_{s_I} \Pi_{[s_I,r_I],I,\widetilde{J},I\backslash J,A}(\pi_{K_I}),\\
\Pi_{\mathrm{an},r_I,I,J,\breve{I\backslash J},A}(\pi_{K_I}):=\varprojlim_{s_I}\Pi_{[s_I,r_I],I,J,\breve{I\backslash J},A}(\pi_{K_I}),\\	
\Pi_{\mathrm{an},r_I,I,\breve{J},\breve{I\backslash J},A}(\pi_{K_I}):=\varprojlim_{s_I} \Pi_{[s_I,r_I],I,\breve{J},\breve{I\backslash J},A}(\pi_{K_I}),\\	
\Pi_{\mathrm{an},r_I,I,\widetilde{J},\breve{I\backslash J},A}(\pi_{K_I}):=\varprojlim_{s_I} \Pi_{[s_I,r_I],I,\widetilde{J},\breve{I\backslash J},A}(\pi_{K_I}),\\
\Pi_{\mathrm{an},r_I,I,J,\widetilde{I\backslash J},A}(\pi_{K_I}):=\varprojlim_{s_I} \Pi_{[s_I,r_I],I,J,\widetilde{I\backslash J},A}(\pi_{K_I}),\\	
\Pi_{\mathrm{an},r_I,I,\breve{J},\widetilde{I\backslash J},A}(\pi_{K_I}):=\varprojlim_{s_I} \Pi_{[s_I,r_I],I,\breve{J},\widetilde{I\backslash J},A}(\pi_{K_I}),\\	
\Pi_{\mathrm{an},r_I,I,\widetilde{J},\widetilde{I\backslash J},A}(\pi_{K_I}):=\varprojlim_{s_I} \Pi_{[s_I,r_I],I,\widetilde{J},\widetilde{I\backslash J},A}(\pi_{K_I})	
\end{align}
to be the corresponding totalization of the following complex:
\[
\xymatrix@C+0pc@R+0pc{
0 \ar[r] \ar[r] \ar[r] &C^\bullet_{\varphi_{I\backslash |I|}}(M_{...,r_{|I|}}) \ar[r]^{\varphi_{|I|}-1}\ar[r]\ar[r] &C^\bullet_{\varphi_{I\backslash |I|}}(M_{...,r_{|I|}/p}) \ar[r] \ar[r] \ar[r] &0
}
\]
as long as $C^\bullet_{\varphi_{I\backslash |I|}}(M)$ is constructed. We define by induction the corresponding $\varphi_I$-complex $C^\bullet_{\varphi_I}(M)$ of a corresponding $(\varphi_I,\Gamma_I)$-module $M$ over
\begin{align}
\Pi_{[s_I,r_I],I,\breve{J},I\backslash J,A}(\pi_{K_I}),\\	
\Pi_{[s_I,r_I],I,\widetilde{J},I\backslash J,A}(\pi_{K_I}),\\
\Pi_{[s_I,r_I],I,J,\breve{I\backslash J},A}(\pi_{K_I}),\\	
\Pi_{[s_I,r_I],I,\breve{J},\breve{I\backslash J},A}(\pi_{K_I}),\\
\Pi_{[s_I,r_I],I,\widetilde{J},\breve{I\backslash J},A}(\pi_{K_I}),\\
\Pi_{[s_I,r_I],I,J,\widetilde{I\backslash J},A}(\pi_{K_I}),\\	
\Pi_{[s_I,r_I],I,\breve{J},\widetilde{I\backslash J},A}(\pi_{K_I}),\\
\Pi_{[s_I,r_I],I,\widetilde{J},\widetilde{I\backslash J},A}(\pi_{K_I})	
\end{align}	
to be the corresponding totalization of the following complex:
\[
\xymatrix@C+0pc@R+0pc{
0 \ar[r] \ar[r] \ar[r] &C^\bullet_{\varphi_{I\backslash |I|}}(M_{...,[s_{|I|},r_{|I|}]}) \ar[r]^{\varphi_{|I|}-1}\ar[r]\ar[r] &C^\bullet_{\varphi_{I\backslash |I|}}(M_{...,[s_{|I|},r_{|I|}/p]}) \ar[r] \ar[r] \ar[r] &0
}
\]
as long as $C^\bullet_{\varphi_{I\backslash |I|}}(M)$ is constructed. Here we assume $0< s_\alpha \leq r_\alpha/p$ for each $\alpha\in I$.	
\end{definition}

\begin{definition} \mbox{\bf{(After KPX, \cite[Definition 2.3.3]{KPX})}}
We define by induction the corresponding $\psi_I$-complex $C^\bullet_{\psi_I}$ of a corresponding $(\varphi_I,\Gamma_I)$-module $M$ over 
\begin{align}
\Pi_{\mathrm{an},r_I,I,\breve{J},I\backslash J,A}(\pi_{K_I}):=\varprojlim_{s_I}\Pi_{[s_I,r_I],I,\breve{J},I\backslash J,A}(\pi_{K_I}),\\	
\Pi_{\mathrm{an},r_I,I,\widetilde{J},I\backslash J,A}(\pi_{K_I}):=\varprojlim_{s_I} \Pi_{[s_I,r_I],I,\widetilde{J},I\backslash J,A}(\pi_{K_I}),\\
\Pi_{\mathrm{an},r_I,I,J,\breve{I\backslash J},A}(\pi_{K_I}):=\varprojlim_{s_I}\Pi_{[s_I,r_I],I,J,\breve{I\backslash J},A}(\pi_{K_I}),\\	
\Pi_{\mathrm{an},r_I,I,\breve{J},\breve{I\backslash J},A}(\pi_{K_I}):=\varprojlim_{s_I} \Pi_{[s_I,r_I],I,\breve{J},\breve{I\backslash J},A}(\pi_{K_I}),\\	
\Pi_{\mathrm{an},r_I,I,\widetilde{J},\breve{I\backslash J},A}(\pi_{K_I}):=\varprojlim_{s_I} \Pi_{[s_I,r_I],I,\widetilde{J},\breve{I\backslash J},A}(\pi_{K_I}),\\
\Pi_{\mathrm{an},r_I,I,J,\widetilde{I\backslash J},A}(\pi_{K_I}):=\varprojlim_{s_I} \Pi_{[s_I,r_I],I,J,\widetilde{I\backslash J},A}(\pi_{K_I}),\\	
\Pi_{\mathrm{an},r_I,I,\breve{J},\widetilde{I\backslash J},A}(\pi_{K_I}):=\varprojlim_{s_I} \Pi_{[s_I,r_I],I,\breve{J},\widetilde{I\backslash J},A}(\pi_{K_I}),\\	
\Pi_{\mathrm{an},r_I,I,\widetilde{J},\widetilde{I\backslash J},A}(\pi_{K_I}):=\varprojlim_{s_I} \Pi_{[s_I,r_I],I,\widetilde{J},\widetilde{I\backslash J},A}(\pi_{K_I})	
\end{align}
to be the corresponding totalization of the following complex:
\[
\xymatrix@C+0pc@R+0pc{
0 \ar[r] \ar[r] \ar[r] &C^\bullet_{\psi_{I\backslash |I|}}(M_{...,r_{|I|}}) \ar[r]^{\psi_{|I|}-1}\ar[r]\ar[r] &C^\bullet_{\psi_{I\backslash |I|}}(M_{...,pr_{|I|}}) \ar[r] \ar[r] \ar[r] &0
}
\]
as long as $C^\bullet_{\psi_{I\backslash |I|}}(M)$ is constructed. We define by induction the corresponding $\psi_I$-complex $C^\bullet_{\psi_I}(M)$ of a corresponding $(\varphi_I,\Gamma_I)$-module $M$ over
\begin{align}
\Pi_{[s_I,r_I],I,\breve{J},I\backslash J,A}(\pi_{K_I}),\\	
\Pi_{[s_I,r_I],I,\widetilde{J},I\backslash J,A}(\pi_{K_I}),\\
\Pi_{[s_I,r_I],I,J,\breve{I\backslash J},A}(\pi_{K_I}),\\	
\Pi_{[s_I,r_I],I,\breve{J},\breve{I\backslash J},A}(\pi_{K_I}),\\
\Pi_{[s_I,r_I],I,\widetilde{J},\breve{I\backslash J},A}(\pi_{K_I}),\\
\Pi_{[s_I,r_I],I,J,\widetilde{I\backslash J},A}(\pi_{K_I}),\\	
\Pi_{[s_I,r_I],I,\breve{J},\widetilde{I\backslash J},A}(\pi_{K_I}),\\
\Pi_{[s_I,r_I],I,\widetilde{J},\widetilde{I\backslash J},A}(\pi_{K_I})	
\end{align}	
to be the corresponding totalization of the following complex:
\[
\xymatrix@C+0pc@R+0pc{
0 \ar[r] \ar[r] \ar[r] &C^\bullet_{\psi_{I\backslash |I|}}(M_{...,[s_{|I|},r_{|I|}]}) \ar[r]^{\psi_{|I|}-1}\ar[r]\ar[r] &C^\bullet_{\psi_{I\backslash |I|}}(M_{...,[ps_{|I|},r_{|I|}]}) \ar[r] \ar[r] \ar[r] &0
}
\]
as long as $C^\bullet_{\psi_{I\backslash |I|}}(M)$ is constructed. Here we assume $0< s_\alpha \leq r_\alpha/p$ for each $\alpha\in I$.	
\end{definition}

\begin{definition} \mbox{\bf{(After KPX, \cite[Definition 2.3.3]{KPX})}}
We define by induction the corresponding $\Gamma_I$-complex $C^\bullet_{\Gamma_I}$ of a corresponding $(\varphi_I,\Gamma_I)$-module $M$ over 
\begin{align}
\Pi_{\mathrm{an},r_I,I,\breve{J},I\backslash J,A}(\pi_{K_I}):=\varprojlim_{s_I}\Pi_{[s_I,r_I],I,\breve{J},I\backslash J,A}(\pi_{K_I}),\\	
\Pi_{\mathrm{an},r_I,I,\widetilde{J},I\backslash J,A}(\pi_{K_I}):=\varprojlim_{s_I} \Pi_{[s_I,r_I],I,\widetilde{J},I\backslash J,A}(\pi_{K_I}),\\
\Pi_{\mathrm{an},r_I,I,J,\breve{I\backslash J},A}(\pi_{K_I}):=\varprojlim_{s_I}\Pi_{[s_I,r_I],I,J,\breve{I\backslash J},A}(\pi_{K_I}),\\	
\Pi_{\mathrm{an},r_I,I,\breve{J},\breve{I\backslash J},A}(\pi_{K_I}):=\varprojlim_{s_I} \Pi_{[s_I,r_I],I,\breve{J},\breve{I\backslash J},A}(\pi_{K_I}),\\	
\Pi_{\mathrm{an},r_I,I,\widetilde{J},\breve{I\backslash J},A}(\pi_{K_I}):=\varprojlim_{s_I} \Pi_{[s_I,r_I],I,\widetilde{J},\breve{I\backslash J},A}(\pi_{K_I}),\\
\Pi_{\mathrm{an},r_I,I,J,\widetilde{I\backslash J},A}(\pi_{K_I}):=\varprojlim_{s_I} \Pi_{[s_I,r_I],I,J,\widetilde{I\backslash J},A}(\pi_{K_I}),\\	
\Pi_{\mathrm{an},r_I,I,\breve{J},\widetilde{I\backslash J},A}(\pi_{K_I}):=\varprojlim_{s_I} \Pi_{[s_I,r_I],I,\breve{J},\widetilde{I\backslash J},A}(\pi_{K_I}),\\	
\Pi_{\mathrm{an},r_I,I,\widetilde{J},\widetilde{I\backslash J},A}(\pi_{K_I}):=\varprojlim_{s_I} \Pi_{[s_I,r_I],I,\widetilde{J},\widetilde{I\backslash J},A}(\pi_{K_I})	
\end{align}
to be the corresponding totalization of the following complex:
\[
\xymatrix@C+0pc@R+0pc{
0 \ar[r] \ar[r] \ar[r] &C^\bullet_{\Gamma_{I\backslash |I|}}(M) \ar[r]^{\gamma_{|I|}-1}\ar[r]\ar[r] &C^\bullet_{\Gamma_{I\backslash |I|}}(M) \ar[r] \ar[r] \ar[r] &0
}
\]
as long as $C^\bullet_{\Gamma_{I\backslash |I|}}(M)$ is constructed. We define by induction the corresponding $\Gamma_I$-complex $C^\bullet_{\Gamma_I}(M)$ of a corresponding $(\varphi_I,\Gamma_I)$-module $M$ over
\begin{align}
\Pi_{[s_I,r_I],I,\breve{J},I\backslash J,A}(\pi_{K_I}),\\	
\Pi_{[s_I,r_I],I,\widetilde{J},I\backslash J,A}(\pi_{K_I}),\\
\Pi_{[s_I,r_I],I,J,\breve{I\backslash J},A}(\pi_{K_I}),\\	
\Pi_{[s_I,r_I],I,\breve{J},\breve{I\backslash J},A}(\pi_{K_I}),\\
\Pi_{[s_I,r_I],I,\widetilde{J},\breve{I\backslash J},A}(\pi_{K_I}),\\
\Pi_{[s_I,r_I],I,J,\widetilde{I\backslash J},A}(\pi_{K_I}),\\	
\Pi_{[s_I,r_I],I,\breve{J},\widetilde{I\backslash J},A}(\pi_{K_I}),\\
\Pi_{[s_I,r_I],I,\widetilde{J},\widetilde{I\backslash J},A}(\pi_{K_I})	
\end{align}	
to be the corresponding totalization of the following complex:
\[
\xymatrix@C+0pc@R+0pc{
0 \ar[r] \ar[r] \ar[r] &C^\bullet_{\Gamma_{I\backslash |I|}}(M) \ar[r]^{\gamma_{|I|}-1}\ar[r]\ar[r] &C^\bullet_{\Gamma_{I\backslash |I|}}(M) \ar[r] \ar[r] \ar[r] &0
}
\]
as long as $C^\bullet_{\Gamma_{I\backslash |I|}}(M)$ is constructed. Here we assume $0< s_\alpha \leq r_\alpha/p$ for each $\alpha\in I$.	
\end{definition}

\begin{definition} \mbox{\bf{(After KPX, \cite[Definition 2.3.3]{KPX})}}
For any $(\varphi_I,\Gamma_I)$-module $M$ over 
\begin{align}
\Pi_{\mathrm{an},r_I,I,\breve{J},I\backslash J,A}(\pi_{K_I}):=\varprojlim_{s_I}\Pi_{[s_I,r_I],I,\breve{J},I\backslash J,A}(\pi_{K_I}),\\	
\Pi_{\mathrm{an},r_I,I,\widetilde{J},I\backslash J,A}(\pi_{K_I}):=\varprojlim_{s_I} \Pi_{[s_I,r_I],I,\widetilde{J},I\backslash J,A}(\pi_{K_I}),\\
\Pi_{\mathrm{an},r_I,I,J,\breve{I\backslash J},A}(\pi_{K_I}):=\varprojlim_{s_I}\Pi_{[s_I,r_I],I,J,\breve{I\backslash J},A}(\pi_{K_I}),\\	
\Pi_{\mathrm{an},r_I,I,\breve{J},\breve{I\backslash J},A}(\pi_{K_I}):=\varprojlim_{s_I} \Pi_{[s_I,r_I],I,\breve{J},\breve{I\backslash J},A}(\pi_{K_I}),\\	
\Pi_{\mathrm{an},r_I,I,\widetilde{J},\breve{I\backslash J},A}(\pi_{K_I}):=\varprojlim_{s_I} \Pi_{[s_I,r_I],I,\widetilde{J},\breve{I\backslash J},A}(\pi_{K_I}),\\
\Pi_{\mathrm{an},r_I,I,J,\widetilde{I\backslash J},A}(\pi_{K_I}):=\varprojlim_{s_I} \Pi_{[s_I,r_I],I,J,\widetilde{I\backslash J},A}(\pi_{K_I}),\\	
\Pi_{\mathrm{an},r_I,I,\breve{J},\widetilde{I\backslash J},A}(\pi_{K_I}):=\varprojlim_{s_I} \Pi_{[s_I,r_I],I,\breve{J},\widetilde{I\backslash J},A}(\pi_{K_I}),\\	
\Pi_{\mathrm{an},r_I,I,\widetilde{J},\widetilde{I\backslash J},A}(\pi_{K_I}):=\varprojlim_{s_I} \Pi_{[s_I,r_I],I,\widetilde{J},\widetilde{I\backslash J},A}(\pi_{K_I})	
\end{align} 
we define the corresponding complex $C^\bullet_{\varphi_I,\Gamma_I}(M)$ to be the corresponding totalization of $C^\bullet_{\varphi_I}C^\bullet_{\Gamma_I}(M)$.
For any $(\varphi_I,\Gamma_I)$-module $M$ over 
\begin{align}
\Pi_{[s_I,r_I],I,\breve{J},I\backslash J,A}(\pi_{K_I}),\\	
\Pi_{[s_I,r_I],I,\widetilde{J},I\backslash J,A}(\pi_{K_I}),\\
\Pi_{[s_I,r_I],I,J,\breve{I\backslash J},A}(\pi_{K_I}),\\	
\Pi_{[s_I,r_I],I,\breve{J},\breve{I\backslash J},A}(\pi_{K_I}),\\
\Pi_{[s_I,r_I],I,\widetilde{J},\breve{I\backslash J},A}(\pi_{K_I}),\\
\Pi_{[s_I,r_I],I,J,\widetilde{I\backslash J},A}(\pi_{K_I}),\\	
\Pi_{[s_I,r_I],I,\breve{J},\widetilde{I\backslash J},A}(\pi_{K_I}),\\
\Pi_{[s_I,r_I],I,\widetilde{J},\widetilde{I\backslash J},A}(\pi_{K_I})	
\end{align}	
we define the corresponding complex $C^\bullet_{\varphi_I,\Gamma_I}(M)$ to be the corresponding totalization of $C^\bullet_{\varphi_I}C^\bullet_{\Gamma_I}(M)$. Here we assume $0< s_\alpha \leq r_\alpha/p$ for each $\alpha\in I$.	
	
\end{definition}


\begin{definition} \mbox{\bf{(After KPX, \cite[Definition 2.3.3]{KPX})}}
For any $(\psi_I,\Gamma_I)$-module $M$ over 
\begin{align}
\Pi_{\mathrm{an},r_I,I,\breve{J},I\backslash J,A}(\pi_{K_I}):=\varprojlim_{s_I}\Pi_{[s_I,r_I],I,\breve{J},I\backslash J,A}(\pi_{K_I}),\\	
\Pi_{\mathrm{an},r_I,I,\widetilde{J},I\backslash J,A}(\pi_{K_I}):=\varprojlim_{s_I} \Pi_{[s_I,r_I],I,\widetilde{J},I\backslash J,A}(\pi_{K_I}),\\
\Pi_{\mathrm{an},r_I,I,J,\breve{I\backslash J},A}(\pi_{K_I}):=\varprojlim_{s_I}\Pi_{[s_I,r_I],I,J,\breve{I\backslash J},A}(\pi_{K_I}),\\	
\Pi_{\mathrm{an},r_I,I,\breve{J},\breve{I\backslash J},A}(\pi_{K_I}):=\varprojlim_{s_I} \Pi_{[s_I,r_I],I,\breve{J},\breve{I\backslash J},A}(\pi_{K_I}),\\	
\Pi_{\mathrm{an},r_I,I,\widetilde{J},\breve{I\backslash J},A}(\pi_{K_I}):=\varprojlim_{s_I} \Pi_{[s_I,r_I],I,\widetilde{J},\breve{I\backslash J},A}(\pi_{K_I}),\\
\Pi_{\mathrm{an},r_I,I,J,\widetilde{I\backslash J},A}(\pi_{K_I}):=\varprojlim_{s_I} \Pi_{[s_I,r_I],I,J,\widetilde{I\backslash J},A}(\pi_{K_I}),\\	
\Pi_{\mathrm{an},r_I,I,\breve{J},\widetilde{I\backslash J},A}(\pi_{K_I}):=\varprojlim_{s_I} \Pi_{[s_I,r_I],I,\breve{J},\widetilde{I\backslash J},A}(\pi_{K_I}),\\	
\Pi_{\mathrm{an},r_I,I,\widetilde{J},\widetilde{I\backslash J},A}(\pi_{K_I}):=\varprojlim_{s_I} \Pi_{[s_I,r_I],I,\widetilde{J},\widetilde{I\backslash J},A}(\pi_{K_I})	
\end{align} 
we define the corresponding complex $C^\bullet_{\psi_I,\Gamma_I}(M)$ to be the corresponding totalization of $C^\bullet_{\psi_I}C^\bullet_{\Gamma_I}(M)$.
For any $(\varphi_I,\Gamma_I)$-module $M$ over 
\begin{align}
\Pi_{[s_I,r_I],I,\breve{J},I\backslash J,A}(\pi_{K_I}),\\	
\Pi_{[s_I,r_I],I,\widetilde{J},I\backslash J,A}(\pi_{K_I}),\\
\Pi_{[s_I,r_I],I,J,\breve{I\backslash J},A}(\pi_{K_I}),\\	
\Pi_{[s_I,r_I],I,\breve{J},\breve{I\backslash J},A}(\pi_{K_I}),\\
\Pi_{[s_I,r_I],I,\widetilde{J},\breve{I\backslash J},A}(\pi_{K_I}),\\
\Pi_{[s_I,r_I],I,J,\widetilde{I\backslash J},A}(\pi_{K_I}),\\	
\Pi_{[s_I,r_I],I,\breve{J},\widetilde{I\backslash J},A}(\pi_{K_I}),\\
\Pi_{[s_I,r_I],I,\widetilde{J},\widetilde{I\backslash J},A}(\pi_{K_I})	
\end{align}	
we define the corresponding complex $C^\bullet_{\psi_I,\Gamma_I}(M)$ to be the corresponding totalization of $C^\bullet_{\psi_I}C^\bullet_{\Gamma_I}(M)$. Here we assume $0< s_\alpha \leq r_\alpha/p$ for each $\alpha\in I$.	
	
\end{definition}

\newpage

\section{Cohomologies of $B_I$-pairs and Mixed-Type Objects over Rigid Analytic Affinoids in Mixed-characteristic Case}

\subsection{Partial $(\varphi_I,\Gamma_I)$-Cohomology and Partial $(\psi_I,\Gamma_I)$-Cohomology}

\noindent Now we define the corresponding cohomologies of the multivariate $(\varphi_I,\Gamma_I)$-modules over the following two groups of rings:

\begin{align}
\varprojlim_{s_I}  B^+_{\mathrm{dR},I'}	\widehat{\otimes}\Pi_{[s_I,r_I],I,?,?',A},?=J,\widetilde{J},\breve{J},?'=I\backslash J,\widetilde{I\backslash J},\breve{I\backslash J}\\
\end{align}

with

\begin{align}
\varprojlim_{s_I} B^+_{\mathrm{dR},I'}	\widehat{\otimes}\Pi_{[s_I,r_I],I,?,?',A}[t_1^{-1},...,t_{I'}^{-1}],?=J,\widetilde{J},\breve{J},?'=I\backslash J,\widetilde{I\backslash J},\breve{I\backslash J},\\
\end{align}

with 

\begin{align}
\varprojlim_{s_I} B_{e,I'}	\widehat{\otimes}\Pi_{[s_I,r_I],I,?,?',A}, ?=J,\widetilde{J},\breve{J},?'=I\backslash J,\widetilde{I\backslash J},\breve{I\backslash J},\\
\end{align}

and

\begin{align}
B^+_{\mathrm{dR},I'}	\widehat{\otimes}\Pi_{[s_I,r_I],I,?,?',A},?=J,\widetilde{J},\breve{J},?'=I\backslash J,\widetilde{I\backslash J},\breve{I\backslash J}
\end{align}

with

\begin{align}
B^+_{\mathrm{dR},I'}	\widehat{\otimes}\Pi_{[s_I,r_I],I,?,?',A}[t_1^{-1},...,t_{I'}^{-1}],?=J,\widetilde{J},\breve{J},?'=I\backslash J,\widetilde{I\backslash J},\breve{I\backslash J}
\end{align}

with 

\begin{align}
B_{e,I'}	\widehat{\otimes}\Pi_{[s_I,r_I],I,?,?',A},?=J,\widetilde{J},\breve{J},?'=I\backslash J,\widetilde{I\backslash J},\breve{I\backslash J}.
\end{align}

\begin{definition} \mbox{\bf{(After KPX, \cite[Definition 2.3.3]{KPX})}}
We define by induction the corresponding $\varphi_I$-complex $C^\bullet_{\varphi_I}$ of a corresponding $B_{I'}$-$(\varphi_I,\Gamma_I)$-module $M$ over the first three groups of rings
to be the corresponding totalization of the following complex:
\[
\xymatrix@C+0pc@R+0pc{
0 \ar[r] \ar[r] \ar[r] &C^\bullet_{\varphi_{I\backslash |I|}}(M_{...,r_{|I|}}) \ar[r]^{\varphi_{|I|}-1}\ar[r]\ar[r] &C^\bullet_{\varphi_{I\backslash |I|}}(M_{...,r_{|I|}/p}) \ar[r] \ar[r] \ar[r] &0
}
\]
as long as $C^\bullet_{\varphi_{I\backslash |I|}}(M)$ is constructed. We define by induction the corresponding $\varphi_I$-complex $C^\bullet_{\varphi_I}(M)$ of a corresponding $B_{I'}$-$(\varphi_I,\Gamma_I)$-module $M$ over the second three groups of rings to be the corresponding totalization of the following complex:
\[
\xymatrix@C+0pc@R+0pc{
0 \ar[r] \ar[r] \ar[r] &C^\bullet_{\varphi_{I\backslash |I|}}(M_{...,[s_{|I|},r_{|I|}]}) \ar[r]^{\varphi_{|I|}-1}\ar[r]\ar[r] &C^\bullet_{\varphi_{I\backslash |I|}}(M_{...,[s_{|I|},r_{|I|}/p]}) \ar[r] \ar[r] \ar[r] &0
}
\]
as long as $C^\bullet_{\varphi_{I\backslash |I|}}(M)$ is constructed. Here we assume $0< s_\alpha \leq r_\alpha/p$ for each $\alpha\in I$.	
\end{definition}

\begin{definition} \mbox{\bf{(After KPX, \cite[Definition 2.3.3]{KPX})}}
We define by induction the corresponding $\psi_I$-complex $C^\bullet_{\psi_I}$ of a corresponding $B_{I'}$-$(\varphi_I,\Gamma_I)$-module $M$ over the first three groups of rings 
to be the corresponding totalization of the following complex:
\[
\xymatrix@C+0pc@R+0pc{
0 \ar[r] \ar[r] \ar[r] &C^\bullet_{\psi_{I\backslash |I|}}(M_{...,r_{|I|}}) \ar[r]^{\psi_{|I|}-1}\ar[r]\ar[r] &C^\bullet_{\psi_{I\backslash |I|}}(M_{...,pr_{|I|}}) \ar[r] \ar[r] \ar[r] &0
}
\]
as long as $C^\bullet_{\psi_{I\backslash |I|}}(M)$ is constructed. We define by induction the corresponding $\psi_I$-complex $C^\bullet_{\psi_I}(M)$ of a corresponding $B_{I'}$-$(\varphi_I,\Gamma_I)$-module $M$ over the second three groups of rings
to be the corresponding totalization of the following complex:
\[
\xymatrix@C+0pc@R+0pc{
0 \ar[r] \ar[r] \ar[r] &C^\bullet_{\psi_{I\backslash |I|}}(M_{...,[s_{|I|},r_{|I|}]}) \ar[r]^{\psi_{|I|}-1}\ar[r]\ar[r] &C^\bullet_{\psi_{I\backslash |I|}}(M_{...,[ps_{|I|},r_{|I|}]}) \ar[r] \ar[r] \ar[r] &0
}
\]
as long as $C^\bullet_{\psi_{I\backslash |I|}}(M)$ is constructed. Here we assume $0< s_\alpha \leq r_\alpha/p$ for each $\alpha\in I$.	
\end{definition}

\begin{definition} \mbox{\bf{(After KPX, \cite[Definition 2.3.3]{KPX})}}
We define by induction the corresponding $\Gamma_I$-complex $C^\bullet_{\Gamma_I}$ of a corresponding $B_{I'}$-$(\varphi_I,\Gamma_I)$-module $M$ over 
the first three groups of rings
to be the corresponding totalization of the following complex:
\[
\xymatrix@C+0pc@R+0pc{
0 \ar[r] \ar[r] \ar[r] &C^\bullet_{\Gamma_{I\backslash |I|}}(M) \ar[r]^{\gamma_{|I|}-1}\ar[r]\ar[r] &C^\bullet_{\Gamma_{I\backslash |I|}}(M) \ar[r] \ar[r] \ar[r] &0
}
\]
as long as $C^\bullet_{\Gamma_{I\backslash |I|}}(M)$ is constructed. We define by induction the corresponding $\Gamma_I$-complex $C^\bullet_{\Gamma_I}(M)$ of a corresponding $B_{I'}$-$(\varphi_I,\Gamma_I)$-module $M$ over
the second three groups of rings 	
to be the corresponding totalization of the following complex:
\[
\xymatrix@C+0pc@R+0pc{
0 \ar[r] \ar[r] \ar[r] &C^\bullet_{\Gamma_{I\backslash |I|}}(M) \ar[r]^{\gamma_{|I|}-1}\ar[r]\ar[r] &C^\bullet_{\Gamma_{I\backslash |I|}}(M) \ar[r] \ar[r] \ar[r] &0
}
\]
as long as $C^\bullet_{\Gamma_{I\backslash |I|}}(M)$ is constructed. Here we assume $0< s_\alpha \leq r_\alpha/p$ for each $\alpha\in I$.	
\end{definition}

\begin{definition} \mbox{\bf{(After KPX, \cite[Definition 2.3.3]{KPX})}}
For any $B_{I'}$-$(\varphi_I,\Gamma_I)$-module $M$ over 
the first three groups of rings 
we define the corresponding complex $C^\bullet_{\varphi_I,\Gamma_I}(M)$ to be the corresponding totalization of $C^\bullet_{\varphi_I}C^\bullet_{\Gamma_I}(M)$.
For any $(\varphi_I,\Gamma_I)$-module $M$ over 
the second three groups of rings
we define the corresponding complex $C^\bullet_{\varphi_I,\Gamma_I}(M)$ to be the corresponding totalization of $C^\bullet_{\varphi_I}C^\bullet_{\Gamma_I}(M)$. Here we assume $0< s_\alpha \leq r_\alpha/p$ for each $\alpha\in I$.	
	
\end{definition}

\begin{definition} \mbox{\bf{(After KPX, \cite[Definition 2.3.3]{KPX})}}
For any $B_{I'}$-$(\varphi_I,\Gamma_I)$-module $M$ over 
the first three groups of rings
we define the corresponding complex $C^\bullet_{\psi_I,\Gamma_I}(M)$ to be the corresponding totalization of $C^\bullet_{\psi_I}C^\bullet_{\Gamma_I}(M)$.
For any $B_{I'}$-$(\varphi_I,\Gamma_I)$-module $M$ over 
the second three groups of rings
we define the corresponding complex $C^\bullet_{\psi_I,\Gamma_I}(M)$ to be the corresponding totalization of $C^\bullet_{\psi_I}C^\bullet_{\Gamma_I}(M)$. Here we assume $0< s_\alpha \leq r_\alpha/p$ for each $\alpha\in I$.	
	
\end{definition}

\subsection{Partial $B_{I'}$-Cohomology}

\noindent We now define the corresponding partial $B_{I'}$-cohomology  in the situation

\begin{definition}  \mbox{\bf{(After Nakamura, \cite[Appendix 5]{Nak1})}} 
For any $B_{I'}$-$(\varphi_I,\Gamma_I)$-module $M$ over 
\begin{align}
\varprojlim_{s_I}  B^+_{\mathrm{dR},I'}	\widehat{\otimes}\Pi_{[s_I,r_I],I,?,?',A},?=J,\widetilde{J},\breve{J},?'=I\backslash J,\widetilde{I\backslash J},\breve{I\backslash J},\\
\end{align}

with

\begin{align}
\varprojlim_{s_I} B^+_{\mathrm{dR},I'}	\widehat{\otimes}\Pi_{[s_I,r_I],I,?,?',A}[t_1^{-1},...,t_{I'}^{-1}],?=J,\widetilde{J},\breve{J},?'=I\backslash J,\widetilde{I\backslash J},\breve{I\backslash J},\\
\end{align}

with 

\begin{align}
\varprojlim_{s_I} B_{e,I'}	\widehat{\otimes}\Pi_{[s_I,r_I],I,?,?',A},?=J,\widetilde{J},\breve{J},?'=I\backslash J,\widetilde{I\backslash J},\breve{I\backslash J},\\
\end{align}

and

\begin{align}
B^+_{\mathrm{dR},I'}	\widehat{\otimes}\Pi_{[s_I,r_I],I,J,I\backslash J,A},\\	
B^+_{\mathrm{dR},I'}	\widehat{\otimes}\Pi_{[s_I,r_I],I,\breve{J},I\backslash J,A},\\	
B^+_{\mathrm{dR},I'}	\widehat{\otimes}\Pi_{[s_I,r_I],I,\widetilde{J},I\backslash J,A},\\
B^+_{\mathrm{dR},I'}	\widehat{\otimes}\Pi_{[s_I,r_I],I,J,\breve{I\backslash J},A},\\	
B^+_{\mathrm{dR},I'}	\widehat{\otimes}\Pi_{[s_I,r_I],I,\breve{J},\breve{I\backslash J},A},\\
B^+_{\mathrm{dR},I'}	\widehat{\otimes}\Pi_{[s_I,r_I],I,\widetilde{J},\breve{I\backslash J},A},\\
B^+_{\mathrm{dR},I'}	\widehat{\otimes}\Pi_{[s_I,r_I],I,J,\widetilde{I\backslash J},A},\\	
B^+_{\mathrm{dR},I'}	\widehat{\otimes}\Pi_{[s_I,r_I],I,\breve{J},\widetilde{I\backslash J},A},\\
B^+_{\mathrm{dR},I'}	\widehat{\otimes}\Pi_{[s_I,r_I],I,\widetilde{J},\widetilde{I\backslash J},A}
\end{align}

with

\begin{align}
B^+_{\mathrm{dR},I'}	\widehat{\otimes}\Pi_{[s_I,r_I],I,J,I\backslash J,A}[t_1^{-1},...,t_{I'}^{-1}],\\	
B^+_{\mathrm{dR},I'}	\widehat{\otimes}\Pi_{[s_I,r_I],I,\breve{J},I\backslash J,A}[t_1^{-1},...,t_{I'}^{-1}],\\	
B^+_{\mathrm{dR},I'}	\widehat{\otimes}\Pi_{[s_I,r_I],I,\widetilde{J},I\backslash J,A}[t_1^{-1},...,t_{I'}^{-1}],\\
B^+_{\mathrm{dR},I'}	\widehat{\otimes}\Pi_{[s_I,r_I],I,J,\breve{I\backslash J},A}[t_1^{-1},...,t_{I'}^{-1}],\\	
B^+_{\mathrm{dR},I'}	\widehat{\otimes}\Pi_{[s_I,r_I],I,\breve{J},\breve{I\backslash J},A}[t_1^{-1},...,t_{I'}^{-1}],\\
B^+_{\mathrm{dR},I'}	\widehat{\otimes}\Pi_{[s_I,r_I],I,\widetilde{J},\breve{I\backslash J},A}[t_1^{-1},...,t_{I'}^{-1}],\\
B^+_{\mathrm{dR},I'}	\widehat{\otimes}\Pi_{[s_I,r_I],I,J,\widetilde{I\backslash J},A}[t_1^{-1},...,t_{I'}^{-1}],\\	
B^+_{\mathrm{dR},I'}	\widehat{\otimes}\Pi_{[s_I,r_I],I,\breve{J},\widetilde{I\backslash J},A}[t_1^{-1},...,t_{I'}^{-1}],\\
B^+_{\mathrm{dR},I'}	\widehat{\otimes}\Pi_{[s_I,r_I],I,\widetilde{J},\widetilde{I\backslash J},A}[t_1^{-1},...,t_{I'}^{-1}],
\end{align}

with 

\begin{align}
B_{e,I'}	\widehat{\otimes}\Pi_{[s_I,r_I],I,J,I\backslash J,A},\\	
B_{e,I'}	\widehat{\otimes}\Pi_{[s_I,r_I],I,\breve{J},I\backslash J,A},\\	
B_{e,I'}	\widehat{\otimes}\Pi_{[s_I,r_I],I,\widetilde{J},I\backslash J,A},\\
B_{e,I'}	\widehat{\otimes}\Pi_{[s_I,r_I],I,J,\breve{I\backslash J},A},\\	
B_{e,I'}	\widehat{\otimes}\Pi_{[s_I,r_I],I,\breve{J},\breve{I\backslash J},A},\\
B_{e,I'}	\widehat{\otimes}\Pi_{[s_I,r_I],I,\widetilde{J},\breve{I\backslash J},A},\\
B_{e,I'}	\widehat{\otimes}\Pi_{[s_I,r_I],I,J,\widetilde{I\backslash J},A},\\	
B_{e,I'}	\widehat{\otimes}\Pi_{[s_I,r_I],I,\breve{J},\widetilde{I\backslash J},A},\\
B_{e,I'} \widehat{\otimes}\Pi_{[s_I,r_I],I,\widetilde{J},\widetilde{I\backslash J},A}
\end{align}
we define the corresponding $B_{I'}$-complex $C^\bullet_{B_{I'}}(M)$ to be 
\[
\xymatrix@C+0pc@R+0pc{
0 \ar[r] \ar[r] \ar[r] &{*}_1 \times {*}_2  \ar[r]^{?-?'}\ar[r]\ar[r] &{*}_3\ar[r] \ar[r] \ar[r] &0,
}
\]
where
\begin{align}
*_1:=C^\bullet(\mathrm{Gal}_{\mathbb{Q}_p,1}\times...\times \mathrm{Gal}_{\mathbb{Q}_p,I'}, M^+_\mathrm{dR}),\\
*_2:=C^\bullet(\mathrm{Gal}_{\mathbb{Q}_p,1}\times...\times \mathrm{Gal}_{\mathbb{Q}_p,I'}, M_e),\\
*_3:=C^\bullet(\mathrm{Gal}_{\mathbb{Q}_p,1}\times...\times \mathrm{Gal}_{\mathbb{Q}_p,I'}, M_\mathrm{dR}).	
\end{align}

\end{definition}

\begin{definition}
Now we consider any $B_{I'}$-$(\varphi_I,\Gamma_I)$-module $M$ over 

\begin{align}
\varprojlim_{s_I}  B^+_{\mathrm{dR},I'}	\widehat{\otimes}\Pi_{[s_I,r_I],I,?,?',A},?=J,\widetilde{J},\breve{J},?'=I\backslash J,\widetilde{I\backslash J},\breve{I\backslash J}\\
\end{align}

with

\begin{align}
\varprojlim_{s_I} B^+_{\mathrm{dR},I'}	\widehat{\otimes}\Pi_{[s_I,r_I],I,?,?',A}[t_1^{-1},...,t_{I'}^{-1}],?=J,\widetilde{J},\breve{J},?'=I\backslash J,\widetilde{I\backslash J},\breve{I\backslash J},\\
\end{align}

with 

\begin{align}
\varprojlim_{s_I} B_{e,I'}	\widehat{\otimes}\Pi_{[s_I,r_I],I,?,?',A}, ?=J,\widetilde{J},\breve{J},?'=I\backslash J,\widetilde{I\backslash J},\breve{I\backslash J},\\
\end{align}

and

\begin{align}
B^+_{\mathrm{dR},I'}	\widehat{\otimes}\Pi_{[s_I,r_I],I,?,?',A},?=J,\widetilde{J},\breve{J},?'=I\backslash J,\widetilde{I\backslash J},\breve{I\backslash J}
\end{align}

with

\begin{align}
B^+_{\mathrm{dR},I'}	\widehat{\otimes}\Pi_{[s_I,r_I],I,?,?',A}[t_1^{-1},...,t_{I'}^{-1}],?=J,\widetilde{J},\breve{J},?'=I\backslash J,\widetilde{I\backslash J},\breve{I\backslash J}
\end{align}

with 

\begin{align}
B_{e,I'}	\widehat{\otimes}\Pi_{[s_I,r_I],I,?,?',A},?=J,\widetilde{J},\breve{J},?'=I\backslash J,\widetilde{I\backslash J},\breve{I\backslash J}
\end{align}

Then we can define the corresponding $C^\bullet_{B_{I'},\varphi_I,\Gamma_I}$-cohomology complex by taking the corresponding totalization of the corresponding double complex $C^\bullet_{B_{I'}}C^\bullet_{\varphi_I,\Gamma_I}(M)$.
\end{definition}

\newpage

\section{The Results on the Cohomologies}

\subsection{Comparisons for $C^\bullet_{\varphi_I,\Gamma_I},C^\bullet_{\psi_I,\Gamma_I},C^\bullet_{\psi_I}$}

\noindent We now consider the following categories for $A$ a rigid affinoid over $\mathbb{Q}_p$:\\

\noindent A. The corresponding category of all the $(\varphi_I,\Gamma_I)$-modules over the corresponding rings carrying the corresponding cohomologies $C^\bullet_{(\varphi_I,\Gamma_I)},C^\bullet_{(\psi_I,\Gamma_I)},C^\bullet_{\psi_I}$ (for sufficiently small $r_{I,0}$): \\

\begin{align}
\Pi_{\mathrm{an},r_{I,0},I,\breve{J},I\backslash J,A}(\pi_{K_I}):=\varprojlim_{s_I}\Pi_{[s_I,r_{I,0}],I,\breve{J},I\backslash J,A}(\pi_{K_I}),\\	
\Pi_{\mathrm{an},r_{I,0},I,\widetilde{J},I\backslash J,A}(\pi_{K_I}):=\varprojlim_{s_I} \Pi_{[s_I,r_{I,0}],I,\widetilde{J},I\backslash J,A}(\pi_{K_I}),\\
\Pi_{\mathrm{an},r_{I,0},I,J,\breve{I\backslash J},A}(\pi_{K_I}):=\varprojlim_{s_I}\Pi_{[s_I,r_{I,0}],I,J,\breve{I\backslash J},A}(\pi_{K_I}),\\	
\Pi_{\mathrm{an},r_{I,0},I,\breve{J},\breve{I\backslash J},A}(\pi_{K_I}):=\varprojlim_{s_I} \Pi_{[s_I,r_{I,0}],I,\breve{J},\breve{I\backslash J},A}(\pi_{K_I}),\\	
\Pi_{\mathrm{an},r_{I,0},I,\widetilde{J},\breve{I\backslash J},A}(\pi_{K_I}):=\varprojlim_{s_I} \Pi_{[s_I,r_{I,0}],I,\widetilde{J},\breve{I\backslash J},A}(\pi_{K_I}),\\
\Pi_{\mathrm{an},r_{I,0},I,J,\widetilde{I\backslash J},A}(\pi_{K_I}):=\varprojlim_{s_I} \Pi_{[s_I,r_{I,0}],I,J,\widetilde{I\backslash J},A}(\pi_{K_I}),\\	
\Pi_{\mathrm{an},r_{I,0},I,\breve{J},\widetilde{I\backslash J},A}(\pi_{K_I}):=\varprojlim_{s_I} \Pi_{[s_I,r_{I,0}],I,\breve{J},\widetilde{I\backslash J},A}(\pi_{K_I}),\\	
\Pi_{\mathrm{an},r_{I,0},I,\widetilde{J},\widetilde{I\backslash J},A}(\pi_{K_I}):=\varprojlim_{s_I} \Pi_{[s_I,r_{I,0}],I,\widetilde{J},\widetilde{I\backslash J},A}(\pi_{K_I}).	
\end{align}

\noindent B. The corresponding category of all the $(\varphi_I,\Gamma_I)$-modules over the corresponding rings carrying the corresponding cohomologies $C^\bullet_{(\varphi_I,\Gamma_I)},C^\bullet_{(\psi_I,\Gamma_I)},C^\bullet_{\psi_I}$ (where $0<s_\alpha\leq r_\alpha/p\leq  r_{\alpha,0}$ for each $\alpha\in I$): \\

\begin{align}
\Pi_{[s_I,r_I],I,\breve{J},I\backslash J,A}(\pi_{K_I}),\\	
\Pi_{[s_I,r_I],I,\widetilde{J},I\backslash J,A}(\pi_{K_I}),\\
\Pi_{[s_I,r_I],I,J,\breve{I\backslash J},A}(\pi_{K_I}),\\	
\Pi_{[s_I,r_I],I,\breve{J},\breve{I\backslash J},A}(\pi_{K_I}),\\	
\Pi_{[s_I,r_I],I,\widetilde{J},\breve{I\backslash J},A}(\pi_{K_I}),\\
\Pi_{[s_I,r_I],I,J,\widetilde{I\backslash J},A}(\pi_{K_I}),\\	
\Pi_{[s_I,r_I],I,\breve{J},\widetilde{I\backslash J},A}(\pi_{K_I}),\\	
\Pi_{[s_I,r_I],I,\widetilde{J},\widetilde{I\backslash J},A}(\pi_{K_I}).	
\end{align}

\begin{theorem}
Let $I$ be a set of two elements. Let $M$ be some object over $\Pi_{\mathrm{an},r_{I,0},I,*,*,A}(\pi_{K_I})$ in the corresponding category $A$ and let $M_{[s_I,r_I]}$ be the corrresponding (under the horizontal equivalence of the categories for the rings of the same type) object over $\Pi_{[s_I,r_{I}],I,*,*,A}(\pi_{K_I})$ in the category $B$. Then we have the following quasi-isomorphisms:
\begin{align}
C^\bullet_{\varphi_I,\Gamma_I}(M)\overset{\sim}{\rightarrow}C^\bullet_{\varphi_I,\Gamma_I}(M_{[s_I,r_I]}), \\
C^\bullet_{\psi_I,\Gamma_I}(M)\overset{\sim}{\rightarrow}C^\bullet_{\psi_I,\Gamma_I}(M_{[s_I,r_I]}), \\
C^\bullet_{\psi_I}(M)\overset{\sim}{\rightarrow}C^\bullet_{\psi_I}(M_{[s_I,r_I]}), \\	
\end{align}
in the bounded derived category of $A$-modules $D^\flat(A)$.	
\end{theorem}

\begin{proof}
The proof relies on the following intermediate $(\varphi_I,\Gamma_I)$-modules. Let $I=\{1,2\}$, we now defined the following rings:
\begin{align}
\Pi_{\mathrm{an},r_{1,0},[s_2,r_2],I,\breve{J},I\backslash J,A}(\pi_{K_I}):=\varprojlim_{s_1}\Pi_{[s_1,r_{1,0}]\times[s_2,r_{2}],I,\breve{J},I\backslash J,A}(\pi_{K_I}),\\	
\Pi_{\mathrm{an},r_{1,0},[s_2,r_2],I,\widetilde{J},I\backslash J,A}(\pi_{K_I}):=\varprojlim_{s_1} \Pi_{[s_1,r_{1,0}]\times[s_2,r_{2}],I,\widetilde{J},I\backslash J,A}(\pi_{K_I}),\\
\Pi_{\mathrm{an},r_{1,0},[s_2,r_2],I,J,\breve{I\backslash J},A}(\pi_{K_I}):=\varprojlim_{s_1}\Pi_{[s_1,r_{1,0}]\times[s_2,r_{2}],I,J,\breve{I\backslash J},A}(\pi_{K_I}),\\	
\Pi_{\mathrm{an},r_{1,0},[s_2,r_2],I,\breve{J},\breve{I\backslash J},A}(\pi_{K_I}):=\varprojlim_{s_1} \Pi_{[s_1,r_{1,0}]\times[s_2,r_{2}],I,\breve{J},\breve{I\backslash J},A}(\pi_{K_I}),\\	
\Pi_{\mathrm{an},r_{1,0},[s_2,r_2],I,\widetilde{J},\breve{I\backslash J},A}(\pi_{K_I}):=\varprojlim_{s_1} \Pi_{[s_1,r_{1,0}]\times[s_2,r_{2}],I,\widetilde{J},\breve{I\backslash J},A}(\pi_{K_I}),\\
\Pi_{\mathrm{an},r_{1,0},[s_2,r_2],I,J,\widetilde{I\backslash J},A}(\pi_{K_I}):=\varprojlim_{s_1} \Pi_{[s_1,r_{1,0}]\times[s_2,r_{2}],I,J,\widetilde{I\backslash J},A}(\pi_{K_I}),\\	
\Pi_{\mathrm{an},r_{1,0},[s_2,r_2],I,\breve{J},\widetilde{I\backslash J},A}(\pi_{K_I}):=\varprojlim_{s_1} \Pi_{[s_1,r_{1,0}]\times[s_2,r_{2}],I,\breve{J},\widetilde{I\backslash J},A}(\pi_{K_I}),\\	
\Pi_{\mathrm{an},r_{1,0},[s_2,r_2],I,\widetilde{J},\widetilde{I\backslash J},A}(\pi_{K_I}):=\varprojlim_{s_1} \Pi_{[s_1,r_{1,0}]\times[s_2,r_{2}],I,\widetilde{J},\widetilde{I\backslash J},A}(\pi_{K_I}).	
\end{align}
This category serves as a corresponding intermediate category which factors through the corresponding original equivalence on the categories involved. From $A$ to this category we have the equivalence on the cohomology groups by considering the $(\varphi_2,\Gamma_2)$-module structure, then by regarding the corresponding $(\varphi_2,\Gamma_2)$-cohomology groups as Yoneda extension groups we have the isomorphism on $(\varphi_2,\Gamma_2)$-cohomology groups which further shows the equivalence on the full cohomology groups. Similarly one compares this category with the category $B$ to finish.

\end{proof}

\newpage

\subsection*{Acknowledgements}

This is the natural continuation of our previous work \cite{T1} after \cite{CKZ} and \cite{PZ}. One could easily feel that we are actually inspired now by many programs in $p$-adic analysis. The corresponding Iwasawa consideration is still a very inspiring point, however $p$-adic Langlands program also introduces some motivation. We would like to thank Professor Kedlaya for many helpful discussion along my study in these directions presented in this paper such as the higher dimensional period ring $B_{\mathrm{dR},I}$.

\newpage

\bibliographystyle{ams}

\end{document}